\documentclass[12pt]{amsart}
\usepackage{enumerate,amssymb,version,aliascnt,geometry,stmaryrd,mathrsfs,chngcntr,graphicx}
\usepackage[all]{xy}
\usepackage[mathscr]{euscript}
\usepackage{hyperref}

\hypersetup{
pdfauthor={Olivier Haution},
pdftitle={Motivic Pontryagin classes and hyperbolic orientations},
pdfkeywords={},
hidelinks
}

\usepackage[T1]{fontenc}
\usepackage{textcomp}

\let\origsetminus\setminus
\let\originfty\infty
\let\origpartial\partial
\let\origin\in
\let\origprod\prod
\let\origsum\sum
\let\origsubset\subset
\let\origsimeq\simeq
\let\origast\ast
\let\origsum\sum
\usepackage{mathabx}
\let\setminus\origsetminus
\let\infty\originfty
\let\partial\origpartial
\let\in\origin
\let\prod\origprod
\let\subset\origsubset
\let\simeq\origsimeq
\let\ast\origast
\let\sum\origsum

\newcommand{\Au}{\mathbb{A}^1}
\DeclareMathOperator{\Th}{Th}
\DeclareMathOperator{\SH}{SH}
\DeclareMathOperator{\Hop}{H_\bullet}
\DeclareMathOperator{\Hot}{H}
\DeclareMathOperator{\Spc}{Spc}
\DeclareMathOperator{\Spcp}{Spc_\bullet}
\DeclareMathOperator{\Spt}{Spt}
\DeclareMathOperator{\Ld}{\mathbf{L}}
\DeclareMathOperator{\MH}{MH}
\DeclareMathOperator{\T}{Th}
\DeclareMathOperator{\Sy}{\mathfrak{S}}
\DeclareMathOperator{\SL}{SL}

\DeclareMathOperator{\id}{id}
\DeclareMathOperator{\Hom}{Hom}
\DeclareMathOperator{\BGL}{BGL}
\DeclareMathOperator{\BSL}{BSL}
\DeclareMathOperator{\GL}{GL}
\DeclareMathOperator{\MGL}{MGL}
\DeclareMathOperator{\Sp}{Sp}
\DeclareMathOperator{\MSp}{MSp}
\DeclareMathOperator{\Oo}{O}
\DeclareMathOperator{\Gr}{Gr}
\DeclareMathOperator{\rank}{rank}
\DeclareMathOperator{\Sm}{Sm}
\DeclareMathOperator{\Ab}{\mathbb{A}}
\DeclareMathOperator{\colim}{colim}
\DeclareMathOperator{\sw}{sw}
\DeclareMathOperator{\thom}{th}
\DeclareMathOperator{\can}{can}

\newcommand{\vo}{\oplus}
\newcommand{\Uni}{{\mathcal{U}}}
\newcommand{\Quo}{{\mathcal{Q}}}

\newcommand{\Un}{\mathbf{1}}
\newcommand{\Aa}{A^{*,*}}

\newcommand{\Su}{\Sigma^{\infty}}
\newcommand{\Sup}{\Sigma^{\infty}_+}

\newcommand{\hypo}{\mathfrak{o}}
\newcommand{\hypt}{\mathfrak{t}}
\newcommand{\phyp}{\mathfrak{p}}
\newcommand{\Oc}{\mathcal{O}}
\newcommand{\Ec}{\mathcal{E}}
\newcommand{\Zz}{\mathbb{Z}}
\newcommand{\Pp}{\mathbb{P}}
\newcommand{\Nn}{\mathbb{N}}

\newcommand{\Gm}{\mathbb{G}_m}
\newcommand{\X}{\mathcal{X}}

\newcommand{\pf}[1]{\overline{#1}}

\newtheorem*{theorem*}{Theorem}
\newtheorem*{proposition*}{Proposition}
\newtheorem{thm}{Theorem}
\newtheorem{prop}[thm]{Proposition}
\newtheorem{cor}[thm]{Corollary}

\swapnumbers

\newtheorem{theorem}{Theorem}[section]
\newaliascnt{proposition}{theorem}
\newtheorem{proposition}[proposition]{Proposition}
\newaliascnt{lemma}{theorem}
\newtheorem{lemma}[lemma]{Lemma}
\newaliascnt{corollary}{theorem}
\newtheorem{corollary}[corollary]{Corollary}

\theoremstyle{definition}
\newaliascnt{remark}{theorem}
\newtheorem{remark}[theorem]{Remark}
\newaliascnt{example}{theorem}
\newtheorem{example}[example]{Example}
\newaliascnt{definition}{theorem}
\newtheorem{definition}[definition]{Definition}
\newaliascnt{notation}{theorem}

\newtheoremstyle{par}
  {}
  {}
  {}
  {}
  {}
  {.}
  { }
  {}%
\theoremstyle{par}
\newtheorem{para}[theorem]{}

\counterwithin{equation}{theorem}
\numberwithin{equation}{theorem}

\newcommand{\rref}[1]{(\ref{#1})}
\newcommand{\dref}[2]{(\ref{#1}.\ref{#2})}

\begin{document}
\begin{abstract}
We introduce the notion of hyperbolic orientation of a motivic ring spectrum, which generalises the various existing notions of orientation (by the groups $\GL$, $\SL^c$, $\SL$, $\Sp$). We show that hyperbolic orientations of $\eta$-periodic ring spectra correspond to theories of Pontryagin classes, much in the same way that $\GL$-orientations of arbitrary ring spectra correspond to theories of Chern classes. We prove that $\eta$-periodic hyperbolically oriented cohomology theories do not admit further characteristic classes for vector bundles, by computing the cohomology of the \'etale classifying space $\BGL_n$. Finally we construct the universal hyperbolically oriented $\eta$-periodic commutative motivic ring spectrum, an analog of Voevodsky's cobordism spectrum $\MGL$.
\end{abstract}

\author{Olivier Haution}
\title{Motivic Pontryagin classes and hyperbolic orientations}
\email{olivier.haution at gmail.com}
\address{Mathematisches Institut, Ludwig-Maximilians-Universit\"at M\"unchen, Theresienstr.\ 39, D-80333 M\"unchen, Germany}
\address{Dipartimento di Matematica e Applicazioni, Università degli Studi di Milano-Bicocca, via Roberto Cozzi 55, 20125 Milano, Italy}
\thanks{This work was supported by the DFG research grant HA 7702/5-1 and Heisenberg grant HA 7702/4-1.}

\subjclass[2010]{}

\date{\today}

\maketitle

\numberwithin{theorem}{section}
\numberwithin{lemma}{section}
\numberwithin{proposition}{section}
\numberwithin{corollary}{section}
\numberwithin{example}{section}
\numberwithin{notation}{section}
\numberwithin{definition}{section}
\numberwithin{remark}{section}

\setcounter{tocdepth}{2}
\tableofcontents

\section*{Introduction}
Let us fix a base scheme $S$ (for instance the spectrum of a field), and consider a cohomology theory $\Aa$ on smooth $S$-schemes, represented by a motivic ring spectrum $A$ in Voevodsky's stable $\Ab^1$-homotopy category $\SH(S)$. When $E$ is a vector bundle of rank $r$ over a smooth $S$-scheme $X$, its Chern classes are elements
\[
c_i(E) \in A^{2i,i}(X) \quad \text{ for $i=0,\dots,r$}.
\]
This classes are commonly used to:
\begin{itemize}
\item[---] detect the presence of nowhere vanishing sections of $E$,
\item[---] distinguish non-isomorphic bundles,
\item[---] exhibit elements in $\Aa(X)$.
\end{itemize}
The existence of such classes having good properties is subject to certain conditions on the ring spectrum $A$, and moreover these classes are not determined by $A$ alone. In fact, Panin \cite{Panin-Oriented_I,Panin-Oriented_II} proved that the theories of Chern classes with values in the theory $\Aa$ are in correspondence with the so-called $\GL$-orientations of the motivic spectrum $A$. Panin and Walter \cite{PW-QGrass} later provided a parallel correspondence between Borel classes and $\Sp$-orientations.\\

The Pontryagin classes of the vector bundle $E \to X$ are expected to be elements
\[
p_i(E) \in A^{8i,4i}(X) \quad \text{ for $i=0,\dots,\lfloor r/2 \rfloor$}.
\]
As was the case with Chern classes, it is not expected that well-behaved Pontryagin classes exist for an arbitrary motivic ring spectrum $A$, nor that they should be determined by $A$ alone. Such classes have been defined using the Borel classes when $A$ is $\Sp$-oriented \cite[Definition~7]{Ananyevskiy-SL_PB}, but have been proved to have good properties only when $A$ is $\SL$-oriented and $\eta$-periodic (see \cite[\S7]{Ana-SL}). This provides an important example of Pontryagin classes, but we are not aware of any systematic study of the theories of such classes.

In this paper, we propose a notion of \emph{hyperbolic orientation} of a motivic ring spectrum, which fits into the following table:
\bgroup
\def\arraystretch{1.5}
\setlength\tabcolsep{15pt}
\medskip
\begin{center}
\begin{tabular}{ |c|c| } 
\hline
Chern classes & $\GL$-orientation\\
\hline
Borel classes & $\Sp$-orientation \\
\hline
Pontryagin classes & hyperbolic orientation \\
\hline
\end{tabular}
\end{center}
\medskip
An important caveat is that we limit ourselves to the consideration of $\eta$-periodic cohomology theories. We will discuss at the end of this introduction the reasons why in some sense this restriction is necessary. It turns out that Chern classes do not exist in the $\eta$-periodic context, and that Pontryagin classes are the natural characteristic classes of vector bundles (without extra structure) with values in hyperbolically oriented cohomology theories (this statement is made precise in Theorem~\ref{intro:BGL} below).\\

The basic idea motivating the notion of a hyperbolic orientation is that the data of Thom classes for all symplectic bundles (i.e.\ a symplectic orientation) is substantially more than what is needed in order to define Pontryagin classes: in fact it turns out that it suffices to have these data for hyperbolic symplectic bundles. Such bundles are determined by the vector bundle (of half rank) whose Pontryagin classes we want to take.

This is perhaps not surprising since after all, the existing construction of Pontryagin classes in $\Sp$-oriented theories only involves the Borel classes of hyperbolic symplectic bundles. But those Borel classes are constructed using the Thom classes of the universal rank two symplectic bundles on the corresponding symplectic grassmannians, and those universal bundles are far from being hyperbolic. So it was not a priori completely clear that this program would succeed.

A hyperbolic orientation will thus be the data of a Thom class for each vector bundle, but an important difference with all existing notions of orientation is that this class does not live in the cohomology of the Thom space of that vector bundle. Rather, the hyperbolic Thom class of a vector bundle $E$ lives in the cohomology of the Thom space of the hyperbolic bundle $E \oplus E^\vee$. That space happens to be isomorphic to the Thom space of $E \oplus E$, and in this paper it will be internally more logical to consider the latter instead. The definition of a hyperbolic orientation is given in \rref{def:hyp_orientation}.

Symplectically oriented theories (and thus also $\GL$-, $\SL^c$-, $\SL$-oriented ones) are naturally hyperbolically oriented, but for instance orthogonally oriented ones also provide examples of hyperbolically oriented theories (see \rref{ex:Sp_hyp} for details).\\

Grothendieck provided in \cite{Gro-58} a seminal construction of Chern classes in Chow theory. He singled out the so-called projective bundle theorem as the key property, which permits to extend the definition of the first Chern classes of line bundles (imposed by the natural relation between the Picard group and the Chow group) to higher Chern classes of arbitrary vector bundles. This method has been revisited over the years, and is central in Panin's study of $\GL$-orientations mentioned above. 

In this paper, we use the same strategy for Pontryagin classes. The main difference is that line bundles should be replaced with rank two bundles. Thus if $E$ is vector bundle, the projective bundle $\Pp(E)$ with its tautological bundle $\Oc(-1)$ is replaced by the Grassmann bundle $\Gr(2,E)$ of rank $2$ subbundles with its universal rank two subbundle $\Uni_2$. The first Pontryagin class $p_1(\Uni_2) \in A^{8,4}(\Gr(2,E))$ is constructed directly from the hyperbolic orientation, and we prove the following analog of the projective bundle theorem (see \rref{cor:splitting_twisted}):
\begin{thm}
\label{intro:PBT}
Let $A \in \SH(S)$ be an $\eta$-periodic hyperbolically oriented ring spectrum. For any vector bundle $E \to X$ of rank $2d$ or $2d+1$, the $\Aa(X)$-module $\Aa(\Gr(2,E))$ is freely generated by the elements $1,p_1(\Uni_2),\dots, p_1(\Uni_2)^{d-1}$.
\end{thm}
We may then use Grothendieck's method to define the higher Pontryagin classes $p_i(E) \in A^{4i,2i}(X)$ for $i=1,\dots,d$. We prove that these classes satisfy the expected properties, the most notable being perhaps the Whitney sum formula (see \rref{prop:Pontryagin_sum}):
\begin{prop}
If $E,F \to X$ are vector bundles, their Pontryagin classes satisfy
\[
p_i(E \oplus F) = \sum_j p_{i-j}(E) p_j(F).
\]
\end{prop}

We also derive the following splitting principle (see \rref{th:splitting}):
\begin{thm}
\label{intro:splitting}
Let $A \in \SH(S)$ be an $\eta$-periodic hyperbolically oriented ring spectrum. Let $E \to X$ be a vector bundle. Then there exists a morphism $f \colon Y \to X$ such that:
\begin{enumerate}[(i)]
\item $f^* \colon \Aa(X) \to \Aa(Y)$ is a split injection,

\item the vector bundle $f^*E$ splits as a direct sum of rank two vector bundles having trivial determinants, and possibly a trivial line bundle.
\end{enumerate}
\end{thm}
Both this splitting principle and the analog of the projective bundle theorem (Theorem~\ref{intro:PBT}) are reminiscent of Ananyevskiy's results in the $\SL$-oriented setting \cite{Ananyevskiy-SL_PB}. The difference is that our results apply to vector bundles without additional structure as opposed to ones with trivialised determinant, and that the assumptions on the ring spectrum $A$ are weaker. As an illustration of this added flexibility, we obtain the following statement, a priori unrelated to hyperbolic orientations:
\begin{cor}
Let $A \in \SH(S)$ be an $\eta$-periodic commutative ring spectrum. Then each $\Sp$-orientation of $A$ is induced by at most one normalised $\SL$-orientation.
\end{cor}

One may notice that the above construction of Pontryagin classes in fact only uses the hyperbolic Thom classes of rank two bundles. We prove in \S\ref{sect:o_wo} that such data, that we call a weak hyperbolic orientation (see \rref{def:hyp_thom} for a precise definition), is in fact equivalent to the data of a hyperbolic orientation. Part of this statement (the uniqueness) is of course a consequence of the splitting principle, but there is more to it. Indeed, the axioms of a hyperbolic orientation include a multiplicativity property (the Thom class of a direct sum is the product of the Thom classes of the summands), which has no counterpart for weak hyperbolic orientations (as they concern only rank two bundles). The verification of this multiplicative property is thus not a simple formality, and in fact relies quite a bit on the theory developed earlier in the paper.

We also prove in \S\ref{sect:Pontryagin_structures} the comparatively easier fact that it is equivalent to specify the Thom classes of the rank two bundles, or their first Pontryagin classes. In particular, a hyperbolic orientation is precisely what is needed in order to obtain a theory of Pontryagin classes.\\

Next, we compute in \rref{th:A_BGL} the cohomology of the \'etale classifying space $\BGL_n$, generalising a computation of Levine \cite[Theorem~4.1]{Levine-motivic_Euler} (by including for instance the case when $A$ is $\Sp$-oriented)
\begin{thm}
\label{intro:BGL}
Let $A \in \SH(S)$ be an $\eta$-periodic hyperbolically oriented ring spectrum. Let $r \in \Nn$, and $n \in \{2r,2r+1\}$. Then
\[
\Aa(\BGL_n) = \Aa(S)[[p_1,\dots,p_r]]_h.
\]
\end{thm}
Here the index $h$ refers to the homogeneous power series ring. This theorem asserts that there are no universal relations between the Pontryagin classes, and that there are no further invariants of vector bundles with values in $\eta$-periodic hyperbolically oriented cohomology theories. Unsurprisingly, the theorem is deduced from the computation of the cohomology of higher Grassmannians $\Gr(n,1^{\oplus s})$ for appropriate values of $n,s$.\\

In the last section, we first show that a weak hyperbolic orientation of an $\eta$-periodic cohomology can be defined by specifying the Thom class of the universal rank two vector bundle, in the form of an element in the cohomology of a certain motivic space $\MH_2$, subject to a normalisation condition.

We then construct a motivic ring spectrum $\MH$ similar to Voevodsky's cobordism spectrum $\MGL$, and show that it becomes the universal $\eta$-periodic hyperbolically oriented commutative ring spectrum after inverting $\eta$. The considerations of that section are parallel to those of Panin--Pimenov--R\"ondigs \cite{PPR-MGL} on $\MGL$, and of Panin--Walter \cite{PW-MSL-Msp} on $\MSp$.\\

To summarise, we describe in this paper five equivalent structures on an $\eta$-periodic commutative ring spectrum $A\in \SH(S)$:
\begin{enumerate}[(i)]
\item hyperbolic orientations \rref{def:hyp_orientation},

\item weak hyperbolic orientations \rref{def:hyp_thom},

\item Pontryagin structures \rref{def:Pontryagin_str},

\item ``normalised'' elements of $A^{8,4}(\MH_2)$ \rref{p:def_MHn},
\item morphisms of ring spectra $\MH \to A$ \rref{def:MH}.
\end{enumerate}
These equivalences are established in \rref{prop:Pontryagin_str-wo}, \rref{th:Thom_orientation}, \rref{prop:MH_2_weak}, \rref{th:MH_univ}.\\

Let us now comment on the standing assumption of $\eta$-periodicity. In order to construct Pontryagin classes using Grothendieck's method, we need the analog of the projective bundle theorem (Theorem~\ref{intro:PBT}). When the vector bundle $E$ is trivial of rank three over the base $S$, the Grassmannian $\Gr(2,E)$ is identified with the projective space $\Pp^2$, and the conclusion of theorem is that the pullback $\Aa(S) \to \Aa(\Pp^2)$ is an isomorphism. This condition is equivalent to the requirement that the pullback along the Hopf map $\eta \colon \Ab^2\smallsetminus \{0\} \to \Pp^1$ induces an isomorphism in $\Aa$, and so making this assumption is necessary if we are to obtain Pontryagin classes satisfying the analog of the projective bundle theorem. Looking at the situation in topology reinforces this idea: there it is known that the Pontryagin classes satisfy the Whitney sum formula only up to $2$-torsion, and the Hopf element $\eta \in \pi^{-1,-1}(\mathbb{R})$ maps to $2$ upon real realisation. It would of course still be interesting to systematically study Pontryagin classes beyond the $\eta$-periodic context, but one probably should not expect such a simple picture to emerge.\\

Apart from the construction of the motivic stable homotopy category (as well as certain abstract considerations on symmetric spectra in \S\ref{sect:MH}), the paper is fairly self-contained; the only external results used in a essential way are contained in the paper \cite{eta} (those are \rref{prop:splitting_odd}, \rref{prop:PE} and \rref{lemm:PE_even_Q}).

\noindent \textbf{Acknowledgments}.
I am grateful to the referee for the suggestions, and in particular for drawing my attention to a rather serious sign oversight in a previous version of this paper.

\section{Notation and basic facts}

\begin{para}
Throughout the paper, we work over a noetherian base scheme $S$ of finite dimension. The category of smooth separated $S$-schemes of finite type will be denoted by $\Sm_S$. All schemes will be implicitly assumed to belong to $\Sm_S$, and the notation $\Ab^n,\Pp^n,\Gm$ will refer to the corresponding $S$-schemes. We will denote by $1$ the trivial line bundle over a given scheme in $\Sm_S$.
\end{para}

\begin{para}
We will work with the $\Ab^1$-homotopy theory introduced by Morel--Voevodsky \cite{MV-A1}. We will denote by $\Spc(S)$ the category of motivic spaces (i.e.\ simplicial presheaves on $\Sm_S$), by $\Spcp(S)$ its pointed version, and by $\Spt(S)$ the category of $T$-spectra, where $T= \mathbb{A}^1/\Gm$. We endow these with the motivic equivalences, resp.\ stable motivic equivalences, and denote by $\Hot(S), \Hop(S), \SH(S)$ the respective homotopy categories. We refer to e.g.\ \cite[Appendix~A]{PMR} for more details.

We have an infinite suspension functor $\Su\colon \Spcp(S) \to \Spt(S)$. Composing with the functor $\Spc(S) \to \Spcp(S)$ adding an external base-point, we obtain a functor $\Sup \colon \Spc(S) \to \Spt(S)$.

The spheres are denoted as usual by $S^{p,q} \in \Spcp(S)$ for $p,q\in \Nn$ with $p \geq q$, (where $T\simeq S^{2,1}$). The motivic sphere spectrum will be denoted by $\Un_S =\Sup S \in \Spt(S)$. When $A$ is a motivic spectrum, we denote its $(p,q)$-th suspension by $\Sigma^{p,q}A = A \wedge \Su S^{p,q}$. This yields functors $\Sigma^{p,q} \colon \SH(S) \to \SH(S)$ for $p,q\in \Zz$.
\end{para}

\begin{para}
\label{p:f_sharp}
Let $X \in \Sm_S$, with structural morphism $f\colon X \to S$. Then viewing a smooth $X$-scheme as an $S$-scheme induces a functor $f_\sharp \colon \Hop(X) \to \Hop(S)$ (see e.g.\ \cite[p.104]{MV-A1} where it is denoted by $\Ld \! f_\sharp$). We will also denote by $f_\sharp \colon \SH(X) \to \SH(S)$ the induced functor.
\end{para}

\begin{para}
When $E\to X$ is a vector bundle with $X \in \Sm_S$, we denote by $E^{\circ}=E \smallsetminus X$ the complement of its zero-section. The Thom space of $E$ is the pointed motivic space $\Th_X(E) = E/E^\circ$. When $f \colon Y \to X$ is a morphism in $\Sm_S$, we will usually write $\Th_Y(E)$ instead of $\Th_Y(f^*E)$. Note that a vector bundle inclusion $E \subset F$ over $X$ induces a map $\Th_X(E) \to \Th_X(F)$.
\end{para}

\begin{para}
\label{p:sigma_V}
Let $V \to S$ be a vector bundle. We denote by $\Sigma^V \colon \Hop(S) \to \Hop(S)$ the derived functor induced by $A \mapsto A \wedge \Th_S(V)$. We will also denote by $\Sigma^V \colon \SH(S) \to \SH(S)$ the induced functor.
\end{para}

\begin{para}
In an attempt to lighten the notation, we will sometimes remove the symbols $\Su$ or $\Sup$ for the notation, when the context makes the meaning sufficiently clear. In particular, we will systematically write $\Th_X(E) \in \Spt(S)$ instead of $\Su \Th_X(E)$, and when $f$ is a morphism in $\Sm_S$ or $\Spc(S)$ (resp.\ in $\Spcp(S)$), we will often write $f$ instead of $\Sup f$ (resp.\ $\Su f$).
\end{para}

\begin{para}
\label{p:Thom_ses}
Assume that $0 \to E' \to E \to E'' \to 0$ is an exact sequence of vector bundles over $X \in \Sm_S$. Then by \cite[\S4.1]{Riou} or \cite[Remark~3.2.7]{Hoyois-equivariant} (see e.g.\ \cite[Proof of Lemma 4]{Ana-Pushforwards} for the explicit homotopies) we have a canonical isomorphism in $\Hop(S)$
\[
\Th_X(E) \simeq \Th_X(E' \oplus E''),
\]
which is induced by any splitting of the above exact sequence, if such exists.
\end{para}

\begin{para}
\label{p:Jouanolou}
(See \cite[p.243]{Riou}.) We will use the following alternative to Jouanolou's trick. Assume that $0 \to E' \to E \to E'' \to 0$ is an exact sequence of vector bundles over $X \in \Sm_S$. Consider the scheme $Y$ parametrising sections of $E \to E''$ (a closed subscheme of the vector bundle $\Hom(E'',E)$ over $X$). The morphism $f\colon Y \to X$ is a torsor under the vector bundle $\Hom(E'',E')$, and in particular $Y \in \Sm_S$. In addition there exists an isomorphism $f^*E \simeq f^*E' \oplus f^*E''$ of vector bundles over $Y$. 
\end{para}

\begin{para}
Let $A \in \Spt(S)$ be a motivic spectrum. For $C \in \Spt(S)$, we write
\[
A^{p,q}(C) = \Hom_{\SH(S)}(C,\Sigma^{p,q}A),
\]
and $\Aa(C) = \bigoplus_{p,q\in\Zz} A^{p,q}(C)$. When $\X$ is a pointed motivic space, resp.\ a motivic space, we will write $\Aa(\X)$ instead of $\Aa(\Su \X)$, resp.\ $\Aa(\Sup \X)$. If $E \to X$ is a vector bundle of constant rank $r$ with $X \in \Sm_S$, we write
\[
A^{p,q}(X;E) = A^{p+2r,q+r}(\Th_X(E)),
\]
and extend this notation to arbitrary vector bundles in an obvious way. Note that for any $d \in \Nn$ the auto-equivalence $\Sigma^{2d,d}$ of $\SH(S)$ induces isomorphisms
\begin{equation}
\label{eq:twist_shift}
\Sigma^{2d,d} \colon A^{p,q}(X;E) \xrightarrow{\sim} A^{p,q}(X;E\oplus 1^{\oplus d}),
\end{equation}

A morphism $f\colon B \to B'$ in $\SH(S)$ induces a pullback $f^* \colon \Aa(B') \to \Aa(B)$. A morphism $\psi \colon A \to A'$ in $\SH(S)$ induces a pushforward $\psi_* \colon \Aa(B) \to A'^{*,*}(B)$ for any $B \in \Spt(S)$.
\end{para}

\begin{para}
\label{p:eta}
We denote by $\eta \colon \Ab^2 \smallsetminus \{0\} \to \Pp^1$ in $\Spcp(S)$ the map $(x,y) \mapsto [x:y]$, where $\Ab^2\smallsetminus \{0\}$ is pointed by $(1,1)$ and $\Pp^1$ by $[1:1]$. 
\end{para}

\begin{para}
\label{p:eta_periodic}
A motivic spectrum $A \in \Spt(S)$ is called \emph{$\eta$-periodic} if the map
\[
A \wedge \Su (\Ab^2 \smallsetminus \{0\}) \xrightarrow{\id \wedge \Su \eta} A \wedge \Su \Pp^1
\]
is an isomorphism in $\SH(S)$. The full subcategory of such objects will be denoted $\Spt(S)[\eta^{-1}]$, and may be viewed as a left Bousfield localisation of $\Spt(S)$, as explained in \cite[\S6]{Bachmann-real}. The homotopy category of $\Spt(S)[\eta^{-1}]$ will be denoted by $\SH(S)[\eta^{-1}]$, and we will usually omit the mention of the inclusion and localisation functors. When $A \in \Spt(S)$ is $\eta$-periodic and $\mathcal{X} \in \Spcp(S)$, we thus have a natural identification
\[
A^{p,q}(\mathcal{X}) = \Hom_{\SH(S)[\eta^{-1}]}(\Su \mathcal{X},\Sigma^{p,q}A).
\]
The functor of \rref{p:f_sharp} descends to a functor $f_\sharp \colon \SH(X)[\eta^{-1}] \to \SH(S)[\eta^{-1}]$, and the functor of \rref{p:sigma_V} to a functor $\Sigma^V \colon \SH(S)[\eta^{-1}] \to \SH(S)[\eta^{-1}]$.
\end{para}

\begin{para}
\label{p:cup_product}
By a ring spectrum, resp.\ commutative ring spectrum, we will mean a monoid, resp.\ commutative monoid, in $(\SH(S),\wedge,\Un_S)$. Let $A \in \SH(S)$ be a ring spectrum, with multiplication map $\mu \colon A \wedge A \to A$. Let $X \in \Sm_S$, and $\Delta_X \colon X \to X \times_S X$ its diagonal. If $x,y \in \Aa(X)$, we denote by $x \cup y \in \Aa(X)$, or simply $xy$, the composite in $\SH(S)$
\[
\Sup X \xrightarrow{\Delta_X} \Sup(X \times_S X) = (\Sup X) \wedge (\Sup X) \xrightarrow{x \wedge y} A \wedge A \xrightarrow{\mu} A.
\]
More generally let $V,W \to X$ be vector bundles, and denote by $p_1,p_2 \colon X \times_S X \to X $ the projections. Then $\Delta_X$ induces a morphism in $\SH(S)$
\[
\Th_X(V \oplus W) = \Th_X(\Delta_X^*(p_1^*V \oplus p_2^*W)) \to \Th_{X \times_S X}(p_1^*V \oplus p_2^*W) = \Th_X(V) \wedge \Th_X(W)
\]
that we denote again by $\Delta_X$. As above, if $x\in \Aa(X;V)$ and $y\in \Aa(Y;W)$, we may define an element $x \cup y \in \Aa(X;V \oplus W)$ as $\mu \circ (x \wedge y) \circ \Delta_X$.

In particular, this endows $\Aa(X)$ with a ring structure, and $\Aa(X;V)$ with an $\Aa(X)$-bimodule structure.

Observe that the isomorphism of \eqref{eq:twist_shift} is given by the cup product with $\Sigma^{2d,d}1 \in A^{0,0}(X;1^{\oplus d})$, so that we have
\begin{equation}
\label{eq:Sigma_cup}
\Sigma^{2d,d}(x\cup y) = x \cup \Sigma^{2d,d}y, \quad \text{for any $x \in \Aa(X;V),y \in \Aa(X;W)$}.
\end{equation}
\end{para}

\begin{para}
\label{p:algebra}
All rings will be associative and unital. Let $R$ be a ring. In this paper, an $R$-algebra will mean a monoid in the category of $R$-bimodules. By $R[x_1,\dots,x_n]$ we will mean the polynomial ring where the variables $x_1,\dots,x_n$ are understood to be central.
\end{para}

\begin{para}
\label{p:homogeneous_power_series}
Let $M$ be an abelian group, and $R$ an $M$-graded ring. If $m_1,\dots,m_r\in M$, we may view $R[x_1,\dots,x_r]$ as an $M$-graded $R$-algebra, where each $x_i$ has degree $m_i$. We let $L \subset R[x_1,\dots,x_r]$ is the (two-sided) ideal generated by $x_1,\dots,x_r$. We will denote by $R[[x_1,\dots,x_r]]_h$ the $M$-graded ring of \emph{homogeneous power series}, defined as the limit in the category of $M$-graded $R$-algebras of $R[x_1,\dots,x_r]/L^n$ for $n\in \Nn$. (Each homogeneous component is the limit in the category of abelian groups of the corresponding homogeneous components of $R[x_1,\dots,x_r]/L^n$.)
\end{para}

\begin{para}
\label{p:Thom_dual_vb}
(See \cite[Lemma~2]{Ana-Pushforwards}.)
Let $X \in \Sm_S$, and $E,V \to X$ be vector bundles. We claim that there exists a canonical isomorphism in $\Hop(S)$
\begin{equation}
\label{eq:sigma_E}
\sigma_E \colon \Th_X(V \oplus E) \xrightarrow{\sim} \Th_X(V \oplus E^\vee).
\end{equation}
Indeed we reduce to the case $X=S$ using the functor $f_\sharp$ of \rref{p:f_sharp}, and then to the case $V=0$ using the functor $\Sigma_V$ of \rref{p:sigma_V}. Consider then the closed subscheme $Y \subset E \oplus E^\vee$ consisting of those pairs $(x,f)$ such that $f(x)=1$. Then the projection $Y \to E^\circ$ given by $(x,f) \mapsto x$ is an affine bundle, hence a weak equivalence of motivic spaces. Since the projection $E \oplus E^\vee \to E$ is a weak equivalence, we obtain a weak equivalence $(E \oplus E^\vee)/Y \to E/E^\circ = \Th_X(E)$. Similarly we have a weak equivalence  $(E \oplus E^\vee)/Y \to \Th_X(E^\vee)$. This yields the required isomorphism $\sigma_E \colon \Th_X(E) \xrightarrow{\sim} \Th_X(E^\vee)$ in $\Hop(S)$. Let us record that we have a commutative diagram in $\Hop(S)$, where the maps from $X_+$ are the zero-sections,
\begin{equation}
\label{eq:diag_Thom_dual}
\begin{gathered}
 \xymatrix{
&E_+\ar[r] &  \Th_X(E)\ar@/^4.0pc/[dd]^{\sigma_E} \\ 
X_+ \ar[r] \ar[ru] \ar[rd] & (E \oplus E^\vee)_+\ar[d] \ar[u] \ar[r]&(E \oplus E^\vee)/Y\ar[d] \ar[u]\\
&(E^\vee)_+ \ar[r] & \Th_X(E^\vee) 
}
\end{gathered}
\end{equation}
\end{para}

\begin{para}
\label{p:epsilon}
Let us denote by $\can \colon \Th_S(1^\vee) \to \Th_S(1)$ the morphism in $\SH(S)$ induced by the canonical isomorphism $1^\vee \xrightarrow{\sim} 1$ of vector bundles over $S$. We consider the isomorphism in $\SH(S)$
\[
\epsilon \colon \Un_S = \Sigma^{-2,-1}\Th_S(1) \xrightarrow{\Sigma^{-2,-1}(\sigma_1)} \Sigma^{-2,-1}\Th_S(1^\vee) \xrightarrow{\Sigma^{-2,-1}(\can)} \Sigma^{-2,-1}\Th_S(1) = \Un_S.
\]
For any motivic ring spectrum $A \in \SH(S)$, we will write $\epsilon \in A^{0,0}(S)$ instead of $\epsilon^*(1)$.

Let us mention that $\epsilon=-\langle -1 \rangle$, using a standard notation (see e.g.\ \cite[Lemma~6.3.5]{Morel-Intro_A1}). 

Since $\eta \epsilon = \eta$ by \cite[Lemma~6.2.3]{Morel-Intro_A1} (where the fact that base scheme is the spectrum of a perfect field plays no role), we have $\epsilon =1 \in  A^{0,0}(S)$ when $A$ is $\eta$-periodic.
\end{para}

\section{Hyperbolic structures}
\numberwithin{theorem}{subsection}
\numberwithin{lemma}{subsection}
\numberwithin{proposition}{subsection}
\numberwithin{corollary}{subsection}
\numberwithin{example}{subsection}
\numberwithin{notation}{subsection}
\numberwithin{definition}{subsection}
\numberwithin{remark}{subsection}

\subsection{Zero-sections}
\begin{para}
\label{p:z}
Let $E\to X$ be a vector bundle with $X \in \Sm_S$. We denote by
\begin{equation}
\label{eq:p:z:1}
z_E \colon X_+ \to \Th_X(E) \in \Hop(S)
\end{equation}
the composite of the zero-section $X_+ \to E_+$ followed by the canonical map $E_+ \to E/(E^\circ) = \Th_X(E)$. Equivalently, the map $z_E$ is the composite $X_+ = \Th_X(0) \to \Th_X(E)$ where the second map is induced by the inclusion $0 \subset E$. More generally, when $V \to X$ is a vector bundle, the zero-section of $E$ induces an inclusion $V \to E \oplus V$, and thus a morphism of pointed motivic spaces
\begin{equation}
\label{eq:p:z:2}
z_E \colon \Th_X(V) \to \Th_X(E\oplus V).
\end{equation}
When $X=S$, observe that \eqref{eq:p:z:2} is obtained from \eqref{eq:p:z:1} by applying the functor $\Sigma^V$ of \rref{p:sigma_V}.
\end{para}

\begin{para}
\label{p:z_1}
The element $1 \in H^0(S,\Gm)$ yields a section of the projection $\Gm \to S$, which implies that the map $z_1 \colon S_+ \to (\Au)_+/(\Gm)_+=\Th_S(1)$ is zero in $\Hop(S)$.
\end{para}

We will use the following splitting principle from \cite{eta}:
\begin{proposition}[{\cite[(4.2.5)]{eta}}]
\label{prop:splitting_odd}
Let $X \in \Sm_S$, and $V \to X$ be a vector bundle of constant odd rank. Then there exists a morphism $f\colon Y \to X$ in $\Sm_S$ whose image in $\SH(S)[\eta^{-1}]$ admits a section, and a vector bundle $W \to Y$ such that $f^*V \simeq W \oplus 1$.
\end{proposition}

\begin{lemma}
\label{lemm:zero_pb}
Let $X \in \Sm_S$. Let $E \to X$ be a vector bundle, and $F \subset E$ a subbundle such that $E/F$ has constant odd rank. Then the morphism $\Th_X(F) \to \Th_X(E)$ vanishes in $\SH(S)[\eta^{-1}]$.
\end{lemma}
\begin{proof}
Let $Q=E/F$. By \rref{p:Jouanolou}, we may assume that $F \subset E$ extends to an isomorphism $Q \oplus F \simeq E$. Then the composite $\Th_X(F) \to \Th_X(E) \simeq \Th_X(Q \oplus F)$ is the map $z_Q$ of \rref{eq:p:z:2}. Applying the functor $f_\sharp \colon \SH(X)[\eta^{-1}] \to \SH(S)[\eta^{-1}]$ of \rref{p:f_sharp} (in view of \rref{p:eta_periodic}) we may assume that $X=S$, and applying the functor $\Sigma^F \colon \SH(S)[\eta^{-1}] \to \SH(S)[\eta^{-1}]$ of \rref{p:sigma_V} (in view of \rref{p:eta_periodic}) we reduce to the case $F=0$. It will thus suffice to show that $z_Q \colon \Un_S \to \Th_S(Q)$ vanishes in $\SH(S)[\eta^{-1}]$. By \rref{prop:splitting_odd}, we may find a morphism $f \colon Y \to S$ in $\Sm_S$ admitting a section $\sigma$ in $\SH(S)[\eta^{-1}]$, and such that $f^*Q$ admits the trivial line bundle as a direct summand over $Y$. In view of the commutative diagram in $\SH(S)[\eta^{-1}]$ (where horizontal composites are the identities)
\[ \xymatrix{
\Un_S \ar[rr]^{\sigma} \ar[d]_{z_Q} && \Sup Y \ar[rr]^f \ar[d]^{\id \wedge z_Q}&& \Un_S \ar[d]^{z_Q}\\ 
\Th_S(Q) \ar[rr]^-{\sigma \wedge \id_{\Th_S(Q)}}&&\Sup Y \wedge \Th_S(Q) \ar[rr]^-{f \wedge \id_{\Th_S(Q)}}&&\Th_S(Q)
}\]
it will suffice to prove the vanishing of the middle vertical map. That map may be identified with $z_{f^*Q} \colon \Sup Y \to \Th_Y(Q)$, hence factors through $z_1 \colon \Sup Y \to \Th_Y(1)$ (because the inclusion $0 \subset f^*Q$ factors through $1$ over $Y$), which vanishes by \rref{p:z_1}.
\end{proof}

\subsection{Hyperbolic orientations}

\begin{para}
\label{p:sw}
Let $E,F \to X$ be vector bundles with $X \in \Sm_S$. We denote by
\[
\sw_{E,F} \colon \Th_X(E \oplus F \oplus E \oplus F) \xrightarrow{\sim} \Th_X(E \oplus E \oplus F \oplus F)
\]
the isomorphism of pointed motivic spaces over $S$ given by $(e,f,e',f') \mapsto (e,e',f,f')$.
\end{para}

\begin{definition}
\label{def:hyp_orientation}
Let $A \in \SH(S)$ be a ring spectrum. A \emph{hyperbolic orientation} on $A$ is the datum of a class
\[
\hypo_E \in A^{0,0}(X;E \oplus E)
\]
for each vector bundle $E \to X$ with $X \in \Sm_S$, subject to the following conditions:
\begin{enumerate}[(i)]
\item \label{def:hyp_orientation:central}
the class $\hypo_E$ is $\Aa(X)$-central,
\item \label{def:hyp_orientation:funct}
if $f \colon Y \to X$ is a morphism in $\Sm_S$ and $E \to X$ a vector bundle, then $f^*\hypo_E = \hypo_{f^*E}$,

\item \label{def:hyp_orientation:isom}
if $E \xrightarrow{\sim} F$ is an isomorphism of vector bundles over $X \in \Sm_S$, then the induced isomorphism $\Aa(X;F \oplus F) \xrightarrow{\sim} \Aa(X;E \oplus E)$ maps $\hypo_F$ to $\hypo_E$,

\item 
\label{def:hyp_orientation:norm}
$\hypo_1 =\Sigma^{4,2}1$,

\item
\label{def:hyp_orientation:mult}
if $E,F$ are vector bundles over $X \in \Sm_S$, we have $\hypo_{E \oplus F} = \sw_{E,F}^*(\hypo_E \cup \hypo_F)$.
\end{enumerate}
\end{definition}

\begin{remark}
Axioms \dref{def:hyp_orientation}{def:hyp_orientation:norm} and \dref{def:hyp_orientation}{def:hyp_orientation:mult} imply that $\hypo_0=1$.
\end{remark}

\begin{para}
\label{p:com_hyp}
If the ring spectrum $A$ is commutative, then the axiom \dref{def:hyp_orientation}{def:hyp_orientation:central} is automatically satisfied: indeed if $E$ has constant rank $r$, then $\hypo_E \in A^{0,0}(X;E \oplus E)= A^{4r,2r}(\Th_X(E \oplus E))$ commutes with every elements of $\Aa(X)$ (see e.g.\ \cite[Theorem~2.4]{PW-MSL-Msp}).
\end{para}

\begin{lemma}
\label{lemm:hyp_isom}
Let $A \in \SH(S)$ be a hyperbolically oriented ring spectrum. If $E,V$ are vector bundles over $X \in \Sm_S$, then the morphism
\[
\Aa(X;V) \xrightarrow{\sim} \Aa(X;V \oplus E\oplus E), \quad x \mapsto x \cup \hypo_E
\]
is bijective.
\end{lemma}
\begin{proof}
By a Mayer--Vietoris argument and \dref{def:hyp_orientation}{def:hyp_orientation:funct}, we can assume that $E \simeq 1^{\oplus s}$ for some $s\in \Nn$, and then by \dref{def:hyp_orientation}{def:hyp_orientation:isom} that $E=1^{\oplus s}$. We reduce to the case $s=1$ using \dref{def:hyp_orientation}{def:hyp_orientation:mult} (replacing $V$ with $V \oplus 1^{\oplus 2n}$ for $n=0,\dots,s-1$), and conclude with \dref{def:hyp_orientation}{def:hyp_orientation:norm}.
\end{proof}

\begin{example}
\label{ex:Sp_hyp}
Let $A \in \SH(S)$ be a commutative ring spectrum equipped with a normalised $\Sp$-orientation, in the sense of \cite[Definition~3.3]{Ana-SL}. Thus every symplectic bundle $V \to X$ with $X \in \Sm_S$ has a Thom class $\thom_V \in A^{0,0}(X;V)$. For a vector bundle $E \to X$ with $X \in \Sm_S$ of constant rank $n$, we set (using the element $\epsilon$ of \rref{p:epsilon})
\begin{equation}
\label{eq:hyp_Sp}
\hypo_E = \epsilon^{-n} \cdot \sigma_E^*(\thom_{H(E)}) \in A^{0,0}(X;E \oplus E)
\end{equation}
where $\sigma_E \colon \Th_X(E\oplus E) \xrightarrow{\sim} \Th_X(E\oplus E^\vee)$ is the isomorphism of \rref{p:Thom_dual_vb}, and $H(E)$ the vector bundle $E \oplus E^\vee$ equipped with the hyperbolic symplectic form. Then one verifies that $E \mapsto \hypo_E$ defines a hyperbolic orientation of $A$. Similarly any normalised $\Oo$-orientation (defined in the expected way) yields a hyperbolic orientation. (This is the origin of the terminology of \rref{def:hyp_orientation}.)

Recall from \rref{p:epsilon} that $\epsilon=1 \in A^{0,0}(S)$ when $A$ is $\eta$-periodic, so that the formula \eqref{eq:hyp_Sp} simplifies to
\begin{equation}
\label{eq:hyp_Sp_eta}
\hypo_E = \sigma_E^*(\thom_{H(E)}) \quad \text{when $A$ is $\eta$-periodic.}
\end{equation}
\end{example}

\begin{definition}
\label{def:Euler:hyp}
Let $A \in \SH(S)$ be a hyperbolically oriented ring spectrum. When $E \to X$ is a vector bundle with $X \in \Sm_S$, we define its \emph{Euler class} 
\[
e(E) = z_E^*(\hypo_E) \in \Aa(X;E)
\]
(the map $z_E$ was defined in \rref{p:z}), and its \emph{top Pontryagin class}
\[
\pi(E) = z_{E\oplus E}^*(\hypo_E) \in \Aa(X).
\]
If $E$ has constant rank $r$, then $e(E) \in A^{2r,r}(X;E)$ and $\pi(E) \in A^{4r,2r}(X)$.
\end{definition}

\begin{para}
\label{prop:Euler}
The following basic properties of the Euler and top Pontryagin classes are easily verified, where $E\to X$ is a vector bundle with $X \in \Sm_S$:
\begin{enumerate}[(i)]
\item the elements $e(E) \in \Aa(X;E)$ and $\pi(E) \in \Aa(X)$ are $\Aa(X)$-central,
\item \label{prop:Euler:pb}
$e(f^*E) = f^*e(E)$ and $\pi(f^*E) = f^*(\pi(E))$ for any morphism $f\colon Y \to X$ in $\Sm_S$,
\item \label{prop:Euler:1}
$e(1) =0$ and $\pi(1) =0$,
\item \label{prop:Euler:0}
$e(0) =1$ and $\pi(0) =1$,
\item \label{prop:Euler:mult}
$e(E\oplus F) = e(E) \cup e(F)$ and $\pi(E\oplus F) = \pi(E) \pi(F)$ for any vector bundle $F \to X$.
\end{enumerate}
\end{para}

The next proposition expresses the familiar fact that the top Pontryagin class is the square of the Euler class.
\begin{proposition}
Let $A \in \SH(S)$ be a hyperbolically oriented ring spectrum. Let $X \in \Sm_S$, and $E\to X$ be a vector bundle. Then
\[
e(E) \cup e(E) = \pi(E) \cup \hypo_E \in \Aa(X;E \oplus E).
\]
\end{proposition}
\begin{proof}
Consider the maps in $\SH(S)$
\[
a \colon \Th_X(E^{\oplus 2}) \to \Th_X(E^{\oplus 4}), \; b \colon \Th_X(E^{\oplus 2}) \to \Th_X(E^{\oplus 4}), \; c \colon \Th_X(E^{\oplus 4}) \to \Th_X(E^{\oplus 4})
\]
respectively given by the injective matrices
\[
A=\begin{pmatrix}
0 & 0\\
1 & 0\\
0 & 0\\
0 & 1
\end{pmatrix}, \quad
B=\begin{pmatrix}
0 & 0\\
0 & 0\\
1 & 0\\
0 & 1
\end{pmatrix}, \quad
C=\begin{pmatrix}
-1 & 0 & 0 & 0\\
0 &  0 & 1 & 0\\
0 &  1 & 0 & 0\\
0 &  0 & 0 & 1
\end{pmatrix}.
\]
Then in $\Aa(X;E \oplus E)$
\begin{equation}
\label{eq:e_pi}
e(E) \cup e(E) = a^*(\hypo_E\cup \hypo_E) \quad ; \quad \pi(E) \cup \hypo_E=b^*(\hypo_E\cup \hypo_E).
\end{equation}
Since $CA=B$, we have $c \circ a =b$. As the matrix $C$ has coefficients in $\Zz$ and determinant one, it is a product of transvections, hence we have $c=\id$ in $\SH(S)$ (see e.g.\ \cite[Lemma~1]{Ana-Pushforwards} where the fact $S$ is the spectrum of a field plays no role). Thus $a=b$ in $\SH(S)$, and the result follows from the formulas \eqref{eq:e_pi}.
\end{proof}

The next proposition is a variant of results of Ananyevskiy \cite[Corollary~2]{Ananyevskiy-SL_PB}, \cite[Theorem~7.4]{Ana-SL} and Levine \cite[Lemma~4.3]{Levine-motivic_Euler}.

\begin{proposition}
\label{prop:Euler_odd}
Let $A \in \SH(S)$ be an $\eta$-periodic hyperbolically oriented ring spectrum. Let $E$ be a vector bundle over $X \in \Sm_S$. If $E$ admits a quotient of constant odd rank, then
\[
e(E)=0 \quad \text{ and } \quad \pi(E)=0.
\]
\end{proposition}
\begin{proof}
Let $F \subset E$ be a subbundle such that $E/F$ has odd rank. Then by definition $e(E)$ is the image of $\hypo_E$ under the pullback along the composite $\Th_X(E) \xrightarrow{z_F} \Th_X(F \oplus E) \to \Th_X(E \oplus E)$ in $\SH(S)$. The latter vanishes in $\SH(S)[\eta^{-1}]$ by \rref{lemm:zero_pb} (applied to the subbundle $F\oplus E \subset E \oplus E$), hence $e(E)=0$. Therefore $\pi(E) = z_E^*(e(E))=0$.
\end{proof}

\subsection{Weak hyperbolic orientations}

\begin{definition}
\label{def:hyp_thom}
Let $A \in \SH(S)$ be a ring spectrum. A \emph{weak hyperbolic orientation} on $A$ is the datum of a class
\[
\hypt_E \in A^{0,0}(X;E \oplus E)
\]
for each rank two vector bundle $E \to X$ with $X \in \Sm_S$, subject to the following conditions:
\begin{enumerate}[(i)]
\item
\label{def:hyp_thom:central}
the class $\hypt_E$ is $\Aa(X)$-central,

\item 
\label{def:hyp_thom:funct}
if $f \colon Y \to X$ is a morphism in $\Sm_S$ and $E \to X$ a rank two vector bundle, then $f^*\hypt_E = \hypt_{f^*E}$,

\item 
\label{def:hyp_thom:isom}
if $E \xrightarrow{\sim} F$ is an isomorphism of rank two vector bundles over $X \in \Sm_S$, then the induced isomorphism $\Aa(X;F \oplus F) \xrightarrow{\sim} \Aa(X;E \oplus E)$ maps $\hypt_F$ to $\hypt_E$,

\item \label{def:hyp_thom:norm}
$\hypt_{1^{\oplus 2}} =\sw_{1,1}^*(\Sigma^{8,4}1)$ (in the notation of \rref{p:sw}).
\end{enumerate}
\end{definition}

\begin{remark}
Using a standard notation (see e.g.\ \cite[Lemma~6.3.4]{Morel-Intro_A1}), we have $\sw_{1,1}^*(\Sigma^{8,4}1)=\Sigma^{8,4}\langle -1 \rangle$.
\end{remark}

\begin{remark}
It is clear that a hyperbolic orientation (see \rref{def:hyp_orientation}) on a motivic ring spectrum $A$ induces a weak hyperbolic orientation. We will see in \rref{th:Thom_orientation} that the two notions in fact coincide when $A$ is $\eta$-periodic.
\end{remark}

\begin{para}
\label{p:com_whyp}
As explained in \rref{p:com_hyp}, the axiom \dref{def:hyp_thom}{def:hyp_thom:central} is automatically satisfied when the ring spectrum $A$ is commutative.
\end{para}

\begin{lemma}
\label{lemm:whyp_isom}
Let $A \in \SH(S)$ be a hyperbolically oriented ring spectrum. If $E \to X$ is a rank two vector bundle with $X \in \Sm_S$, then the morphism
\[
\Aa(X) \xrightarrow{\sim} \Aa(X;E\oplus E), \quad x \mapsto x \cup \hypt_E
\]
is bijective.
\end{lemma}
\begin{proof}
By a Mayer--Vietoris argument and \dref{def:hyp_thom}{def:hyp_thom:funct}, we can assume that $E \simeq 1^{\oplus 2}$, and then by \dref{def:hyp_thom}{def:hyp_thom:isom} that $E=1^{\oplus 2}$. Then the statement follows from \dref{def:hyp_thom}{def:hyp_thom:norm}.
\end{proof}

\begin{definition}
\label{def:Euler_thom}
Let $A \in \SH(S)$ be a ring spectrum with a weak hyperbolic orientation. For every rank two vector bundle $E\to X$ with $X \in \Sm_S$, we define its \emph{Euler class}
\[
e(E) = z_E^*(\hypt_E) \in A^{4,2}(X;E)
\]
(the map $z_E$ was defined in \rref{p:z}), and \emph{top Pontryagin class}
\[
\pi(E) =z_{E \oplus E}^*(\hypt_E) \in A^{8,4}(X).
\]
\end{definition}

\begin{lemma}
\label{lemm:pi_dual}
Let $A\in \SH(S)$ be an $\eta$-periodic ring spectrum with a weak hyperbolic orientation. Let $X \in \Sm_S$ and $E \to X$ be a rank two vector bundle. Then
\[
\pi(E) = \pi(E^\vee).
\]
\end{lemma}
\begin{proof}
By \rref{prop:splitting_odd} we may assume that the line bundle $\det E$ admits a trivialisation. Then the alternated form $E^{\otimes 2} \to \Lambda^2 E=\det E \simeq 1$ is nondegenerate, hence $E \simeq E^\vee$ as vector bundles.
\end{proof}

\begin{remark}
Lemma \rref{lemm:pi_dual} will be generalised to vector bundles of higher rank when $A$ is hyperbolically oriented (see \rref{cor:Pontryagin_dual} and \rref{lemm:Pontryagin_pi}).
\end{remark}

\subsection{Pushforwards along closed immersions}
\begin{para}
\label{p:pf}
(See e.g.\ \cite[\S3.5]{Hoyois-equivariant}.)
Let $i \colon Y \to X$ be a closed immersion in $\Sm_S$ with normal bundle $N$, and $u \colon U \to X$ its open complement. There exists a canonical isomorphism $X/U \simeq \Th_Y(N)$ in $\Hop(S)$, called purity isomorphism. We 
consider the composite
\[
\pf{i} \colon X_+ \to X/U \simeq \Th_Y(N) \; \text{in $\Hop(S)$.}
\]
More generally when $V\to X$ is a vector bundle, a map in $\Hop(S)$
\[
\pf{i}\colon \Th_X(V) \to \Th_Y(N \oplus i^*V)
\]
is constructed from the above map by first using the functor $f_\sharp \colon \Hop(X) \to \Hop(S)$ of \rref{p:f_sharp} to reduce to the case when $S=X$, and then applying the functor $\Sigma^V$ of \rref{p:sigma_V} to reduce to the case $V=0$. We then have a distinguished triangle in $\SH(S)$
\begin{equation}
\label{eq:loc_SH}
\Th_U(V) \xrightarrow{u} \Th_X(V) \xrightarrow{\pf{i}} \Th_Y(N \oplus i^*V) \to \Sigma^{1,0} \Th_U(V).
\end{equation}
\end{para}

\begin{para}
\label{p:localisation}
In the situation of \rref{p:pf}, let $A \in \SH(S)$ be a ring spectrum. When $N$ has constant rank $r$, we will write
\[
i_* \colon A^{p-2r,q-r}(Y;N \oplus i^*V)=A^{p,q}(\Th_Y(N\oplus i^*V)) \xrightarrow{\pf{i}^*} A^{p,q}(X;V),
\]
and extend this notation in an obvious fashion to the case when $N$ is arbitrary. We thus have a long exact sequence
\begin{equation}
\label{eq:loc_coh}
\dots \to \Aa(Y;N\oplus i^*V) \xrightarrow{i_*} \Aa(X;V) \xrightarrow{u^*} \Aa(U;V) \to \cdots
\end{equation}
\end{para}

\begin{para}
\label{p:pf_zero_section}
Let $X \in \Sm_S$, and $E \to X$ be a vector bundle. It follows from the discussion in \cite[\S2.4.5]{Panin-Oriented_II} that the purity isomorphism $E/E^\circ \simeq \Th_X(E)$ in $\Hop(S)$ coincides with the identification arising from the definition of the Thom space $\Th_X(E)$. 
\end{para}

\begin{para}
\label{p:pf_composite}
(See e.g.\ \cite[p.233]{Hoyois-equivariant}.)
Let $i\colon Y \to X$ and $j\colon Z \to Y$ be closed immersions in $\Sm_S$. Denote by $N_i,N_j,N_{j\circ i}$ the respective normal bundles of $i,j,j\circ i$. Then we have an exact sequence of vector bundles over $Z$
\[
0 \to N_j \to N_{j\circ i} \to j^*N_i \to 0.
\]
Let $V \to X$ be a vector bundle. Then the composite
\[
\Th_X(V) \xrightarrow{\pf{i}} \Th_Y(N_i \oplus i^*V) \xrightarrow{\pf{j}} \Th_Z(N_j \oplus j^*N_i \oplus j^*i^*V) \overset{\rref{p:Thom_ses}}{\simeq} \Th_Z(N_{j\circ i})
\]
coincides with $\pf{i\circ j}$ in $\Hop(S)$.
\end{para}

\begin{para}
\label{p:pf_funct}
Consider a cartesian square in $\Sm_S$
\begin{equation}
\label{squ:pf_funct}
\begin{gathered}
\xymatrix{
Y'\ar[r]^{i'} \ar[d]_g & X' \ar[d]^f \\ 
Y \ar[r]^i & X
}
\end{gathered}
\end{equation}
where $i,i'$ are closed immersions with respective normal bundles $N,N'$. Let $V \to X$ be a vector bundle. Then we have a commutative diagram in $\Hop(S)$, where $e$ is induced by the natural inclusion $N' \subset g^*N$, and $h=i\circ g = f\circ i'$,
\begin{equation}
\label{diag:pf_funct}
\begin{gathered}
\xymatrix{
\Th_{Y'}(N \oplus h^*V) \ar[d]^g &\Th_{Y'}(N'\oplus h^*V)\ar[l]_e &\Th_{X'}(V) \ar[d]^f \ar[l]_-{\pf{i'}}\\ 
\Th_Y(N) && \Th_X(V) \ar[ll]_-{\pf{i}} 
}
\end{gathered}
\end{equation}
\end{para}

\begin{para}
\label{p:pf_empty}
If $Y'=\varnothing$ in \rref{p:pf_funct}, then it follows from the commutative diagram \eqref{diag:pf_funct} that the composite $\Th_{X'}(V) \xrightarrow{f} \Th_X(V) \xrightarrow{\pf{i}} \Th_Y(N \oplus i^*V)$ vanishes in $\Hop(S)$.
\end{para}

\begin{para}
\label{p:transverse}
We say that the cartesian square \eqref{squ:pf_funct} is \emph{transverse} if the natural inclusion $N' \subset g^*N$ is an isomorphism, in which case the map $e$ in the diagram \eqref{diag:pf_funct} is an isomorphism.
\end{para}

We will use the following form of the projection formula:
\begin{lemma}
\label{lemm:proj_formula}
Let $A \in \SH(S)$ be a ring spectrum. Let $i\colon Y \to X$ be a closed immersion in $\Sm_S$ with normal bundle $N$. 
\begin{enumerate}[(i)]
\item
\label{lemm:proj_formula:1}
Let $V \to X$ be a vector bundle, and $x \in \Aa(Y;N)$ and $a\in \Aa(X;V)$. Then
\[
i_* (x \cup i^*a) = i_*(x) \cup a \in \Aa(X;V).
\]
\item
\label{lemm:proj_formula:2}
If $b \in \Aa(X)$ and $y \in \Aa(Y;N)$, then we have
\[
i_* (i^*b \cup y) = b \cup i_*(y) \in \Aa(X)
\]
\end{enumerate}
\end{lemma}
\begin{proof}
We use the notation of \rref{p:cup_product}. The cartesian squares in $\Sm_S$
\[ \xymatrix{
Y\ar[rr]^-i \ar[d]_{(\id_Y,i)} && X \ar[d]^{\Delta_X} 
&& Y\ar[rr]^-i \ar[d]_{(i,\id_Y)} && X \ar[d]^{\Delta_X}\\ 
Y \times_S X \ar[rr]^-{i \times \id_X} && X \times_S X
&& X \times_S Y \ar[rr]^-{\id_X \times i} && X \times_S X
}\]
are transverse, and thus yield by \rref{p:pf_funct} and \rref{p:transverse} commutative diagrams in $\Hop(S)$
\[ \xymatrix{
\Th_Y(N \oplus i^*V) \ar[d] && \Th_X(V)\ar[ll]_-{\pf{i}} \ar[d]^{\Delta_X} 
& \Th_Y(N) \ar[d] && X_+\ar[ll]_-{\pf{i}} \ar[d]^{\Delta_X}\\ 
\Th_Y(N)\wedge \Th_X(V) && X_+ \wedge \Th_X(V) \ar[ll]_-{\pf{i} \wedge \id}
&X_+ \wedge \Th_Y(N) && X_+ \wedge X_+ \ar[ll]_-{\id \wedge \pf{i}}
}\]
The left vertical morphism of the left square factors as
\[
\Th_Y(N \oplus i^*V) \xrightarrow{\Delta_Y} \Th_Y(N) \wedge \Th_Y(V) \xrightarrow{\id \wedge i} \Th_Y(N) \wedge \Th_X(V),
\]
while the left vertical morphism of the right square factors as
\[
\Th_Y(N) \xrightarrow{\Delta_Y} Y_+ \wedge \Th_Y(N) \xrightarrow{i \wedge \id} X_+ \wedge \Th_Y(N).
\]
Thus, denoting by $\mu \colon A \wedge A \to A$ the product, for $x,a$ as in \eqref{lemm:proj_formula:1} we have
\begin{align*}
i_* (x \cup i^*a) 
&= i_* \circ \Delta_Y^* \circ (\id \wedge i)^*(\mu \circ (x \wedge a)) \\ 
&= \Delta_X^* \circ (\pf{i} \wedge \id)^*(\mu \circ (x \wedge a))\\
&= i_*(x) \cup a,
\end{align*}
proving the first formula. The other formula follows from the computation
\begin{align*}
i_* (i^*b \cup y)
&= i_* \circ \Delta_Y^* \circ (i \wedge \id)^*(\mu \circ (b \wedge y))\\
&= \Delta_X^* \circ (\id \wedge \pf{i})^*(\mu \circ (b \wedge y))\\
&= b \cup i_*(y).\qedhere
\end{align*}
\end{proof}

\begin{para}
\label{p:transverse_section}
Let $X \in \Sm_S$ and $E\to X$ be a vector bundle. Let $s$ be a section of $E$, and consider its zero-locus $i \colon Y \to X$, defined as the equaliser of $s$ and the zero-section in the category of $S$-schemes. We will say that $s$ is \emph{transverse to the zero-section} if $Y\in \Sm_S$ and the natural inclusion $N \subset i^*E$ is an isomorphism, where $N$ is the normal bundle of $i$.
\end{para}

\begin{lemma}
\label{lemm:push_pull}
Let $E,V$ be vector bundles over $X \in \Sm_S$. Let $s\colon X \to E$ be a section transverse to the zero-section, whose zero-locus we denote by $i\colon Y \to X$. Then in $\Hop(S)$ the composite
\[
\Th_X(V) \xrightarrow{\pf{i}} \Th_Y(E \oplus V) \xrightarrow{i} \Th_X(E \oplus V)
\]
coincides with $z_E$ (see \rref{p:z}).
\end{lemma}
\begin{proof}
Using the functor $f_\sharp \colon \Hop(X) \to \Hop(S)$ of \rref{p:f_sharp}, we reduce to the case when $S=X$. Applying the functor $\Sigma^V$ of \rref{p:sigma_V}, we reduce to the case $V=0$.

Let us denote by $z \colon X \to E$ the zero-section. The cartesian square in $\Sm_S$
\[ \xymatrix{
Y\ar[r]^i \ar[d]_i & X \ar[d]^s \\ 
X \ar[r]^z & E
}\]
is transverse, and yields by \rref{p:pf_funct} and \rref{p:transverse} a commutative square in $\Hop(S)$
\[ \xymatrix{
\Th_Y(E) \ar[d]_i & X_+ \ar[l]_-{\pf{i}} \ar[d]^{s_+} \\ 
\Th_X(E) & E_+\ar[l]_-{\pf{z}} 
}\]
Now the maps $s,z \colon X \to E$ coincide in $\Hot(S)$, being sections of the vector bundle projection $E\to X$. Thus $s_+=z_+ \colon X_+ \to E_+$ in $\Hop(S)$. In view of \rref{p:pf_zero_section} and \rref{p:z}, the composite $\pf{z} \circ z_+ \colon X_+ \to \Th_X(E)$ coincides with $z_E$, and the statement follows.
\end{proof}

\begin{lemma}
\label{lemm:e_pi_push_pull}
Let $A \in \SH(S)$ be a ring spectrum. Let  $X \in \Sm_S$, and $E \to X$ be a vector bundle. Assume that $E$, resp.\ $E \oplus E$, admits a section transverse to the zero-section, whose zero-locus we denote by $i \colon Y \to X$.
\begin{enumerate}[(i)]
\item
\label{lemm:e_pi_push_pull:1}
If $A$ is endowed with a hyperbolic orientation, then the composite
\[
\Aa(X) \xrightarrow{i^*} \Aa(Y) \xrightarrow{\cup \hypo_{i^*E}}  \Aa(Y;E \oplus E) \xrightarrow{i_*} \Aa(X;E),
\]
\[
\text{resp.} \quad \Aa(X) \xrightarrow{i^*} \Aa(Y) \xrightarrow{\cup \hypo_{i^*E}}  \Aa(Y;E \oplus E) \xrightarrow{i_*} \Aa(X),
\]
is (left or right) multiplication with $e(E)$, resp.\ $\pi(E)$.

\item
\label{lemm:e_pi_push_pull:2}
Assume that $E$ has rank two. If $A$ is endowed with a weak hyperbolic orientation, then the composite
\[
\Aa(X) \xrightarrow{i^*} \Aa(Y) \xrightarrow{\cup \hypt_{i^*E}}  \Aa(Y;E \oplus E) \xrightarrow{i_*} \Aa(X;E),
\]
\[
\text{resp.} \quad \Aa(X) \xrightarrow{i^*} \Aa(Y) \xrightarrow{\cup \hypt_{i^*E}}  \Aa(Y;E \oplus E) \xrightarrow{i_*} \Aa(X)
\]
is (left or right) multiplication with $e(E)$, resp.\ $\pi(E)$.
\end{enumerate}
\end{lemma}
\begin{proof}
Let $\mathfrak{h}_E = \hypo_E \in A^{0,0}(X;E \oplus E)$ in case \eqref{lemm:e_pi_push_pull:1}, and $\mathfrak{h}_E = \hypt_E \in A^{0,0}(X;E \oplus E)$ in case \eqref{lemm:e_pi_push_pull:2}. Then the composite of the statement coincides with
\[
\Aa(X) \xrightarrow{\cup \mathfrak{h}_E} \Aa(X;E\oplus E) \xrightarrow{i^*} \Aa(Y;E \oplus E) \xrightarrow{i_*} \Aa(X;E),
\]
\[
\text{resp.} \quad \Aa(X) \xrightarrow{\cup \mathfrak{h}_E} \Aa(X;E\oplus E) \xrightarrow{i^*} \Aa(Y;E \oplus E) \xrightarrow{i_*} \Aa(X),
\]
hence maps $1$ to $e(E)$, resp.\ $\pi(E)$, by \rref{lemm:push_pull} (with $V=E$, resp.\ $V=0$). The statement then follows from the projection formula \dref{lemm:proj_formula}{lemm:proj_formula:2}.
\end{proof}

\section{Grassmannians of \texorpdfstring{$2$}{2}-planes}

\subsection{Projective bundles}

We will use the following:
\begin{proposition}[{\cite[(4.1.6) and (4.1.5)]{eta}}]
\label{prop:PE}
Let $E \to X$ be a vector bundle of constant rank $r$, with $X \in \Sm_S$.
\begin{enumerate}[(i)]
\item\label{prop:PE:odd}
 If $r$ is odd, then the map $\Sup \Pp(E) \to \Sup X$ is an isomorphism in $\SH(S)[\eta^{-1}]$.

\item\label{prop:PE:even}
 If $r$ is even, then $\Th_{\Pp(E)}(\Oc_{\Pp(E)}(1))=0$ in $\SH(S)[\eta^{-1}]$.
\end{enumerate}
\end{proposition}

\begin{para}
Let $E \to X$ be a vector bundle with $X \in \Sm_S$, and denote by $q \colon \Pp(E) \to X$ the projection. Then $\Oc_{\Pp(E)}(-1)$ is naturally a subbundle of $q^*E$, and we denote by $\Quo_E=q^*E/\Oc_{\Pp(E)}(-1)$ the quotient bundle.
\end{para}

\begin{lemma}
\label{lemm:PE_even_Q}
Let $E \to X$ be a vector bundle of constant even rank, with $X \in \Sm_S$. Then $\Th_{\Pp(E)}(\Quo_E)=0$ in $\SH(S)[\eta^{-1}]$.
\end{lemma}
\begin{proof}
By \rref{p:Thom_ses} we have an isomorphism in $\SH(S)$
\[
\Th_{\Pp(E)}(\Quo_E \oplus \Oc_{\Pp(E)}(-1) \oplus \Oc_{\Pp(E)}(-1)) \simeq \Th_{\Pp(E)}(q^*E \oplus \Oc_{\Pp(E)}(-1)),
\]
where $q\colon \Pp(E) \to X$ is the projection. Now in $\SH(S)[\eta^{-1}]$ we have $\Th_{\Pp(E)}(\Quo_E \oplus \Oc_{\Pp(E)}(-1) \oplus \Oc_{\Pp(E)}(-1)) \simeq \Th_{\Pp(E)}(\Quo_E)$ by \cite[(3.3.1.i)]{eta}, and $\Th_{\Pp(E)}(q^*E \oplus \Oc_{\Pp(E)}(-1)) =0$ by \cite[(4.1.5)]{eta}. This proves the statement.
\end{proof}

\begin{lemma}
\label{lemm:PE_odd_Q}
Let $E \to X$ be a vector bundle of constant even rank, with $X \in \Sm_S$. Consider the closed immersion $i\colon X=\Pp(1) \to \Pp(E \oplus 1)$. Then the natural isomorphism $E \xrightarrow{\sim} i^*\Quo_{E\oplus 1}$ induces an isomorphism in $\SH(S)[\eta^{-1}]$ (see \rref{p:pf})
\[
\pf{i} \colon \Th_{\Pp(E \oplus 1)}(\Quo_{E\oplus 1})\xrightarrow{\sim}\Th_X(E \oplus E).
\]
\end{lemma}
\begin{proof}
The open complement $W$ of $i$ is a line bundle over $\Pp(E)$ (namely $\Oc_{\Pp(E)}(1)$). Its zero-section $\Pp(E) \to W$ induces an isomorphism in $\SH(S)$, and the composite $\Pp(E) \to W \subset \Pp(E \oplus 1)$ is the closed immersion $j\colon \Pp(E) \to \Pp(E \oplus 1)$ induced by the inclusion $E \subset E \oplus 1$. Since the normal bundle to $i$ is $E$ and $j^*\Quo_{E\oplus 1}=\Quo_E \oplus 1$, this yields by \rref{p:pf} a distinguished triangle in $\SH(S)$
\[
\Th_{\Pp(E)}(\Quo_E \oplus 1) \xrightarrow{j} \Th_{\Pp(E \oplus 1)}(\Quo_{E\oplus 1}) \xrightarrow{\pf{i}} \Th_X(E \oplus i^*\Quo_{E\oplus 1}) \to \Sigma^{1,0}\Th_{\Pp(E)}(\Quo_E \oplus 1).
\]
Now $\Th_{\Pp(E)}(\Quo_E \oplus 1) =0$ by \rref{lemm:PE_even_Q}, hence $\pf{i}$ is an isomorphism. The composite $\Oc_{\Pp(E\oplus 1)}(-1) \subset q^*(E \oplus 1) \to 1$ restricts to an isomorphism on the open complement of $j$, where $q \colon \Pp(E \oplus 1) \to X$ is the projection. Since $i$ factors through that open subscheme, it follows that the composite $i^*\Oc_{\Pp(E\oplus 1)}(-1) \to E \oplus 1 \to 1$ is an isomorphism, hence so is the composite $E \subset E \oplus 1 \to (E \oplus 1)/i^*\Oc_{\Pp(E\oplus 1)}(-1) = i^*\Quo_{E\oplus 1}$. This induces an isomorphism $\Th_X(E \oplus i^*\Quo_{E\oplus 1}) \simeq \Th_X(E \oplus E)$, completing the proof.
\end{proof}

\subsection{Grassmann bundles}
We now gather basic observations on Grassmann bundles that will be used repeatedly in the paper.

\begin{para}
\label{def:Gr}
Let $X \in \Sm_S$ and $E\to X$ be a vector bundle. Let $n\in \Zz$. We will denote by $\Gr(n,E)$ the Grassmann bundle of $n$-planes in $E$, classifying the rank $n$ subbundles of $U \subset E$ (for us a subbundle is locally split, so $E/U$ is also a vector bundle). Denoting by $q \colon \Gr(n,E) \to X$ the projection, the scheme $\Gr(n,E)$ carries a universal rank $n$ subbundle $\Uni_n \subset q^*E$, and quotient bundle $\Quo_n = q^*E/\Uni_n$.
\end{para}

\begin{para}
\label{eq:Gr_1_P}
When $E \to X$ is a vector bundle with $X \in \Sm_S$, we have $\Gr(1,E) = \Pp(E)$ and $\Uni_1=\Oc(-1)$. Moreover $\Quo_1$ is the quotient bundle $\Quo=q^*E/\Oc(-1)$, where $q \colon \Pp(E) \to X$ is the projection.
\end{para}

\begin{para}
\label{p:g_h}
In the situation of \rref{def:Gr}, assume that $n\geq 1$. We have a closed immersion
\[
g_E \colon \Gr(n,E) \to \Gr(n,E\oplus 1) \quad ; \quad (U \subset E) \mapsto (U \oplus 0 \subset E \oplus 1),
\]
which satisfies
\begin{equation}
\label{eq:g_pb}
g_E^*\Uni_n=\Uni_n \quad ; \quad g_E^*\Quo_n = \Quo_n \oplus 1.
\end{equation}
We also have a closed immersion
\[
h_E \colon \Gr(n-1,E) \to \Gr(n,E\oplus 1) \quad ; \quad (U \subset E) \mapsto (U \oplus 1 \subset E \oplus 1),
\]
which satisfies
\begin{equation}
\label{eq:h_pb}
h_E^*\Uni_n=\Uni_{n-1} \oplus 1 \quad ; \quad h_E^*\Quo_n = \Quo_{n-1}.
\end{equation}
\end{para}

\begin{para}
\label{p:Gr_n-1_n_n}
The closed immersion $g_E$ of \rref{p:g_h} is the zero-locus of a section of the vector bundle $\Uni_n^\vee \to \Gr(n,E\oplus 1)$ transverse to the zero-section, hence its normal bundle is $g_E^*\Uni_n^\vee=\Uni_n^\vee$. The open complement $Y \subset \Gr(n,E\oplus 1)$ of $g_E$ is naturally a vector bundle over $\Gr(n-1,E)$. Its zero-section $\Gr(n-1,E) \to Y$ induces an isomorphism in $\SH(S)$, and the composite $\Gr(n-1,E) \to Y \subset \Gr(n,E\oplus 1)$ is the closed immersion $h_E$. In view of \eqref{eq:loc_SH} we thus have a distinguished triangle in $\SH(S)$, for any vector bundle $V \to \Gr(n,E \oplus 1)$
\[
\Th_{\Gr(n-1,E)}(V) \xrightarrow{h_E} \Th_{\Gr(n,E\oplus 1)}(V) \xrightarrow{\pf{g_E}} \Th_{\Gr(n,E)}(\Uni_n^\vee \oplus g_E^*V) \to \Sigma^{1,0}\Th_{\Gr(n-1,E)}(V),
\]
and thus, by \eqref{eq:loc_coh}, for any ring spectrum $A \in \SH(S)$ a long exact sequence
\[
\dots \to \Aa(\Gr(n,E);\Uni_n^\vee \oplus g_E^*V) \xrightarrow{{g_E}_*} \Aa(\Gr(n,E\oplus 1);V) \xrightarrow{h_E^*} \Aa(\Gr(n-1,E);V) \to \cdots
\]
\end{para}

\begin{para}
\label{p:Gr_n_n_n-1}
The closed immersion $h_E$ of \rref{p:g_h} is the zero-locus of a section of $\Quo_n \to \Gr(n,E\oplus 1)$ transverse to the zero-section, hence its normal bundle is $h_E^*\Quo_n = \Quo_{n-1}$. The open complement $W \subset \Gr(n,E\oplus 1)$ of $h_E$ is a vector bundle over $\Gr(n,E)$. Its zero-section $\Gr(n,E) \to W$ induces an isomorphism in $\SH(S)$, and the composite $\Gr(n,E) \to W \subset \Gr(n,E\oplus 1)$ is the closed immersion $g_E$. In view of \eqref{eq:loc_SH} we thus have a distinguished triangle in $\SH(S)$, for any vector bundle $V \to \Gr(n,E \oplus 1)$
\[
\Th_{\Gr(n,E)}(V) \xrightarrow{g_E} \Th_{\Gr(n,E\oplus 1)}(V) \xrightarrow{\pf{h_E}} \Th_{\Gr(n-1,E)}(\Quo_{n-1} \oplus h_E^*V) \to \Sigma^{1,0}\Th_{\Gr(n,E)}(V),
\]
and thus, by \eqref{eq:loc_coh}, for any ring spectrum $A \in \SH(S)$ a long exact sequence
\[
\dots \to \Aa(\Gr(n-1,E);\Quo_{n-1} \oplus h_E^*V) \xrightarrow{{h_E}_*} \Aa(\Gr(n,E\oplus 1);V) \xrightarrow{g_E^*} \Aa(\Gr(n,E);V) \to \cdots
\]
\end{para}

\begin{para}
\label{p:g_h_s}
When $X=S$, we will write $\Gr(n,s)$ instead of $\Gr(n,1^{\oplus s})$.
\end{para}

\begin{lemma}
\label{prop:Gr+1}
Let $E \to S$ be a vector bundle of constant even rank. Denote by  $q\colon \Gr(2,E \oplus 1^{\oplus 2}) \to S$ the projection. Then we have isomorphisms in $\SH(S)[\eta^{-1}]$
\begin{equation}
\label{prop:Gr+1:2}
g_E \colon \Sup \Gr(2,E) \xrightarrow{\sim} \Sup \Gr(2,E\oplus 1),
\end{equation}
\begin{equation}
\label{prop:Gr+1:1}
\pf{g_E} \colon \Th_{\Gr(2,E\oplus 1)}(\Uni_2^\vee) \xrightarrow{\sim} \Th_{\Gr(2,E)}(\Uni_2^\vee \oplus \Uni_2^\vee),
\end{equation}
\begin{equation}
\label{prop:Gr+1:3}
(q,\pf{g_{E\oplus 1}}) \colon \Sup \Gr(2,E \oplus 1^{\oplus 2}) \to \Un_S\vo \Th_{\Gr(2,E \oplus 1)}(\Uni_2^\vee).
\end{equation}
\end{lemma}
\begin{proof}
In view of \rref{eq:Gr_1_P}, we have by \rref{p:Gr_n_n_n-1} a distinguished triangle in $\SH(S)$
\[
\Sup \Gr(2,E) \xrightarrow{g_E} \Sup \Gr(2,E\oplus 1) \xrightarrow{\pf{h_E}} \Th_{\Pp(E)}(\Quo) \to \Sigma^{1,0}\Sup \Gr(2,E).
\]
Since $\Th_{\Pp(E)}(\Quo) =0$ in  $\SH(S)[\eta^{-1}]$ by \rref{lemm:PE_even_Q}, we deduce \eqref{prop:Gr+1:2}.

The vector bundle $\Uni_2^\vee$ restricts to $\Oc(1) \oplus 1$ along $h_E \colon \Pp(E) \to \Gr(2,E\oplus 1)$ (see \rref{eq:Gr_1_P} and \rref{eq:h_pb}). Thus by \rref{p:Gr_n-1_n_n}, we have a distinguished triangle in $\SH(S)$
\[
\Th_{\Pp(E)}(\Oc(1) \oplus 1) \xrightarrow{h_E} \Th_{\Gr(2,E\oplus 1)}(\Uni_2^\vee) \xrightarrow{\pf{g_E}} \Th_{\Gr(2,E)}(\Uni_2^\vee \oplus \Uni_2^\vee) \to \Sigma^{1,0} \Th_{\Pp(E)}(\Oc(1) \oplus 1).
\]
Since $\Th_{\Pp(E)}(\Oc(1) \oplus 1) =0$ in $\SH(S)[\eta^{-1}]$  by \dref{prop:PE}{prop:PE:even}, we deduce \eqref{prop:Gr+1:1}.

We now prove \eqref{prop:Gr+1:3}. Since the projection $\Pp(E \oplus 1) \to S$ induces an isomorphism in $\SH(S)[\eta^{-1}]$ by \dref{prop:PE}{prop:PE:odd}, so does its section $s\colon S=\Pp(1) \to \Pp(E \oplus 1)$. The closed immersion $j \colon S = \Gr(2,1^{\oplus 2}) \to \Gr(2,E \oplus 1^{\oplus 2})$ factors as $j = h_{E \oplus 1} \circ s$. Therefore by \rref{p:Gr_n-1_n_n}, we have a distinguished triangle in $\SH(S)[\eta^{-1}]$
\[
\Un_S \xrightarrow{j} \Sup \Gr(2,E\oplus 1^{\oplus 2}) \xrightarrow{\pf{g_{E\oplus 1}}} \Th_{\Gr(2,E\oplus 1)}(\Uni_2^\vee) \to \Sigma^{1,0} \Un_S.
\]
Since $q \circ j=\id_S$, this triangle splits, and we deduce \eqref{prop:Gr+1:3}
\end{proof}

\begin{para}
\label{p:E_F}
Let $E$ be a vector bundle over $X \in \Sm_S$, and $F \subset E$ a subbundle. Then for every $n \in \Nn$ the closed immersion
\[
i \colon \Gr(n,F) \to \Gr(n,E) \quad ; \quad (U \subset F) \mapsto (U \subset E)
\]
is the zero-locus of a section of $\Hom(\Uni_2,q^*(E/F)) \simeq \Uni_2^\vee\otimes q^*(E/F)$ transverse to the zero-section (namely, the composite $\Uni_2 \subset q^*E \to q^*(E/F)$).
\end{para}

\begin{lemma}
\label{lemm:E+D}
Let $E,D$ be vector bundles over $S$. Assume that $D$ has rank two and that $E$ has constant even rank. Consider the closed immersion $i\colon \Gr(2,E) \to \Gr(2,E\oplus D)$ induced by the inclusion $E \subset E \oplus D$ (see \rref{p:E_F}), and the projection $q \colon \Gr(2,E \oplus D) \to S$. Then we have an isomorphism in $\SH(S)[\eta^{-1}]$ 
\[
(q,\pf{i}) \colon \Sup \Gr(2,E \oplus D) \xrightarrow{\sim} \Un_S \vo \Th_{\Gr(2,E)}(\Uni_2^\vee\otimes q^*D).
\]
\end{lemma}
\begin{proof}
The statement is local in the Zariski topology of $S$, so we may assume that the vector bundle $D$ is trivial. Then $i$ factors as
\[
i \colon \Gr(2,E) \xrightarrow{g_E} \Gr(2,E\oplus 1) \xrightarrow{g_{E \oplus 1}} \Gr(2,E \oplus 1^{\oplus 2}).
\]
In view of \rref{p:pf_composite}, the statement follows by combining \eqref{prop:Gr+1:1} with \eqref{prop:Gr+1:3}.
\end{proof}

\begin{remark}
Taking $E=1^{\oplus 2}$ and $D=1^{\oplus 2}$ in \rref{lemm:E+D} yields a natural isomorphism
\[
\Sup \Gr(2,4) \simeq \Un_S \oplus \Sigma^{8,4} \Un_S \in \SH(S)[\eta^{-1}].
\]
\end{remark}

\subsection{Cohomology of \texorpdfstring{$2$}{2}-grassmannians}

\begin{para}
\label{p:i_j}
Let $A\in \SH(S)$ be an $\eta$-periodic ring spectrum. Let $X\in \Sm_S$ and $E,D\to X$ vector bundles, with $D$ of rank two. Denote by $q \colon \Gr(2,E\oplus D) \to X$ the projection. The inclusions $E \subset E \oplus D$ and $D \subset E \oplus D$ yield closed immersions (see \rref{p:E_F})
\[
i \colon \Gr(2,E) \to \Gr(2,E \oplus D) \quad \text{and} \quad j \colon X=\Gr(2,D) \to \Gr(2,E \oplus D).
\]
Since $i$ factors through the open complement of $j$, it follows from \rref{p:pf_empty} that $j^* \circ i_* \colon \Aa(\Gr(2,E);\Uni_2^\vee \otimes q^*D) \to \Aa(X)$ vanishes. As $q\circ j =\id$, Lemma~\rref{lemm:E+D} yields a split exact sequence of $\Aa(X)$-modules
\[
0\to \Aa(\Gr(2,E);\Uni_2^\vee \otimes q^*D) \xrightarrow{i_*} \Aa(\Gr(2,E \oplus D)) \xrightarrow{j^*} \Aa(X) \to 0.
\]
\end{para}

\begin{para}
\label{p:1_oplus_2}
Let $A \in \SH(S)$ be an $\eta$-periodic ring spectrum with a weak hyperbolic orientation (see \rref{def:hyp_thom}). Consider the situation of \rref{p:i_j}, and assume that $D= 1^{\oplus 2}$. By \rref{p:E_F} the closed immersion $i \colon \Gr(2,E) \to \Gr(2,E \oplus 1^{\oplus 2})$ is the zero-locus of a section of $\Uni_2^\vee \otimes (1^{\oplus 2})=\Uni_2^\vee \oplus \Uni_2^\vee$ transverse to the zero-section, hence by \dref{lemm:e_pi_push_pull}{lemm:e_pi_push_pull:2} we have
\begin{equation}
\label{eq:i_hypt_pi}
i_*(\hypt_{\Uni_2^\vee}) = \pi(\Uni_2^\vee) \in \Aa(\Gr(2,E \oplus 1^{\oplus 2})).
\end{equation}
Together with the projection formula \dref{lemm:proj_formula}{lemm:proj_formula:2}, this implies that, for any $k \in \Nn$
\begin{equation}
\label{eq:i_hypt_pik}
i_*(\pi(\Uni_2^\vee)^k \cup \hypt_{\Uni_2^\vee}) = \pi(\Uni_2^\vee)^{k+1} \in \Aa(\Gr(2,E \oplus 1^{\oplus 2})).
\end{equation}
\end{para}

\begin{proposition}
\label{prop:A_Gr_2}
Let $A \in \SH(S)$ be an $\eta$-periodic ring spectrum with a weak hyperbolic orientation (see \rref{def:hyp_thom}). Let $d\in \Nn$ and $s \in \{2d,2d+1\}$. Sending $u$ to the top Pontryagin class $\pi(\Uni_2^\vee)$ (see \rref{def:Euler_thom}) yields an isomorphism of $\Aa(S)$-algebras
\[
\Aa(\Gr(2,s)) \simeq \Aa(S)[u]/u^d.
\]

In addition, when $s$ is odd the (left or right) $\Aa(\Gr(2,s))$-module $\Aa(\Gr(2,s);\Uni_2^\vee)$ is freely generated by the Euler class $e(\Uni_2^\vee)$ (see \rref{def:Euler_thom}).
\end{proposition}
\begin{proof}
\emph{Case $s=2d$:} The case $s=0$, being clear, we may assume that $s\geq 2$. Consider the closed immersion $i\colon \Gr(2,s-2) \to \Gr(2,s)$ given by the vanishing of the last two coordinates. We are in the situation of \rref{p:i_j} with $E=1^{\oplus s}$ and $D=1^{\oplus 2}$, so that, taking into account  \rref{lemm:whyp_isom}, we obtain an exact sequence of $\Aa(S)$-modules
\[
0 \to \Aa(\Gr(2,s-2)) \xrightarrow{\alpha} \Aa(\Gr(2,s)) \xrightarrow{j^*} \Aa(S) \to 0,
\]
where $\alpha$ is given by $x \mapsto i_*(x \cup \hypt_{\Uni_2^\vee})$, and $j \colon S = \Gr(2,2) \to \Gr(2,s)$ is given by the vanishing of the first $s-2$ coordinates. Note that $\alpha(\pi(\Uni_2^\vee)^k)=\pi(\Uni_2^\vee)^{k+1}$ by \eqref{eq:i_hypt_pik}. Since $\pi(\Uni_2^\vee)^{d-1}=0$ in $\Aa(\Gr(2,s-2))$ by induction, it follows that $\pi(\Uni_2^\vee)^d=0$ in $\Aa(\Gr(2,s))$. Moreover $j^*\pi(\Uni_2^\vee) = j^* \circ \alpha(1) =0$. We thus obtain a commutative diagram with exact rows
\[ \xymatrix{
0 \ar[r] &\Aa(S)[u]/u^{d-1}\ar[r]^{\cdot u} \ar[d] & \Aa(S)[u]/u^d \ar[d] \ar[rr]^-{u \mapsto 0} && \Aa(S) \ar[d]^= \ar[r] &0 \\ 
0 \ar[r] & \Aa(\Gr(2,s-2)) \ar[r]^-{\alpha} & \Aa(\Gr(2,s)) \ar[rr]^-{j^*} && \Aa(S) \ar[r] & 0
}\]
Since the left vertical arrow is an isomorphism by induction, it follows that the middle vertical arrow is one, concluding the proof in this case.

\emph{Case $s=2d+1$:} Since by \rref{prop:Gr+1:2} (with $E=1^{\oplus 2d}$) we have an isomorphism of $\Aa(S)$-algebras $\Aa(\Gr(2,2d+1)) \simeq \Aa(\Gr(2,2d))$ mapping $\pi(\Uni_2^\vee)$ to $\pi(\Uni_2^\vee)$, the isomorphism $\Aa(\Gr(2,2d+1)) \simeq \Aa(S)[u]/u^d$ follows from the case $s=2d$ above. To prove the remaining statement, note that the composite
\begin{align*}
\Aa(\Gr(2,2d+1)) &\xrightarrow{g_{2d}^*} \Aa(\Gr(2,2d))\xrightarrow{\cup \hypt_{\Uni_2^\vee}} \\
&\Aa(\Gr(2,2d);\Uni_2^\vee \oplus \Uni_2^\vee) 
\xrightarrow{{g_{2d}}_*} \Aa(\Gr(2,2d+1);\Uni_2^\vee)
\end{align*}
is bijective by \eqref{prop:Gr+1:2}, \rref{lemm:whyp_isom} and \eqref{prop:Gr+1:1}. By \dref{lemm:e_pi_push_pull}{lemm:e_pi_push_pull:2} that composite is multiplication by the Euler class $e(\Uni_2^\vee)$.
\end{proof}

\begin{remark}
\label{rem:X_S}
Let $X \in \Sm_S$ with structural morphism $f\colon X \to S$. Since the image under $f^* \colon \SH(S) \to \SH(X)$ of an $\eta$-periodic ring spectrum with a weak hyperbolic orientation remains one, it follows that \rref{prop:A_Gr_2} provides a computation of $\Aa(\Gr(2,s) \times_S X)$. 
\end{remark}

\begin{corollary}
\label{cor:splitting_twisted}
Let $A \in \SH(S)$ be an $\eta$-periodic ring spectrum with a weak hyperbolic orientation. Let $d\in \Nn\smallsetminus\{0\}$. Let $X \in \Sm_S$, and $E,V \to X$ vector bundles, with $E$ of constant rank $s\in\{2d,2d+1\}$. Denote by $q \colon \Gr(2,E) \to X$ the projection.
\begin{enumerate}[(i)]
\item 
\label{cor:splitting_twisted:1}
We have an isomorphism of $\Aa(X)$-bimodules
\[
\varphi_{X,V} \colon \Aa(X;V)^{\oplus d} \xrightarrow{\sim} \Aa(\Gr(2,E);q^*V), \quad (a_0, \dots ,a_{d-1}) \mapsto \sum_{i=0}^{d-1} q^*(a_i)\pi(\Uni_2^\vee)^i.
\]

\item 
\label{cor:splitting_twisted:2}
If $s=2d+1$, we have an isomorphism of $\Aa(\Gr(2,E))$-bimodules
\[
\Aa(\Gr(2,E);p^*V) \xrightarrow{\sim} \Aa(\Gr(2,E);p^*V \oplus \Uni_2^\vee), \quad x\mapsto x\cup e(\Uni_2^\vee).
\]
\end{enumerate}
\end{corollary}
\begin{proof}
Using a Mayer--Vietoris argument, we may assume that both $E$ and $V$ are trivial. Using the isomorphism \eqref{eq:twist_shift} we reduce to the case $V=0$. Then the statements follow from \rref{prop:A_Gr_2} (in view of \rref{rem:X_S}).
\end{proof}

\section{Pontryagin classes}

\subsection{Pontryagin structures}
\label{sect:Pontryagin_structures}
In this section, we introduce a structure which will turn out to be equivalent to that of a weak hyperbolic orientation, where instead of Thom classes of rank two vector bundles, we specify their top Pontryagin classes.

\begin{para}
\label{p:point_Gr_2_4}
Let us denote by $i \colon S= \Gr(2,2) \to \Gr(2,4)$ the closed immersion induced by the inclusion $1^{\oplus 2} \subset 1^{\oplus 4}$ given by the vanishing of the last two coordinates. Then $i$ is the zero-locus of a section of $\Uni_2^\vee \oplus \Uni_2^\vee$ transverse to the zero-section, namely $(s_3,s_4)$, where $s_k$ is the section of $\Uni_2^\vee$ corresponding to the $k$-th coordinate for $k \in \{3,4\}$ (i.e.\ dual to the composite $\Uni_2 \subset 1^{\oplus 4} \to 1$ where the last map is the $k$-th projection). This choice of a section permits to identify the normal bundle of $i$ with $i^*(\Uni_2^\vee \oplus \Uni_2^\vee)=1^{\oplus 4}$.
\end{para}

\begin{definition}
\label{def:Pontryagin_str}
Let $A \in \SH(S)$ be a ring spectrum. A \emph{Pontryagin structure} on $A$ is the datum of a class $\pi_E \in A^{8,4}(X)$ for each rank two vector bundle $E\to X$ with $X\in \Sm_S$, such that:
\begin{enumerate}[(i)]
\item 
the class $\pi_E$ is $\Aa(X)$-central,

\item 
\label{def:Pontryagin_str:pb}
if $f \colon Y \to X$ is a morphism in $\Sm_S$ and $E \to X$ a vector bundle of rank two, then $f^*\pi_E = \pi_{f^*E}$,

\item
\label{def:Pontryagin_str:isom}
if $\varphi \colon E \xrightarrow{\sim} F$ is an isomorphism of rank two vector bundles over $X$, then $\pi_F = \pi_E$,

\item
\label{def:Pontryagin_str:2,4}
in the notation of \rref{p:point_Gr_2_4} and \rref{p:sw}, we have $\pi_{\Uni_2^\vee} = i_*\circ \sw_{1,1}^*(\Sigma^{8,4}1)$.
\end{enumerate}
\end{definition}

\begin{para}
\label{p:induced_Pontryagin}
It follows from \dref{lemm:e_pi_push_pull}{lemm:e_pi_push_pull:2} that a weak hyperbolic orientation (see \rref{def:hyp_thom}) induces a Pontryagin structure, by setting $\pi_E = \pi(E)$ for any rank two vector bundle $E\to X$ with $X\in \Sm_S$.
\end{para}

\begin{lemma}
\label{lemm:Pontryagin_str:1}
Let $A \in \SH(S)$ be a ring spectrum with a Pontryagin structure. Then $\pi_{1^{\oplus 2}} = 0 \in A^{8,4}(S)$.
\end{lemma}
\begin{proof}
Let $j \colon S=\Gr(2,2) \to \Gr(2,4)$ be the immersion given by the vanishing of the first two coordinates. Then $j^*\Uni_2^\vee \simeq 1^{\oplus 2}$, and $j$ factors trough the open complement of the immersion $i$ of \rref{p:point_Gr_2_4}, hence by \rref{p:pf_empty}
\[
0=j^* \circ i_*\circ \sw_{1,1}^*(\Sigma^{8,4}1) 
\overset{\text{\dref{def:Pontryagin_str}{def:Pontryagin_str:2,4}}}{=}
i^*\pi_{\Uni_2^\vee}
\overset{\text{\dref{def:Pontryagin_str}{def:Pontryagin_str:pb}}}{=}
\pi_{i^*\Uni_2^\vee}
\overset{\text{\dref{def:Pontryagin_str}{def:Pontryagin_str:isom}}}{=} \pi_{1^{\oplus 2}}.\qedhere
\]
\end{proof}

\begin{proposition}
\label{prop:Pontryagin_str-wo}
Let $A \in \SH(S)$ be an $\eta$-periodic ring spectrum. Then each Pontryagin structure on $A$ is induced by a unique weak hyperbolic orientation (in the sense of \rref{p:induced_Pontryagin}).
\end{proposition}
\begin{proof}
Let $E \to X$ be a rank two vector bundle, with $X \in \Sm_S$. Then $\Uni_2^\vee$ corresponds to $E$ under the identification $X=\Gr(2,E^\vee)$. Thus by \rref{p:i_j} (with $D=1^{\oplus 2}$ and $E$ replaced with $E^\vee$) we have an exact sequence
\[
0\to \Aa(X;E \oplus E) \xrightarrow{i_*} \Aa(\Gr(2,E^\vee \oplus 1^{\oplus 2})) \xrightarrow{j^*} \Aa(X) \to 0.
\]
Recall from \eqref{eq:i_hypt_pi} that if $A$ is endowed with a weak hyperbolic orientation, we have
\begin{equation}
\label{eq:i_hyp_E}
i_*(\hypt_E) = i_*(\hypt_{\Uni_2^\vee}) = \pi(\Uni_2^\vee).
\end{equation}
In view of the injectivity $i_*$, the uniqueness part of the statement follows.

Assume now that $A$ is endowed with a Pontryagin structure. As $j^*\Uni_2=1^{\oplus 2}$, the element $\pi_{\Uni_2^\vee} \in A^{8,4}(\Gr(2,E\oplus 1^{\oplus 2}))$ satisfies $j^* \pi_{\Uni_2^\vee} = \pi_{1^{\oplus 2}}$, which vanishes by \rref{lemm:Pontryagin_str:1}. Therefore, by the above exact sequence, there exists a unique element $\hypt_E \in A^{0,0}(X;E \oplus E)$ such that $i_*(\hypt_E) = \pi_{\Uni_2^\vee}$. It is then easy to verify that the association $E \mapsto \hypt_E$ defines a weak hyperbolic orientation of $A$. Now, as $E =i^*\Uni_2^\vee$, we have
\[
\pi(E)= i^*\pi(\Uni_2^\vee) \overset{\text{\eqref{eq:i_hyp_E}}}{=} i^* \circ i_*(\hypt_E) = i^*\pi_{\Uni^\vee_2} = \pi_E,
\]
so that this weak hyperbolic orientation induces the original Pontryagin structure.
\end{proof}

\subsection{The splitting principle}
\begin{theorem}
[Splitting principle]
\label{th:splitting}
Let $A \in \SH(S)$ be an $\eta$-periodic hyperbolically oriented ring spectrum. Let $X \in \Sm_S$, and $E \to X$ a vector bundle of constant rank. Then there exists a morphism $f \colon Y \to X$ in $\Sm_S$ and vector bundles $E_1,\dots,E_r$ over $Y$ such that:
\begin{enumerate}[(i)]
\item for any vector bundle $V\to X$, the morphism of $\Aa(S)$-bimodules $f^* \colon \Aa(X;V) \to \Aa(Y;V)$ admits a retraction,

\item $f^*E \simeq E_1 \oplus \dots \oplus E_r$,

\item $\rank E_i \in \{1,2\}$ and $\det E_i \simeq 1$ for $i=1,\dots,r$.
\end{enumerate}
\end{theorem}
\begin{proof}
We proceed by induction on the rank of $E$, the statement being clear when $E=0$. If $\rank E=1$, this follows from \rref{prop:splitting_odd}. Assume that $\rank E \geq 2$, and consider the morphism $q \colon \Gr(2,E) \to X$. Then $q^* \colon \Aa(X;V) \to \Aa(\Gr(2,E);V)$ is a split injection by \rref{cor:splitting_twisted}, and $q^*E$ admits a rank two subbundle $\Uni_2$. Consider the vector bundle $\Quo_2=p^*E/\Uni_2$. By \rref{p:Jouanolou} we find a morphism $g\colon T \to \Gr(2,E)$ such that $g^*q^*E \simeq g^*\Uni_2 \oplus g^*\Quo_2$ and $g^* \colon \Aa(\Gr(2,E);V) \to \Aa(T;V)$ is bijective. By \rref{prop:splitting_odd}, we find a morphism $h \colon Z \to T$, such that the vector bundle $h^*g^*\Uni_2$ has trivial determinant, and $h^* \colon  \Aa(T;V) \to \Aa(Z;V)$ is split injective. We conclude by applying the inductive hypothesis to the vector bundle $h^*g^*\Quo_2$, whose rank is $\rank E -2$.
\end{proof}

As a corollary, we obtain an $\SL$-oriented splitting principle, generalising \cite[\S9]{Ananyevskiy-SL_PB}. (Recall from \rref{ex:Sp_hyp} that a commutative ring spectrum in $\SH(S)$ equipped with a normalised $\Sp$-orientation in the sense of \cite[Definition~3.3]{Ana-SL}, and a fortiori one equipped with a normalised $\SL$-orientation, inherits a hyperbolic orientation.)

\begin{corollary}
\label{th:SL_splitting}
Let $A \in \SH(S)$ be an $\eta$-periodic hyperbolically oriented ring spectrum. Let $X \in \Sm_S$, and $E \to X$ a $\SL$-oriented vector bundle of constant rank (see e.g.\ \cite[Definition~2.2]{Ana-SL}). Then there exists a morphism $f \colon Y \to X$ in $\Sm_S$ such that
\begin{enumerate}[(i)]
\item for any vector bundle $V\to X$, the morphism of $\Aa(S)$-bimodules $f^* \colon \Aa(X;V) \to \Aa(Y;V)$ admits a retraction,

\item the $\SL$-oriented vector bundle $f^*E$ splits as a direct sum of $\SL$-oriented vector bundles of ranks one or two.
\end{enumerate}
\end{corollary}
\begin{proof}
Denote by $\lambda \colon 1 \xrightarrow{\sim} \det E$ the isomorphism giving the $\SL$-orientation of $E$. Let us pick $f\colon Y \to X$ and $E_1,\dots,E_r$ as in \rref{th:splitting}, together with isomorphisms $\lambda_i \colon 1 \xrightarrow{\sim} \det E_i$ for $i=1,\dots,r-1$. We may assume that $E\neq 0$, so that $r\geq 1$. Letting
\[
\lambda_r \colon 1 \xrightarrow{ \lambda \otimes (\lambda_1^{-1} \otimes \dots \otimes \lambda_{r-1}^{-1})^\vee } \det E \otimes (\det E_1 \otimes \dots \otimes \det E_{r-1})^\vee \simeq \det E_r,
\]
we have, as $\SL$-oriented vector bundles
\[
(E,\lambda) = (E_1,\lambda_1) \oplus \dots \oplus (E_r,\lambda_r).\qedhere
\]
\end{proof}

For the next statement, we use the terminology of \cite[Definition~3.3]{Ana-SL}.
\begin{corollary}
Let $A \in \SH(S)$ be an $\eta$-periodic commutative ring spectrum. Then each $\Sp$-orientation of $A$ is induced by at most one normalised $\SL$-orientation.
\end{corollary}
\begin{proof}
Let us assume that $A$ is endowed with a normalised $\SL$-orientation. Then $A$ carries an induced $\Sp$-orientation, and thus by \rref{ex:Sp_hyp} an induced hyperbolic orientation. In view of \rref{th:SL_splitting}, it will suffice to prove that the Thom class (for the $\SL$-orientation) of an $\SL$-oriented vector bundle $E \to X$ is determined by the induced $\Sp$-orientation of $A$, when $E$ has rank one or two. 

If $E$ has rank one, it is isomorphic to the trivial $\SL$-oriented line bundle, hence its Thom class in $A^{0,0}(X;E)$ must be the image of $\Sigma^{2,1}1 \in A^{0,0}(X;1)$, as the $\SL$-orientation is normalised by assumption. If $E$ has rank two, as $\SL_2=\Sp_2$, the vector bundle $E$ is $\Sp$-oriented (explicitly the symplectic form is given by the morphism $E \otimes E \to \Lambda^2E =\det E \simeq 1$), and so its Thom class is determined by the induced $\Sp$-orientation.
\end{proof}

\subsection{Higher Pontryagin classes}
Throughout this section $A \in \SH(S)$ will be an $\eta$-periodic ring spectrum with a weak hyperbolic orientation.

\begin{definition}
\label{def:Pontryagin}
Let $X \in \Sm_S$, and $E \to X$ be a vector bundle of constant rank $2d$ or $2d+1$, with $d\in \Nn$. Consider the Grassmann bundle $q_E \colon \Gr(2,E^\vee) \to X$. We define the \emph{Pontryagin classes}
\[
p_k(E) \in A^{8k,4k}(X) \quad \text{ for $k\in \Zz$}
\]
by the formulas
\[
p_0(E) = 1 \quad \text{ and } \quad p_k(E) =0 \text{ when $k \notin \{0,\dots,d\}$},
\]
and, in view of \rref{cor:splitting_twisted},
\begin{equation}
\label{eq:def_Pontryagin}
\sum_{k=0}^d (-1)^kq_E^*(p_{d-k}(E)) \pi(\Uni_2^\vee)^k=0 \in A^{8d,4d}(\Gr(2,E^\vee)).
\end{equation}
This definition extends in an obvious way to the case when the rank of $E$ is not constant.
\end{definition}

\begin{remark}
In view of \rref{lemm:pi_dual}, we may replace $\pi(\Uni_2^\vee)$ with $\pi(\Uni_2)$ in the formula \eqref{eq:def_Pontryagin}.
\end{remark}

\begin{para}
The Pontryagin classes $p_k(E)$ are functorial in $X \in \Sm_S$, central in $\Aa(X)$, and depend only on the isomorphism class of the vector bundle $E \to X$.
\end{para}

\begin{para}
Assume that $d=1$ and so $E$ has rank two. Then the morphism $q_E \colon \Gr(2,E^\vee) \to X$ is an isomorphism, and $\Uni_2^\vee =q_E^*E$. Therefore \eqref{eq:def_Pontryagin} implies that
\begin{equation}
\label{eq:Pontrygin_pi_2}
p_1(E) = \pi(E) \in A^{8,4}(X).
\end{equation}
\end{para}

\begin{lemma}
\label{lemm:pontryagin_no_line}
Let $X \in \Sm_S$. Let $E \to X$ be a vector bundle, and $F$ the quotient of $E$ by a subbundle of rank one. Then
\[
p_k(E) = p_k(F) \quad \text{for any $k \in \Zz$}.
\]
\end{lemma}
\begin{proof}
We may assume that $E$ has constant rank $r$. If $r=2d+1$ is odd, pulling back the relation \eqref{eq:def_Pontryagin} along the morphism $\Gr(2,F^\vee) \to \Gr(2,E^\vee)$ induced by the inclusion $F^\vee \subset E^\vee$ yields
\[
0=\sum_{k=0}^d (-1)^k q_F^*(p_{d-k}(E)) \pi(\Uni_2^\vee)^k \in A^{8d,4d}(\Gr(2,F^\vee)),
\]
which proves the lemma, in view of the definition of the Pontryagin classes of $F$.

Assume now that $r=2d$ is even. By \rref{p:Jouanolou}, we may assume that $E=F \oplus L$ with $L \to X$ a line bundle. By \rref{prop:splitting_odd} we may assume that $L=1$ and that $F=G \oplus 1$ for some vector bundle $G\to X$. Consider the closed immersion $i \colon \Gr(2,G^\vee) \to \Gr(2,E^\vee)$ induced by the inclusion $G^\vee \subset G^\vee \oplus 1^{\oplus 2} = E^\vee$. Since $G$ has rank $2(d-1)$, the defining relation \eqref{eq:def_Pontryagin} reads
\[
0=\sum_{k=0}^{d-1} (-1)^k q_G^*(p_{d-1-k}(G)) \pi(\Uni_2^\vee)^{k} \in A^{8(d-1),4(d-1)}(\Gr(2,G^\vee)).
\]
Applying the map $x \mapsto -i_*(x \cup \hypt_{\Uni_2^\vee})$ to this relation yields in $A^{8d,4d}(\Gr(2,E^\vee))$
\begin{align*}
0 
&= -\sum_{k=0}^{d-1} (-1)^k  q_E^*(p_{d-1-k}(G)) i_*(\pi(\Uni_2^\vee)^k \cup \hypt_{\Uni_2^\vee}) && \text{by \dref{lemm:proj_formula}{lemm:proj_formula:2}, as $q_G=q_E \circ i$}\\
&=  \sum_{k=0}^{d-1} (-1)^{k+1}  q_E^*(p_{d-1-k}(G)) \pi(\Uni_2^\vee)^{k+1} && \text{by \eqref{eq:i_hypt_pik}}\\
&= \sum_{k=0}^d (-1)^k  q_E^*(p_{d-k}(G)) \pi(\Uni_2^\vee)^k && \text{as $p_d(G)=0$ by definition.}
\end{align*}
It thus follows that $p_k(G) = p_k(E)$ for all $k$. Since $p_k(G) = p_k(G \oplus 1) = p_k(F)$ by the odd rank case considered above, this concludes the proof.
\end{proof}

\begin{remark}
\label{rem:pontryagin_trivial}
It follows from \rref{lemm:pontryagin_no_line} that $p_k(1^{\oplus r})=0$ for $k \neq 0$ and $r \in \Nn$.
\end{remark}

\begin{lemma}
\label{lemm:prod_vanish}
Let $U_1,\dots,U_d$ be rank two vector bundles over $X \in \Sm_S$, and set $E=U_1 \oplus \dots \oplus U_d$. Consider the Grassmann bundle $q_E \colon \Gr(2,E^\vee) \to X$ and its universal rank two vector bundle $\Uni_2$. Then
\[
\prod_{k=1}^d \Big(\pi(\Uni_2^\vee) - q_E^*\pi(U_k)\Big) =0 \in A^{8d,4d}(\Gr(2,E^\vee)).
\]
\end{lemma}
\begin{proof}
We proceed by induction on $d$, the case $d=0$ being clear. Assume that $d \geq 1$. Set $F=U_1 \oplus \dots \oplus U_{d-1}$. By \rref{p:i_j} we have an exact sequence
\[
\Aa(\Gr(2,F^\vee);\Uni_2^\vee \otimes q_F^*U_d^\vee) \xrightarrow{i_*} \Aa(\Gr(2,E^\vee)) \xrightarrow{j^*} \Aa(X),
\]
where $i \colon \Gr(2,F^\vee) \to \Gr(2,E^\vee)$ and $j\colon X=\Gr(2,U_d^\vee) \to \Gr(2,E^\vee)$ are the natural closed immersions. Since $j^*\Uni_2 = U_d^\vee$, the element $\pi(\Uni_2^\vee) - q_E^*\pi(U_d) \in A^{8,4}(\Gr(2,E^\vee))$ belongs to the kernel of $j^*$. Thus there exists an element $x \in A^{0,0}(\Gr(2,F^\vee);\Uni_2^\vee \otimes q_F^*U_d^\vee)$ such that $i_*(x) = \pi(\Uni_2^\vee) - q_E^*\pi(U_d)$. Then we have in $A^{8d,4d}(\Gr(2,E^\vee))$
\begin{align*}
\prod_{k=1}^d \big(\pi(\Uni_2^\vee)-q_E^*\pi(U_k)\big)
&= \Big(\prod_{k=1}^{d-1} \big(\pi(\Uni_2^\vee) - q_E^*\pi(U_k)\big)\Big) \cup i_*(x) \\ 
&= i_*\Big(\prod_{k=1}^{d-1} \big(\pi(\Uni_2^\vee) - q_F^*\pi(U_k)\big)\cup x\Big),
\end{align*}
by the projection formula \dref{lemm:proj_formula}{lemm:proj_formula:2}. This element vanishes by induction.
\end{proof}

\begin{theorem}
\label{prop:pont_sym}
Let $X \in \Sm_S$. Let $U_1,\dots,U_d$ be rank two vector bundles over $X$, and $L_1,\dots,L_r$ line bundles over $X$. Then for any $k \in \Nn$, we have $A^{8k,4k}(X)$
\[
p_k(U_1 \oplus \dots \oplus U_d \oplus L_1 \oplus \dots \oplus L_r) = \sigma_k(\pi(U_1),\dots,\pi(U_d)).
\]
where $\sigma_k$ denotes the $k$-th elementary symmetric polynomial in $d$ variables.
\end{theorem}
\begin{proof}
By \rref{lemm:pontryagin_no_line}, we may assume that $r=0$. Then the theorem follows by expanding the product in \rref{lemm:prod_vanish} and comparing with the defining relation \eqref{eq:def_Pontryagin}.
\end{proof}

\begin{corollary}[Whitney sum formula]
\label{prop:Pontryagin_sum}
If $0 \to E' \to E \to E'' \to 0$ is an exact sequence of vector bundles over $X \in \Sm_S$, for any $k \in \Zz$, we have in $A^{8k,4k}(X)$
\[
p_k(E) = \sum_{i=0}^k p_i(E') p_{k-i}(E'').
\]
\end{corollary}
\begin{proof}
By \rref{p:Jouanolou}, we may assume that the sequence splits. By the splitting principle \rref{th:splitting} we may further assume that $E'$ and $E''$ split as direct sums of vector bundles of ranks $1$ or $2$. Then the corollary follows from \rref{prop:pont_sym}.
\end{proof}

\begin{corollary}
\label{cor:Pontryagin_dual}
If $X \in \Sm_S$ and $E\to X$ is a vector bundle, we have
\[
p_k(E) = p_k(E^\vee) \quad \text{ for any $k \in \Zz$}.
\]
\end{corollary}
\begin{proof}
This follows by combining \rref{th:splitting}, \rref{prop:pont_sym} and \rref{lemm:pi_dual}.
\end{proof}

\begin{corollary}
\label{lemm:Pontryagin_pi}
Assume the weak hyperbolic orientation of $A$ extends to a hyperbolic orientation. Let $X \in \Sm_S$, and $E \to X$ be a vector bundle of constant rank $2d$, with $d\in \Nn$. Then $p_d(E) = \pi(E)$ in $A^{8d,4d}(X)$.
\end{corollary}
\begin{proof}
This follows by combining \rref{th:splitting}, \rref{prop:pont_sym} and \dref{prop:Euler}{prop:Euler:mult}.
\end{proof}

\begin{remark}
We will see in \rref{th:Thom_orientation} that the assumption of \rref{lemm:Pontryagin_pi} is automatically fulfilled.
\end{remark}

\begin{remark}
Let $A \in \SH(S)$ be an $\eta$-periodic commutative ring spectrum equipped with a normalised $\Sp$-orientation, in the sense of \cite[Definition~3.3]{Ana-SL}. Recall from \rref{ex:Sp_hyp} that $A$ inherits a hyperbolic orientation. Then we claim that the Pontryagin classes $p_i$, defined in \rref{def:Pontryagin} from the hyperbolic orientation, coincide with the Pontryagin classes $p_i'$, defined in \cite[Definition~7.7]{Ana-SL} from the $\Sp$-orientation (and denoted by $p_i$ there). 

Indeed by the splitting principle \rref{th:splitting} and the Whitney sum formulas (see \rref{prop:Pontryagin_sum} and \cite[Corollary~7.9~(2)]{Ana-SL}), it will suffice to show that $p_i(E)=p_i'(E)$ when $E$ is a vector bundle of rank $r \in \{1,2\}$. Note that for $i>r/2$ we have $p_i(E)=0$ by definition \rref{def:Pontryagin}, and $p_i'(E)=0$ by \cite[Corollary~7.9~(3)]{Ana-SL}. It will thus suffice to assume that $r=2$ and show that $p_1(E)=p_1'(E)$. But, using the notation of \rref{ex:Sp_hyp}, we have
\[
p_1(E)\overset{\eqref{eq:Pontrygin_pi_2}}{=}\pi(E) \overset{\rref{def:Euler:hyp}}{=} z_{E\oplus E}^*(\hypo_E) \overset{\eqref{eq:hyp_Sp_eta}}{=} z_{E\oplus E}^* \circ \sigma_E^*(\thom_{H(E)}) \overset{\eqref{eq:diag_Thom_dual}}{=} z_{E \oplus E^\vee}^*(\thom_{H(E)}),
\]
which equals $b_2(H(E)) = p_1'(E)$ by \cite[(13.3)]{PW-QGrass}, as required.
\end{remark}

\subsection{Orientations and weak orientations}
\label{sect:o_wo}
In this section we prove that, for an $\eta$-periodic motivic ring spectrum, the datum of a hyperbolic orientation (see \rref{def:hyp_orientation}) is equivalent to that of a weak hyperbolic orientation (see \rref{def:hyp_thom}). Throughout this section $A \in \SH(S)$ will be an $\eta$-periodic ring spectrum with a weak hyperbolic orientation.

\begin{para}
\label{p:f}
It will be convenient to introduce the polynomials, for $r \in \Zz$
\begin{equation}
\label{eq:def_f_r_E}
f_{r,E}= \sum_{k=0}^r (-1)^k p_{r-k}(E) u^k \in \Aa(X)[u],
\end{equation}
for any vector bundle $E \to X$ with $X \in \Sm_S$. Those satisfy the inductive formula
\begin{equation}
\label{eq:f_k_E_rec}
p_r(E) = f_{r,E} + u f_{r-1,E}.
\end{equation}
\end{para}

\begin{para}
\label{eq:f_k_+1}
Let $E \to X$ be a vector bundle, with $X \in \Sm_S$. It follows from \rref{lemm:pontryagin_no_line} that $f_{r,E} = f_{r,E\oplus 1}$ for any $r \in \Zz$.
\end{para}

\begin{para}
Consider an exact sequence $0 \to D \to E \to Q \to 0$ of vector bundles over $X \in \Sm_S$, where $D$ has rank two. Then by the Whitney sum formula \rref{prop:Pontryagin_sum} and \rref{eq:Pontrygin_pi_2} we have $p_k(E) = p_k(Q) + \pi(D) p_{k-1}(Q)$. Plugging this relation into \eqref{eq:def_f_r_E}
 yields, for any $r \in \Zz$
\begin{equation}
\label{eq:f_E_D}
f_{r,E} = f_{r,Q} + \pi(D) f_{r-1,Q} .
\end{equation}
In view of \eqref{eq:f_k_E_rec}, we deduce by induction on $r$ that, for any $r \in \Zz$ (the case $r<0$ being clear)
\begin{equation}
\label{eq:f_r_pi}
f_{r,E}(\pi(D)) = p_r(Q) \in A^{8r,4r}(X).
\end{equation}
\end{para}

\begin{para}
Let $X \in \Sm_S$, and $E \to X$ be a vector bundle. We will implicitly view $\Aa(\Gr(2,E))$ as an $\Aa(X)$-bimodule, via the pullback along the projection $\Gr(2,E) \to X$. By \eqref{eq:f_r_pi}, we have for any $r \in \Zz$
\begin{equation}
\label{eq:f_d_U_Q}
f_{r,E}(\pi(\Uni_2)) = p_r(\Quo_2) \in A^{8r,4r}(\Gr(2,E)).
\end{equation}
\end{para}

\begin{para}
\label{p:s_E}
Let $E \to X$ be a vector bundle of rank $2d$ over $X \in \Sm_S$, with $d\in \Nn$. The inclusion $1^{\oplus 2} \subset E \oplus 1^{\oplus 2}$ induces a closed immersion
\begin{equation}
\label{eq:s_E}
s_E \colon X=\Gr(2,1^{\oplus 2}) \to \Gr(2,E \oplus 1^{\oplus 2}).
\end{equation}
Observe that $s_E$ is defined by the vanishing of a section of $\Hom(1^{\oplus 2},\Quo_2)=\Quo_2 \oplus \Quo_2$ transverse to the zero-section, namely the composite $1^{\oplus 2} \subset q^*(E \oplus 1^{\oplus 2}) \to \Quo_2$, where $q \colon \Gr(2,E \oplus 1^{\oplus 2}) \to X$ is the projection. Since $s_E^*\Quo_2=E$, the closed immersion $s_E$ has normal bundle $E \oplus E$.

By \rref{p:Gr_n_n_n-1}, we have a long exact sequence
\[
\dots\to \Aa(\Gr(1,E \oplus 1);\Quo_1) \xrightarrow{{h_{E\oplus 1}}_*} \Aa(\Gr(2,E\oplus 1^{\oplus 2})) \xrightarrow{g_{E\oplus 1}^*} \Aa(\Gr(2,E \oplus 1)) \to\cdots
\]
It follows from \dref{cor:splitting_twisted}{cor:splitting_twisted:1} that in this sequence the map $g_{E\oplus 1}^*$ is surjective, hence ${h_{E\oplus 1}}_*$ is injective. Moreover $\Gr(1,E\oplus 1) = \Pp(E \oplus 1)$ with $\Quo_1$ corresponding to the quotient bundle $\Quo$  of \rref{eq:Gr_1_P}. Consider the closed immersion $f \colon X = \Pp(1) \to  \Pp(E \oplus 1) = \Gr(1,E\oplus 1)$ induced by the inclusion $1 \subset E \oplus 1$. Then the composite $s_E=h_{E \oplus 1} \circ f$. By \rref{lemm:PE_odd_Q} the pushforward $f_* \colon \Aa(X;E \oplus E) \to \Aa(\Pp(E\oplus 1);\Quo)$ is bijective. We thus obtain a short exact sequence of $\Aa(X)$-bimodules
\[
0 \to \Aa(X;E \oplus E) \xrightarrow{{s_E}_*} \Aa(\Gr(2,E\oplus 1^{\oplus 2})) \xrightarrow{g_{E \oplus 1}^*} \Aa(\Gr(2,E \oplus 1)) \to 0.
\]
Next, we have in $A^{8d,4d}(\Gr(2,E\oplus 1))$
\[
g_{E \oplus 1}^*(f_{d,E}(\pi(\Uni_2))) = f_{d,E}(\pi(\Uni_2)) \overset{\eqref{eq:f_k_+1}}{=} f_{d,E \oplus 1}(\pi(\Uni_2)) \overset{\eqref{eq:f_d_U_Q}}{=} p_d(\Quo_2),
\]
which vanishes since the vector bundle $\Quo_2 \to \Gr(2,E \oplus 1)$ has rank $2d-1$. Thus by the short exact sequence above, there exists a unique element
\begin{equation}
\label{eq:def_hypo}
\hypo_E \in A^{0,0}(X;E \oplus E) \quad \text{such that ${s_E}_*(\hypo_E) = f_{d,E}(\pi(\Uni_2))$}.
\end{equation}
\end{para}

\begin{lemma}
\label{lemm:o_t} 
Let $X \in \Sm_S$, and $E\to X$ a rank two vector bundle. Then
\[
\hypo_E = \hypt_E \in A^{0,0}(X;E \oplus E),
\]
where $\hypo_E$ is defined in \eqref{eq:def_hypo}, and $\hypt_E$ is given by the weak hyperbolic orientation of $A$.
\end{lemma}
\begin{proof}
Recall that the closed immersion $s_E \colon X \to \Gr(2,E \oplus 1^{\oplus 2})$ of \eqref{eq:s_E} is defined by the vanishing of a section of $\Quo_2 \oplus \Quo_2$ transverse to the zero-section. Therefore
\[
{s_E}_*(\hypt_E) \overset{\dref{lemm:e_pi_push_pull}{lemm:e_pi_push_pull:2}}{=} \pi(\Quo_2) \overset{\eqref{eq:Pontrygin_pi_2}}{=} p_1(\Quo_2) \overset{\eqref{eq:f_d_U_Q}}{=} f_{1,E}(\pi(\Uni_2)).\qedhere
\]
\end{proof}

\begin{lemma}
\label{lemm:mult_hypo_2}
Let $X \in \Sm_S$, and $E,D \to X$ vector bundles. Assume that $D$ has rank two, and that $E$ has constant even rank. Then the elements defined in \rref{eq:def_hypo} satisfy
\[
\hypo_{E \oplus D} = \sw_{E,D}^*(\hypo_E \cup \hypo_D).
\]
\end{lemma}
\begin{proof}
Let $q \colon \Gr(2,E \oplus D \oplus 1^{\oplus 2}) \to X$ be the projection. The inclusions $E \oplus 1^{\oplus 2} \subset E \oplus D \oplus 1^{\oplus 2}$ and $D \subset E \oplus D \oplus 1^{\oplus 2}$ induce closed immersions
\[
i\colon \Gr(2,E \oplus 1^{\oplus 2}) \to \Gr(2,E \oplus D \oplus 1^{\oplus 2}) \; \text{and} \; j \colon  X=\Gr(2,D) \to \Gr(2,E \oplus D \oplus 1^{\oplus 2}).
\]
As explained in \rref{p:i_j}, we have an exact sequence of $\Aa(X)$-bimodules
\[
0\to \Aa(\Gr(2,E \oplus  1^{\oplus 2});\Uni_2^\vee \otimes q^*D) \xrightarrow{i_*} \Aa(\Gr(2,E \oplus D\oplus 1^{\oplus 2})) \xrightarrow{j^*} \Aa(X) \to 0.
\]
Since $j^*\Uni_2=D$, there exists a unique element
\[
\omega \in A^{0,0}(\Gr(2,E \oplus 1^{\oplus 2});\Uni_2^\vee \otimes q^*D)
\]
such that
\begin{equation}
\label{eq:def_omega}
i_*(\omega) = q^*\pi(D)  - \pi(\Uni_2)\in A^{8,4}(\Gr(2,E \oplus D \oplus 1^{\oplus 2})).
\end{equation}

Consider now the cartesian square in $\Sm_S$
\[ \xymatrix{
X\ar[r]^{s_D} \ar[d]_{s_E} & \Gr(2,D \oplus 1^{\oplus 2}) \ar[d]^k \\ 
\Gr(2,E \oplus 1^{\oplus 2}) \ar[r]^-i & \Gr(2,E \oplus D \oplus 1^{\oplus 2})
}\]
where $s_D,s_E$ are defined in \eqref{eq:s_E}, and $k$ is induced by the inclusion $D \oplus 1^{\oplus 2} \subset E \oplus D \oplus 1^{\oplus 2}$. Set $q'=k \circ q$. The cartesian square above is transverse, hence by \rref{p:pf_funct} and \rref{p:transverse} we have in $A^{8,4}(\Gr(2,D \oplus 1^{\oplus 2}))$
\[
{s_D}_* \circ s_E^*(\omega) = k^* \circ i_*(\omega) = k^*(q^*\pi(D) - \pi(\Uni_2)) = q'^*\pi(D) - \pi(\Uni_2) = f_{1,D}(\pi(\Uni_2)),
\]
which by definition of $\hypo_D$ (see \eqref{eq:def_hypo}) implies that
\begin{equation}
\label{eq:omega_hypo}
s_E^*(\omega) = \hypo_D \in A^{0,0}(X;D \oplus D).
\end{equation}

Next, it follows from \rref{p:pf_composite} that we have a commutative square in $\Hop(S)$
\begin{equation}
\label{eq:s_E_D_sw}
\begin{gathered}
\xymatrix{
\Th_X(E \oplus D \oplus E \oplus D)  \ar[d]_{\sw_{E,D}}&& \Gr(2,E \oplus D \oplus 1^{\oplus 2})\ar[d]^{\pf{i}} \ar[ll]_-{\pf{s_{E\oplus D}}} \\ 
\Th_X(E \oplus E \oplus D \oplus D) && \Th_{\Gr(2,E \oplus 1^{\oplus 2})}(D \oplus D) \ar[ll]_-{\pf{s_E}} 
}
\end{gathered}
\end{equation}
Let now $d\in \Nn$ be such that $E$ has rank $2d$. Then, viewing $\Aa(\Gr(2,D \oplus E \oplus 1^{\oplus 2}))$ as an $\Aa(X)$-algebra via $q^*$, we have

\begin{align*}
&(s_{E\oplus D})_* \circ \sw_{E,D}^*(\hypo_E \cup \hypo_D)\\
&= i_* \circ (s_E)_*(\hypo_E \cup \hypo_D)&&\text{by \eqref{eq:s_E_D_sw}}\\
&= i_* \circ (s_E)_*(\hypo_E \cup s_E^*(\omega))&&\text{by \eqref{eq:omega_hypo}}\\
&=i_* \big((s_E)_*(\hypo_E) \cup \omega\big)&&\text{by \dref{lemm:proj_formula}{lemm:proj_formula:1}}\\
&= i_* \big(f_{d,E}(\pi(\Uni_2)) \cup \omega\big) && \text{by \eqref{eq:def_hypo}}\\
&=  f_{d,E}(\pi(\Uni_2)) \cup i_*(\omega) && \text{by \dref{lemm:proj_formula}{lemm:proj_formula:2}}\\
&= f_{d,E}(\pi(\Uni_2)) \cdot \pi(D) -  f_{d,E}(\pi(\Uni_2)) \cdot \pi(\Uni_2)&&\text{by \eqref{eq:def_omega}}\\
&= f_{d,E}(\pi(\Uni_2)) \cdot \pi(D) + f_{d+1,E}(\pi(\Uni_2)) - p_{d+1}(E) && \text{by \eqref{eq:f_k_E_rec}}\\
&= f_{d,E}(\pi(\Uni_2)) \cdot \pi(D) + f_{d+1,E}(\pi(\Uni_2))&&\text{as $\rank E \leq 2d+1$}\\
&= f_{d+1,D\oplus E}(\pi(\Uni_2)) && \text{by \eqref{eq:f_E_D}},
\end{align*}
from which the statement follows.
\end{proof}

\begin{proposition}
\label{prop:mult_o}
Let $E,F$ be vector bundles of constant even ranks over $X \in \Sm_S$. Then the elements defined in \rref{eq:def_hypo} satisfy
\[
\hypo_{E \oplus F} = \sw_{E,F}^*(\hypo_E \cup \hypo_F).
\]
\end{proposition}
\begin{proof}
This follows from the splitting principle \rref{th:splitting} and \rref{lemm:mult_hypo_2}.
\end{proof}

\begin{para}
\label{p:def_hypo}
Let $E \to X$ be a vector bundle with $X \in \Sm_S$. Recall that if $E$ has constant even rank, we have defined in \eqref{eq:def_hypo} a class $\hypo_E \in A^{0,0}(X;E \oplus E)$. If $E$ has constant odd rank, we define
\begin{equation}
\label{eq:def_hypo_odd}
\hypo_E = (\Sigma^{4,2})^{-1}\circ (\sw_{E,1}^*)^{-1}(\hypo_{E \oplus 1}) \in  A^{0,0}(X;E \oplus E).
\end{equation}
This permits in an obvious way to define a class $\hypo_E \in A^{0,0}(X;E \oplus E)$ when the rank of $E$ is not necessarily constant.
\end{para}

\begin{lemma}
\label{lemm:E_F_1}
Let $E,F$ be vector bundles over $X \in \Sm_S$, with $F$ of constant odd rank. Then
\[
\sw_{E,F\oplus 1}^*(\hypo_E \cup \hypo_{F\oplus 1}) = \sw_{E \oplus F,1}^* \circ \Sigma^{4,2} \circ \sw_{E,F}^*(\hypo_E \cup \hypo_F).
\]
\end{lemma}
\begin{proof}
Consider the commutative square in $\SH(S)$
\begin{equation}
\label{eq:squ_sigma_tau}
\begin{gathered}
\xymatrix{
\Th_X(E \oplus F \oplus E \oplus F \oplus 1 \oplus 1) \ar[rr]^{\tau_{E,F}} && \Th_X(E \oplus E \oplus F \oplus F \oplus 1 \oplus 1)
 \\ 
\Th_X(E \oplus F \oplus 1 \oplus E \oplus F \oplus 1) \ar[u]^{\sw_{E\oplus F,1}} 
 \ar[rr]^{\sw_{E,F\oplus 1}} && \Th_X(E \oplus E \oplus F \oplus 1 \oplus F \ar[u]^{\sigma_{E,F}}  \oplus 1)
}
\end{gathered}
\end{equation}
where $\tau_{E,F}$, resp.\ $\sigma_{E,F}$, is given by
\[
(e,f,e',f',x,x') \mapsto (e,e',f,f',x,x'), \text{ resp.\ } (e,e',f,x,f',x') \mapsto (e,e',f,f',x,x').
\]
Note that we have commutative squares in $\SH(S)$
\begin{equation}
\label{eq:tau}
\begin{gathered}
\xymatrix{
\Th_X(E \oplus F \oplus E \oplus F \oplus 1 \oplus 1) \ar[rr]^-{\tau_{E,F}} \ar@{=}[d] && \Th_X(E \oplus E \oplus F \oplus F \oplus 1 \oplus 1) \ar@{=}[d] \\ 
\Sigma^{4,2}\Th_X(E \oplus F \oplus E \oplus F) \ar[rr]^-{\Sigma^{4,2}\sw_{E,F}} && \Sigma^{4,2}\Th_X(E \oplus E \oplus F \oplus F)
}
\end{gathered}
\end{equation}
\begin{equation}
\label{eq:sigma}
\begin{gathered}
\resizebox{13.3cm}{!}{
\xymatrix{
\Th_X(E \oplus E \oplus F \oplus 1 \oplus F \oplus 1)\ar[rr]^-{\sigma_{E,F}} \ar[d]_{\Delta_X} && \Th_X( E \oplus E \oplus F \oplus F \oplus 1 \oplus 1)\ar[d]^{\Delta_X} \\ 
\Th_X(E \oplus E) \wedge \Th_X(F \oplus 1 \oplus F \oplus 1) \ar[rr]^-{\id \wedge \sw_{F,1}} && \Th_X(E \oplus E) \wedge \Th_X(F \oplus F \oplus 1 \oplus 1)
}
}
\end{gathered}
\end{equation}
Now we compute in $A^{0,0}(X;E \oplus F \oplus 1 \oplus E \oplus F \oplus 1)$
\begin{align*}
\sw_{E,F \oplus 1}^*(\hypo_E \cup \hypo_{F \oplus 1})
&= \sw_{E,F \oplus 1}^*(\hypo_E \cup \sw_{F,1}^*(\Sigma^{4,2}\hypo_F)) && \text{by \eqref{eq:def_hypo_odd}}\\
&= \sw_{E,F \oplus 1}^* \circ \sigma_{E,F}^*(\hypo_E \cup \Sigma^{4,2}\hypo_F) && \text{by \eqref{eq:sigma}}\\
&= \sw_{E\oplus F,1}^* \circ \tau_{E,F}^* (\hypo_E \cup \Sigma^{4,2}\hypo_F)&& \text{by \rref{eq:squ_sigma_tau}}\\
&= \sw_{E\oplus F,1}^* \circ \tau_{E,F}^* \circ \Sigma^{4,2}(\hypo_E \cup \hypo_F)&&\text{by \eqref{eq:Sigma_cup}}\\
&= \sw_{E\oplus F,1}^* \circ \Sigma^{4,2} \circ \sw_{E,F}^*(\hypo_E \cup \hypo_F) && \text{by \eqref{eq:tau}}.\qedhere
\end{align*}
\end{proof}

\begin{theorem}
\label{th:Thom_orientation}
Let $A \in \SH(S)$ be an $\eta$-periodic ring spectrum. Then every weak hyperbolic orientation on $A$ (see \rref{def:hyp_thom}) is induced by a unique hyperbolic orientation (see \rref{def:hyp_orientation}).
\end{theorem}
\begin{proof}
Let us assume that $A$ carries a weak hyperbolic orientation. We will show below that the assignment $E \mapsto \hypo_E$ defined in \rref{p:def_hypo} is a hyperbolic orientation of $A$. Once this is done, it will follow from \rref{lemm:o_t} that this hyperbolic orientation induces the original weak hyperbolic orientation. The uniqueness of a hyperbolic orientation having this property is a consequence of the splitting principle \rref{th:splitting}, and the proof of the theorem will thus be complete.

We now proceed with the proof that $E \mapsto \hypo_E$ is a hyperbolic orientation. The axioms \dref{def:hyp_orientation}{def:hyp_orientation:central}, \dref{def:hyp_orientation}{def:hyp_orientation:funct} and \dref{def:hyp_orientation}{def:hyp_orientation:isom} are easily verified. The axiom \dref{def:hyp_orientation}{def:hyp_orientation:norm} follows from the computation
\[
\hypo_1 \overset{\eqref{eq:def_hypo_odd}}{=}  (\Sigma^{4,2})^{-1}\circ (\sw_{1,1}^*)^{-1}(\hypo_{1^{\oplus 2}}) \overset{\rref{lemm:o_t}}{=} (\Sigma^{4,2})^{-1}\circ (\sw_{1,1}^*)^{-1}(\hypt_{1^{\oplus 2}}) \overset{\dref{def:hyp_thom}{def:hyp_thom:norm}}{=} \Sigma^{4,2}1.
\]

It remains to show that the axiom \dref{def:hyp_orientation}{def:hyp_orientation:mult} holds. So we consider vector bundles $E,F$ over $X \in \Sm_S$, and prove that
\begin{equation}
\label{eq:hypo_prod}
\hypo_{E\oplus F} = \sw_{E,F}^*(\hypo_E \cup \hypo_F) \in A^{0,0}(X;E\oplus F \oplus E \oplus F).
\end{equation}
We may assume that $E,F$ have constant respective ranks $r_E,r_F$, and we distinguish four cases according to the parities of $r_E,r_F$.\\

\emph{Case 1: $r_E$ is even and $r_F$ is even.} The equality \eqref{eq:hypo_prod} has been proved in \rref{prop:mult_o}.\\

\emph{Case 2: $r_E$ is even and $r_F$ is odd.} We have
\begin{align*}
\hypo_{E \oplus F}&=(\Sigma^{4,2})^{-1} \circ (\sw_{E \oplus F,1}^*)^{-1}(\hypo_{E \oplus F \oplus 1}) && \text{by \eqref{eq:def_hypo_odd}}\\
&=(\Sigma^{4,2})^{-1} \circ (\sw_{E \oplus F,1}^*)^{-1} \circ \sw_{E,F\oplus 1}^*(\hypo_E\cup  \hypo_{F \oplus 1})  && \text{by \rref{prop:mult_o}}\\ 
&= \sw_{E,F}^*(\hypo_E \cup \hypo_F) && \text{by \rref{lemm:E_F_1}},
\end{align*}
proving \eqref{eq:hypo_prod} in this case.\\

\emph{Case 3: $r_E$ is odd and $r_F$ is even.} Consider the commutative square in $\SH(S)$
\[ \xymatrix{
\Th_X(E \oplus F \oplus E \oplus F) \ar[rr]^-{\sw_{E,F}} \ar[d]_{\alpha} && \Th_X(E\oplus E \oplus F \oplus F) \ar[d]^\rho \\ 
\Th_X(F \oplus E \oplus E \oplus F) \ar[rr]^-{\sw_{F,E}} && \Th_X(F \oplus F \oplus E \oplus E)
}\]
where $\alpha$, resp.\ $\rho$, is given by
\[
(e,f,e',f') \mapsto (f,e,f',e'), \text{ resp.\ } (e,e',f,f') \mapsto (f,f',e,e').
\]
Then, using Case 2 treated above (with the roles of $E$ and $F$ exchanged), we have
\[
\sw_{E,F}^*(\hypo_E \cup \hypo_F) = \sw_{E,F}^* \circ \rho^*(\hypo_F \cup \hypo_E) = \alpha^* \circ \sw_{F,E}^*(\hypo_F \cup \hypo_E) =  \alpha^*(\hypo_{F \oplus E}),
\]
which equals $\hypo_{E\oplus F}$ since $\alpha$ is induced by the morphism $E \oplus F \to F \oplus E$ given by $(e,f) \mapsto (f,e)$, and the axiom \dref{def:hyp_orientation}{def:hyp_orientation:isom} holds. This proves \eqref{eq:hypo_prod} in this case.\\

\emph{Case 4: $r_E$ is odd and $r_F$ is odd.} Using Case 3 for the pair $(E,F \oplus 1)$ and Case 2 for the pair $(E \oplus F,1)$, we have
\[
\sw_{E \oplus F,1}^* \circ \Sigma^{4,2} \circ \sw_{E,F}^*(\hypo_E \cup \hypo_F) \overset{\rref{lemm:E_F_1}}{=} \sw_{E,F\oplus 1}^*(\hypo_E \cup \hypo_{F\oplus 1}) = \hypo_{E \oplus F \oplus 1} = \sw_{E \oplus F,1}^*(\hypo_{E \oplus F} \cup \hypo_1).
\]
Applying $(\sw_{E \oplus F,1}^*)^{-1}$, we deduce that (recall that the axiom \dref{def:hyp_orientation}{def:hyp_orientation:norm} has already been established above)
\[
\Sigma^{4,2} \circ \sw_{E,F}^*(\hypo_E \cup \hypo_F) = \hypo_{E \oplus F} \cup \hypo_1 \overset{\dref{def:hyp_orientation}{def:hyp_orientation:norm}}{=} \hypo_{E \oplus F} \cup \Sigma^{4,2}1 \overset{\eqref{eq:Sigma_cup}}{=} \Sigma^{4,2} \hypo_{E \oplus F},
\]
from which \eqref{eq:hypo_prod} follows upon applying $(\Sigma^{4,2})^{-1}$.
\end{proof}

\section{Cohomology of higher Grassmannians}

\subsection{Algebraic interlude}
In this section, we introduce the notation required to express the cohomology of higher Grassmannians in \rref{prop:A_Gr_n_s} below, and prove certain purely algebraic results that will be required.

\begin{para}
\label{def:s}
Let $r\in \Nn$. Let $B$ be a ring (unital and associative as usual), and consider the polynomial ring in $r$ variables $B[p_1,\dots,p_r]$ (the variables $p_i$ are central). We define 
\[
s_j=s_j(p_1,\dots,p_r) \in B[p_1,\dots,p_r] \quad \text{ for $j\in \Zz$}
\]
by the formula
\[
(1+tp_1+\dots+t^rp_r)^{-1} = \sum_{j\in\Zz} s_j t^j \in (B[p_1,\dots,p_r])[[t]].
\]
Thus $s_0=1$ and $s_j=0$ for $j<0$, and we have the inductive relation
\begin{equation}
\label{eq:s_inductive}
s_j = - p_1 s_{j-1} - \dots - p_r s_{j-r} \quad \text{ for $j\in \Zz \smallsetminus \{0\}$}.
\end{equation}
Note that the elements $s_j$ are central in the ring $B[p_1,\dots,p_r]$.
\end{para}

\begin{para}
\label{p:def_J}
Let $d \in \Zz$. We denote by $J_{d,r}$, or simply $J_d$, the (two-sided) ideal of $B[p_1,\dots,p_r]$ generated by $s_j$ for $j \geq d$. In particular $J_d=B[p_1,\dots,p_r]$ when $d \leq 0$.
\end{para}

\begin{para}
When $d \geq 0$, the relation \eqref{eq:s_inductive} shows that the ideal $J_n$ is generated by the elements $s_d,\dots,s_{d+r-1}$.
\end{para}

\begin{lemma}
\label{lemm:J_r_r-1}
Assume that $r \geq 1$. For any $d\in \Zz$, we have an exact sequence of $B$-bimodules
\[
B[p_1,\dots,p_r]/J_{d-1,r} \xrightarrow{p_r \cdot} B[p_1,\dots,p_r]/J_{d,r} \xrightarrow{p_r \mapsto 0} B[p_1,\dots,p_{r-1}]/J_{d,r-1} \to 0,
\]
where the first morphism is induced by the multiplication by $p_r$, and the second by the morphism of $B[p_1,\dots,p_{r-1}]$-algebras $g \colon B[p_1,\dots,p_r] \to B[p_1,\dots,p_{r-1}]$ given by $p_r \mapsto 0$.
\end{lemma}
\begin{proof}
We may assume that $d > 0$. It follows by induction from \eqref{eq:s_inductive} that $g(s_j) = s_j$ for any $j$, and thus $g(J_{d,r}) = J_{d,r-1}$. Thus the exact sequence
\[
B[p_1,\dots,p_r] \xrightarrow{p_r\cdot} B[p_1,\dots,p_r] \xrightarrow{g} B[p_1,\dots,p_{r-1}] \to 0
\]
descends to an exact sequence
\[
B[p_1,\dots,p_r] \xrightarrow{p_r \cdot} B[p_1,\dots,p_r]/J_{d,r} \to B[p_1,\dots,p_{r-1}]/J_{d,r-1} \to 0.
\]
As $d-1+r >0$, it follows from \eqref{eq:s_inductive} that
\[
p_r s_{d-1} = -p_{r-1} s_d -\dots - p_1 s_{d-2+r} - s_{d-1+r} \in J_{d,r},
\]
hence $p_rJ_{d-1,r} \subset J_{d,r}$, concluding the proof.
\end{proof}

\begin{para}
We will denote by $L$ be the (two-sided) ideal generated by $p_1,\dots,p_r$ in $B[p_1,\dots,p_r]$.
\end{para}

\begin{lemma}
\label{lemm:J_L}
For any $k \in \Nn$, we have $J_{rk} \subset L^k$ in $B[p_1,\dots,p_r]$.
\end{lemma}
\begin{proof}
We can assume that $k \geq 1$. The relation \eqref{eq:s_inductive} implies that $J_{rk} \subset LJ_{rk-r}$, from which the statement follows by induction on $k$.
\end{proof}

\begin{lemma}
\label{lemm:L_J_J}
For each $d \in \Zz$ there exists $s \in \Nn$ such that $L^sJ_d \subset J_{d+1}$ in $B[p_1,\dots,p_r]$.
\end{lemma}
\begin{proof}
We proceed by induction on $d$, the statement being clear when $d <0$. Let us fix an integer $d \geq 0$. For $i\in \Nn\smallsetminus \{0\}$, let us denote by $L_i \subset B[p_1,\dots,p_r]$ the ideal generated by the variables $p_j$ for $j \in \{i,\dots,r\}$. We will show that for each $i \in \Nn\smallsetminus \{0\}$ there exists an integer $t$ (depending on $i$) such that
\begin{equation}
\label{eq:Lt_Li_Jd}
L^t L_i J_d \subset J_{d+1}.
\end{equation}
Since $L=L_1$, this will complete the proof of the lemma.

We prove \eqref{eq:Lt_Li_Jd} by descending induction on $i$, the case $i > r$ being clear (as $L_i=0$). Assume that $i \leq r$. From the relation \eqref{eq:s_inductive}, we deduce that
\[
p_i s_d = (-s_{d+i} -\dots -p_{i-1} s_{d+1}) + (-p_{i+1}s_{d-1}-\dots - p_r s_{d+i-r}) \in J_{d+1} + L_{i+1} J_{d+i-r}.
\]
Since the ideal $L_i J_d$ is generated by $p_is_d$ and $J_{d+1} + L_{i+1}J_d$, it follows that
\begin{equation}
\label{eq:Li_Jd}
L_i J_d \subset J_{d+1} + L_{i+1} J_{d+i-r}.
\end{equation}
By induction on $d$ (that is, using the lemma where $d$ is replaced by $d-1,\dots,d+i-r$), we find $u \in \Nn$ such that $L^u J_{d+i-r} \subset J_d$. Together with \eqref{eq:Li_Jd}, this implies that
\begin{equation}
\label{eq:Lu_Li_Jd}
L^u L_i J_d \subset J_{d+1} +  L_{i+1} J_d
\end{equation}
By induction on $i$ (that is, using \eqref{eq:Lt_Li_Jd} where $i$ is replaced by $i+1$), we find $v \in \Nn$ such that $L^v  L_{i+1} J_d \subset J_{d+1}$. Combining with \eqref{eq:Lu_Li_Jd}, we see that \eqref{eq:Lt_Li_Jd}
holds true if we take $t=u+v$.
\end{proof}

\begin{lemma}
\label{lemm:L_J}
For each $n \in \Zz$, there exists $k \in \Nn$ such that $L^k \subset J_n$ in $B[p_1,\dots,p_r]$.
\end{lemma}
\begin{proof}
We proceed by induction on $n$, the case $n \leq 0$ being clear. Let $n \geq 1$. By induction, we find an integer $t$ such that $L^t \subset J_{n-1}$. By \rref{lemm:L_J_J} we find an integer $s$ such that $L^s J_{n-1} \subset J_n$. The statement follows by taking $k = s+t$.
\end{proof}

\begin{proposition}
\label{prop:lim_A_J}
Let $M$ be an abelian group, and $m \in M$. Assume that the ring $B$ is $M$-graded. Letting each $p_i$ have degree $mi$, we have (see \rref{p:homogeneous_power_series})
\[
\lim_d B[p_1,\dots,p_r]/J_d= B[[p_1,\dots,p_r]]_h,
\]
where the limit is computed in the category of $M$-graded rings.
\end{proposition}
\begin{proof}
Using \eqref{eq:s_inductive}, we see that the element $s_i$ is homogeneous of degree $mi$. Therefore the ideals $J_d$ are homogeneous. Since by definition
\[
\lim_d B[p_1,\dots,p_r]/L^d = B[[p_1,\dots,p_r]]_h,
\]
the statement follows from \rref{lemm:J_L} and \rref{lemm:L_J}.
\end{proof}

In the next statement, we denote by $J_d' \subset B[p_1',\dots,p_r']$ denotes the image of $J_d$ under $p_i \mapsto p_i'$. Also, when $I \subset B[p_1,\dots,p_r]$ and $I' \subset B[p_1',\dots,p_r']$ are two-sided ideals, we denote by $I+I' \subset B[p_1,\dots,p_r,p'_1,\dots,p'_r]$ the two-sided ideal generated by $I \cup I'$.

\begin{proposition}
\label{prop:J_J'}
Let $M$ be an abelian group, and $m \in M$. Assume that the ring $B$ is $M$-graded. Letting each $p_i$ and $p_i'$ have degree $mi$, we have (see \rref{p:homogeneous_power_series})
\[
\lim_d B[p_1,\dots,p_r,p_1',\dots,p_r']/(J_d+J_d') = B[[p_1,\dots,p_r,p_1',\dots,p_r']]_h,
\]
where the limit is computed in the category of $M$-graded rings.
\end{proposition}
\begin{proof}

Again, it follows from \eqref{eq:s_inductive} that the ideals $J_d+J_d'$ are homogeneous. Let $L' \subset B[p_1',\dots,p_r']$ be the ideal generated by $p_1',\dots,p_r'$. Observe that by \rref{lemm:J_L}
\[
J_{rk} + J'_{rk} \subset L^k +L'^k \subset (L+L')^k.
\]
On the other hand by \rref{lemm:L_J}, we find for each $n \in \Nn$ an integer $k$ such that
\[
(L+L')^{2k} \subset L^k +L'^k \subset J_n + J'_n,
\]
and we conclude using the fact that, by definition,
\[
\lim_d B[p_1,\dots,p_r,p_1',\dots,p_r']/(L+L')^d = B[[p_1,\dots,p_r,p_1',\dots,p_r']]_h.\qedhere
\]
\end{proof}

\subsection{Higher grassmannians}
We recall from \rref{p:g_h_s} that $\Gr(n,s)$ denotes the grassmannian of $n$-planes in $s$-space over $S$. It is equipped with a universal subbundle $\Uni_n \subset 1^{\oplus s}$ of rank $n$, and a quotient bundle $\Quo_n = 1^{\oplus s}/\Uni_n$ of rank $s-n$.

\begin{proposition}
\label{prop:A_Gr_n_s_U_0}
Let $A \in \SH(S)$ be an $\eta$-periodic ring spectrum with a weak hyperbolic orientation (see \rref{def:hyp_thom}). Let $n \in \Nn$ be odd and $s \in \Nn$ be even. Then
\[
\Aa(\Gr(n,s);\Uni_n^\vee)=0.
\]
\end{proposition}
\begin{proof}
We proceed by induction on $n$. When $n=1$, this follows from \dref{prop:PE}{prop:PE:even}, in view of \rref{eq:Gr_1_P}. Assume that $n\geq 3$. Let us denote by $Y$ the $S$-scheme classifying the vector bundle inclusions $P \subset Q \subset 1^{\oplus s}$ with $P,Q$ of respective ranks $n-2,n$. We have natural morphisms
\[
\Gr(n-2,s) \xleftarrow{p} Y \xrightarrow{q} \Gr(n,s),
\]
where $p$, resp.\ $q$, maps a flag $P \subset Q \subset 1^{\oplus s}$ to $P\subset 1^{\oplus s}$, resp.\ to $Q\subset 1^{\oplus s}$. We have a natural inclusion $p^*\Uni_{n-2} \subset q^*\Uni_n$ of vector bundles over $Y$. Let us denote by $\Ec = q^*\Uni_n / p^*\Uni_{n-2}$ the quotient. The morphism $p \colon Y \to \Gr(n-2,s)$ is the Grassmann bundle $\Gr(2,\Quo_{n-2})$, with universal rank two subbundle $\Ec \subset p^*\Quo_{n-2}$. The morphism $q \colon Y \to \Gr(n,s)$ is the Grassmann bundle $\Gr(n-2,\Uni_n)$, which may be identified with $\Gr(2,\Uni_n^\vee)$.

Since $\rank \Quo_{n-2}=s-n+2$ is odd, applying \dref{cor:splitting_twisted}{cor:splitting_twisted:2} to the Grassmann bundle $p$ yields an isomorphism
\begin{equation}
\label{eq:isom_Euler}
(-) \cup e(\Ec^\vee) \colon \Aa(Y;p^*\Uni_{n-2}^\vee) \xrightarrow{\sim} \Aa(Y;p^*\Uni_{n-2}^\vee \oplus \Ec^\vee) \simeq \Aa(Y;q^*\Uni_n^\vee),
\end{equation}
where the last isomorphism is induced by the exact sequence (see \rref{p:Thom_ses})
\[
0 \to \Ec^\vee \to q^*\Uni_n^\vee \to p^*\Uni_{n-2}^\vee \to 0.
\]
As $\Aa(\Gr(n-2,s);\Uni_{n-2}^\vee)=0$ by induction, it follows from \dref{cor:splitting_twisted}{cor:splitting_twisted:1} (applied to the bundle $p$) that $\Aa(Y;p^*\Uni_{n-2}^\vee)=0$. Using \eqref{eq:isom_Euler} we deduce that $\Aa(Y;q^*\Uni_n^\vee)=0$. Applying \dref{cor:splitting_twisted}{cor:splitting_twisted:1} to the Grassmann bundle $q$ then shows that $\Aa(\Gr(n,s);\Uni_n^\vee)=0$, completing the proof.
\end{proof}

\begin{proposition}
\label{prop:A_Gr_n_s}
Let $A \in \SH(S)$ be an $\eta$-periodic hyperbolically oriented ring spectrum. Let $d \geq r \in \Nn$, and $s\in \{2d,2d+1\}$ and $n \in \{2r,2r+1\}$. If $n$ is odd, assume that $s$ is also odd. Then mapping $p_j$ to the $j$-th Pontryagin class $p_j(\Uni_n^\vee)$ yields an isomorphism of $\Aa(S)$-algebras (see \rref{p:algebra})
\[
\Aa(S)[p_1,\dots,p_r]/J_{d-r+1} \xrightarrow{\sim} \Aa(\Gr(n,s)),
\]
where $J_{d-r+1}$ is the ideal defined in \rref{p:def_J}.
\end{proposition}
\begin{proof}
By \rref{rem:pontryagin_trivial} and the Whitney sum formula \rref{prop:Pontryagin_sum}, we have for any $k \in \Nn\smallsetminus\{0\}$
\[
p_k(\Quo_n^\vee) = - p_1(\Uni_n^\vee) p_{k-1}(\Quo_n^\vee) -\dots - p_{k-1}(\Uni_n^\vee) p_1(\Quo_n^\vee) - p_k(\Uni_n^\vee) \in A^{8k,4k}(\Gr(n,s)).
\]
Recall that $p_j(\Uni_n^\vee)$ vanishes when $2j>\rank \Uni_n^\vee=n$, and therefore when $j >r$. We deduce by induction on $k$ from \rref{eq:s_inductive} that for any $k \in \Zz$ (the case $k\leq 0$ being clear)
\begin{equation}
\label{eq:p_k_s_k}
p_k(\Quo_n^\vee) = s_k(p_1(\Uni_n^\vee),\dots,p_r(\Uni_n^\vee)) \in A^{8k,4k}(\Gr(n,s)).
\end{equation}
Now the vector bundle $\Quo_n^\vee$ has rank $s-n \leq 2(d-r)+1$, so that $p_k(\Quo_n^\vee)$ vanishes when $k \geq d-r+1$. In view of \eqref{eq:p_k_s_k}, it follows that the ideal $J_{d-r+1}$ is mapped to zero in $\Aa(\Gr(n,s))$, so that the morphism of the statement is well-defined.

To prove that it is an isomorphism, we proceed by induction on $n+s$. The cases $n=0$ or $s=0$ are clear, so we assume that $n\geq 1$ and $s \geq 1$. Let us consider the closed immersions $g = g_{1^{\oplus s-1}}\colon \Gr(n,s-1) \to \Gr(n,s)$ and $h=h_{1^{\oplus s-1}} \colon \Gr(n-1,s-1) \to \Gr(n,s)$ described in \rref{p:g_h}.\\

\emph{Case $n$ odd and $s$ odd:} By \rref{p:Gr_n-1_n_n}, we have a long exact sequence
\[
\cdots \to \Aa(\Gr(n,s-1);\Uni_n^\vee) \xrightarrow{g_*} \Aa(\Gr(n,s)) \xrightarrow{h^*} \Aa(\Gr(n-1,s-1)) \to \cdots
\]
Since the term on the left vanishes by \rref{prop:A_Gr_n_s_U_0}, it follows that $h^*$ is bijective. Since $h^*$ is a morphism of $\Aa(S)$-algebras mapping $p_j(\Uni_n^\vee)$ to $p_j(\Uni_n^\vee)$, we deduce the statement by induction.\\

\emph{Case $n$ even and $s$ even:} By \rref{p:Gr_n-1_n_n}, we have a long exact sequence
\[
\cdots \to \Aa(\Gr(n,s-1);\Uni_n^\vee) \xrightarrow{g_*} \Aa(\Gr(n,s)) \xrightarrow{h^*} \Aa(\Gr(n-1,s-1)) \to \cdots
\]
Note that $h^*$ is surjective, since by induction the $\Aa(S)$-algebra $\Aa(\Gr(n-1,s-1))$ is generated by the Pontryagin classes $p_j(\Uni_n^\vee)$, which lie in the image of $h^*$. Thus $g_*$ is injective. By \rref{p:Gr_n-1_n_n}, letting $g'=g_{1^{\oplus s-2}}$, we have an infinite long exact sequence
\[
\Aa(\Gr(n,s-2);\Uni_n^\vee \oplus \Uni_n^\vee) \xrightarrow{g'_*} \Aa(\Gr(n,s-1);\Uni_n^\vee) \to \Aa(\Gr(n-1,s-2);\Uni_n^\vee)
\]
By \rref{prop:A_Gr_n_s_U_0}, the term on the right vanishes, hence $g'_*$ above is an isomorphism. By \eqref{lemm:hyp_isom}, we obtain an exact sequence
\[
0 \to \Aa(\Gr(n,s-2)) \xrightarrow{u} \Aa(\Gr(n,s)) \xrightarrow{h^*} \Aa(\Gr(n-1,s-1)) \to 0,
\]
where $u$ is given by $x \mapsto (g\circ g')_*(x \cup \hypo_{\Uni_n^\vee})$. By \dref{lemm:e_pi_push_pull}{lemm:e_pi_push_pull:1}, the composite
\[
\Aa(\Gr(n,s)) \xrightarrow{(g\circ g')^*} \Aa(\Gr(n,s-2)) \xrightarrow{u} \Aa(\Gr(n,s))
\]
is multiplication by the element $\pi(\Uni_n^\vee)$, which coincides with $p_r(\Uni_n^\vee)$ by \rref{lemm:Pontryagin_pi}. In particular $h^*p_r(\Uni_n^\vee)= h^* \circ u(1)=0$. We thus have a commutative diagram
\[ \xymatrix{
& B_r/J_{d-r,r} \ar[r]^{p_r\cdot} \ar[d] & B_r/J_{d-r+1,r} \ar[r]^{p_r \mapsto 0} \ar[d] & B_{r-1}/J_{d-r+1,r} \ar[r] \ar[d] & 0 \\ 
0 \ar[r]& \Aa(\Gr(n,s-2)) \ar[r]^-u &  \Aa(\Gr(n,s)) \ar[r]^-{h^*} & \Aa(\Gr(n-1,s-1)) \ar[r] &0
}\]
Here for $m\in \{r,r-1\}$ we have written $B_m = \Aa(S)[p_1,\dots,p_m]$, and the upper row is the exact sequence of \rref{lemm:J_r_r-1}. We have seen that the lower row is also exact. The left and right vertical morphism are isomorphisms by induction, hence so is the middle one by a diagram chase, proving the statement in this case.\\

\emph{Case $n$ even and $s$ odd:} By \rref{p:Gr_n_n_n-1}, we have a long exact sequence
\[
\dots\to \Aa(\Gr(n-1,s-1);\Quo_{n-1}) \xrightarrow{h_*} \Aa(\Gr(n,s)) \xrightarrow{g^*} \Aa(\Gr(n,s-1))\to\cdots
\]
Set $G= \Gr(n-1,s-1)$. Then 
\[
\Aa(G;\Quo_{n-1}) \overset{\eqref{lemm:hyp_isom}}{\simeq} \Aa(G;\Quo_{n-1} \oplus \Uni_{n-1} \oplus \Uni_{n-1}) \overset{\rref{p:Thom_ses}}\simeq \Aa(G;1^{\oplus s-1} \oplus \Uni_{n-1}),
\]
which vanishes by \rref{p:Thom_dual_vb} and \rref{prop:A_Gr_n_s_U_0}. Therefore the morphism $g^*$ in the above exact sequence is an isomorphism. Since $g^*$ is a morphism of $\Aa(S)$-algebras mapping $p_j(\Uni_n^\vee)$ to $p_j(\Uni_n^\vee)$, we deduce the statement by induction, which concludes the proof of the proposition. Let us nonetheless make one additional observation. Recall from \rref{p:Gr_n-1_n_n} that we have an infinite long exact sequence
\[
\Aa(\Gr(n,s-1);\Uni_n^\vee \oplus \Uni_n^\vee) \xrightarrow{g_*} \Aa(\Gr(n,s);\Uni_n^\vee) \xrightarrow{h^*} \Aa(\Gr(n-1,s-1);\Uni_{n-1}^\vee \oplus 1)
\]
The term on the right vanishes by \rref{prop:A_Gr_n_s_U_0}, hence $g_*$ above is an isomorphism. By \eqref{lemm:hyp_isom}, the composite
\begin{align*}
\Aa(\Gr(n,s)) &\xrightarrow{g^*} \Aa(\Gr(n,s-1)) \xrightarrow{\cup \hypo_{\Uni_n^\vee}}\\
& \Aa(\Gr(n,s-1);\Uni_n^\vee \oplus \Uni_n^\vee) \xrightarrow{g_*} \Aa(\Gr(n,s);\Uni_n^\vee)
\end{align*}
is thus bijective, and it coincides with the (left or right) cup product with the Euler class $e(\Uni_n^\vee)$ by \dref{lemm:e_pi_push_pull}{lemm:e_pi_push_pull:1}, which proves \rref{prop:A_Gr_n_s_U_Euler} below.
\end{proof}

We record the following statement, obtained in the course of the proof of \rref{prop:A_Gr_n_s}:
\begin{proposition}
\label{prop:A_Gr_n_s_U_Euler}
Let $A \in \SH(S)$ be an $\eta$-periodic hyperbolically oriented ring spectrum. Let $n \in \Nn$ be even and $s \in \Nn$ be odd. Then the (left or right) $\Aa(\Gr(n,s))$-module $\Aa(\Gr(n,s);\Uni_n^\vee)$ is freely generated by the Euler class $e(\Uni_n^\vee)$.
\end{proposition}

\begin{remark}
The parity assumption in \rref{prop:A_Gr_n_s} is necessary, as the case $n=1,s=2$ shows.
\end{remark}

\subsection{Cohomology of \texorpdfstring{$\BGL$}{BGL}}
\label{sect:A_BGL}

\begin{para}
\label{p:Milnor_seq}
(See e.g.\ \cite[Lemma~2.1.3]{PMR}.) Consider a sequence of pointed motivic spaces $E_t \to E_{t+1}$ for $t \in \Nn \smallsetminus \{0\}$, and denote by $E \in \Spcp(S)$ its homotopy colimit. Then for any $A \in \SH(S)$ and $p,q\in\Zz$, we have the \emph{Milnor exact sequence}
\[
0 \to {\lim_t}^1 A^{p-1,q}(E_t) \to A^{p,q}(E) \to \lim_t A^{p,q}(E_t) \to 0.
\]
By cofinality and compatibility of homotopy colimits with the smash product, we also have an exact sequence
\[
0 \to {\lim_t}^1 A^{p-1,q}(E_t \wedge E_t) \to A^{p,q}(E \wedge E) \to \lim_t A^{p,q}(E_t \wedge E_t) \to 0.
\]
\end{para}

\begin{para}
Recall that the \'etale classifying space $\BGL_n$ is obtained as the (homotopy) colimit in $\Spc(S)$ of the grassmannians $\Gr(n,nt)$ over $t \in \Nn \smallsetminus \{0\}$ (see e.g.\ \cite[(5.1.4)]{eta} with $p=n$). Here the transition morphisms
\begin{equation}
\label{eq:Gr_transition}
\Gr(n,nt) \to \Gr(n,n(t+1))
\end{equation}
are the closed immersions induced by the inclusions
\[
1^{\oplus nt} = (1^{\oplus t})^{\oplus n} \subset (1^{\oplus t+1})^{\oplus n} = 1^{\oplus n(t+1)}
\]
where the inclusion $1^{\oplus t} \subset 1^{\oplus t+1}$ is given by the vanishing of the last coordinate. Since the vector bundle $\Uni_n^\vee$ pulls back to $\Uni_n^\vee$ along \eqref{eq:Gr_transition}, there are induced maps of pointed motivic spaces $\Th_{\Gr(n,nt)}(\Uni_n^\vee) \to \Th_{\Gr(n,n(t+1))}(\Uni_n^\vee)$, and we define
\[
\Th_{\BGL_n}(\Uni_n^\vee) = \colim_t \Th_{\Gr(n,nt)}(\Uni_n^\vee) \in \Spcp(S),
\]
and if $A \in \SH(S)$ is a ring spectrum, we set
\[
\Aa(\BGL_n;\Uni_n^\vee) = A^{*+2n,*+n}(\Th_{\BGL_n}(\Uni_n^\vee)).
\]
\end{para}

The next result generalises a computation of Levine \cite[Theorem~4.1]{Levine-motivic_Euler} (see also \cite[(6.3.7)]{eta}):
\begin{theorem}
\label{th:A_BGL}
Let $A \in \SH(S)$ be an $\eta$-periodic hyperbolically oriented ring spectrum. Let $r \in \Nn$, and $n \in \{2r,2r+1\}$. Letting $p_i$ have degree $(8i,4i)$, there exists an isomorphism of $\Zz^2$-graded $\Aa(S)$-algebras (see \rref{p:homogeneous_power_series}), 
\[
\Aa(\BGL_n) \simeq \Aa(S)[[p_1,\dots,p_r]]_h.
\]
In addition the (left of right) $\Aa(\BGL_n)$-module $\Aa(\BGL_n;\Uni_n^\vee)$ vanishes if $n$ is odd, and is free of rank $1$ if $n$ is even.
\end{theorem}
\begin{proof}
By \rref{prop:A_Gr_n_s}, the morphism $\Aa(\Gr(n,n(t+2))) \to \Aa(\Gr(n,nt))$ is surjective when $t$ is odd, so that the system $\Aa(\Gr(n,nt))$ for $t \in \Nn \smallsetminus \{0\}$ satisfies the Mittag-Leffler condition, and thus $\lim^1_t \Aa(\Gr(n,nt))=0$. In addition, by \rref{prop:A_Gr_n_s} and \rref{prop:lim_A_J} (and a cofinality argument) we have an isomorphism
\[
\lim_t \Aa(\Gr(n,nt)) \simeq \Aa(S)[[p_1,\dots,p_r]]_h.
\]
The first statement thus follows from the Milnor sequence \rref{p:Milnor_seq}.

Next, assume that $n$ is odd. Then by \rref{prop:A_Gr_n_s_U_0} we have $\Aa(\Gr(n,nt);\Uni_n^\vee)=0$ when $t$ is even, hence 
\[
{\lim_t}^1 \Aa(\Gr(n,nt);\Uni_n^\vee)=0 \quad \text{and} \quad \lim_t \Aa(\Gr(n,nt);\Uni_n^\vee)=0,
\]
so that $\Aa(\BGL_n;\Uni_n^\vee)=0$ by the Milnor sequence \rref{p:Milnor_seq}.

Finally let us assume that $n$ is even and prove the remaining statement. The case $n=0$ being clear, we assume that $n \geq 2$. A cofinality argument shows that $\BGL_n$, resp.\ $\Th_{\BGL_n}(\Uni_n^\vee)$, is the colimit of $\Gr(n,nt+1)$, resp.\ $\Th_{\Gr(n,nt+1)}(\Uni_n^\vee)$, the transition morphisms being induced by the inclusions
\[
1^{\oplus nt +1} = 1^{\oplus t+1} \oplus (1^{\oplus t})^{\oplus n-1} \subset 1^{\oplus t+2} \oplus (1^{\oplus t+1})^{\oplus n-1} = 1^{\oplus n(t+1) +1},
\]
where the inclusions $1^{\oplus t+1} \subset 1^{\oplus t+2}$ and $1^{\oplus t} \subset 1^{\oplus t+1}$ are given by the vanishing of the last coordinates. By \rref{prop:A_Gr_n_s_U_Euler} we have isomorphisms of $\Aa(\Gr(n,nt+1))$-modules $\Aa(\Gr(n,nt+1)) \xrightarrow{\sim} \Aa(\Gr(n,nt+1);\Uni_n^\vee)$ which are compatible with the transition morphisms as $t$ varies. This yields isomorphisms of $\Aa(\BGL_n)$-modules
\[
0={\lim_t}^1 \Aa(\Gr(n,nt+1)) \xrightarrow{\sim} {\lim_t}^1 \Aa(\Gr(n,nt+1);\Uni_n^\vee) 
\]
\[
\lim_t \Aa(\Gr(n,nt+1)) \xrightarrow{\sim} \lim_t \Aa(\Gr(n,nt+1);\Uni_n^\vee).
\]
We conclude using again the Milnor sequence \rref{p:Milnor_seq}.
\end{proof}

\begin{corollary}
Let $A \in \SH(S)$ be an $\eta$-periodic hyperbolically oriented ring spectrum. Then there exist isomorphisms of $\Zz^2$-graded $\Aa(S)$-algebras, where $\deg p_i=(8i,4i)$,
\[
\Aa(\BSL_{2r+1}) \simeq \Aa(\BSL^c_{2r+1}) \simeq \Aa(S)[[p_1,\dots,p_r]]_h.
\]
\end{corollary}
\begin{proof}
This follows by combining \rref{th:A_BGL} with \cite[(6.2.1), (6.3.3)]{eta}.
\end{proof}

\begin{remark}
Ananyevskiy computed in \cite[Theorem~10]{Ananyevskiy-SL_PB} the ring  $\Aa(\BSL_n)$ when $A \in \SH(S)$ is an $\eta$-periodic $\SL$-oriented commutative ring spectrum and $2$ is invertible in $S$ (see \cite[(6.2.2)]{eta} for more details). When $n$ is even, the result involves Euler classes of $\SL$-oriented vector bundles (with values in untwisted cohomology), whose existence seems to require that $A$ be $\SL$-oriented.
\end{remark}

We will need the following complement to \rref{th:A_BGL} in the next section:
\begin{proposition}
\label{prop:A_BGL_BGL}
Let $A \in \SH(S)$ be an $\eta$-periodic hyperbolically oriented ring spectrum. Let $r \in \Nn$, and $n \in \{2r,2r+1\}$. Then, letting $p_i,p_i'$ have degree $(8i,4i)$, there exists an isomorphism of $\Zz^2$-graded $\Aa(S)$-algebras (see \rref{p:homogeneous_power_series})
\[
\Aa((\BGL_n)_+ \wedge (\BGL_n)_+) \simeq \Aa(S)[[p_1,\dots,p_r,p_1',\dots,p_r']]_h.
\]
\end{proposition}
\begin{proof}
Let us write $R_t = \Aa(\Gr(n,nt) \times_S \Gr(n,nt))$, for $t\in \Nn \smallsetminus \{0\}$. Then the morphisms $\Gr(n,nt) \to \Gr(n,n(t+1))$ described at the beginning of \S\ref{sect:A_BGL} induce transition morphisms $R_{t+1} \to R_t$. Applying \rref{prop:A_Gr_n_s} over the base $\Gr(n,nt)$, and then over $S$, we obtain for $t$ odd 
\[
R_t \simeq \Aa(S)[p_1,\dots,p_r,p_1',\dots,p_r']/(J_{d-r+1} + J'_{d-r+1}),
\]
where $d$ is the integer such that $nt \in \{2d,2d+1\}$ (we use the notation described just above \rref{prop:J_J'}). We deduce that the morphisms $R_{t+2} \to R_t$ are surjective when $t$ is odd, so that the system $R_t$ for $t \in \Nn \smallsetminus \{0\}$ satisfies the Mittag-Leffler condition, and thus $\lim^1_t R_t =0$. By \rref{prop:J_J'} (and a cofinality argument) we have
\[
\lim_t R_t \simeq \Aa(S)[[p_1,\dots,p_r,p_1',\dots,p_r']]_h,
\]
and the statement follows from the Milnor sequence \rref{p:Milnor_seq}.
\end{proof}

\begin{para}
\label{p:BGL_BGL}
For later reference, let us note the following facts, established in the course of the proofs of \rref{th:A_BGL} and \rref{prop:A_BGL_BGL}:
\begin{enumerate}[(i)]
\item
\label{p:BGL_BGL:lim1}
 We have
\[
{\lim_t}^1\Aa(\Gr(n,nt)) =0 \quad\text{and} \quad {\lim_t}^1 \Aa(\Gr(n,nt)\times_S \Gr(n,nt))=0,
\]

\item 
\label{p:BGL_BGL:lim}
The following natural morphisms are bijective:
\[
\Aa(\BGL_n) \xrightarrow{\sim} \lim_t \Aa(\Gr(n,nt)),
\]
\[
\Aa((\BGL_n)_+ \wedge (\BGL_n)_+) \xrightarrow{\sim} \lim_t \Aa(\Gr(n,nt)\times_S \Gr(n,nt)).
\]

\item 
\label{p:BGL_BGL:p}
Under the identification of \rref{th:A_BGL}, resp.\ \rref{prop:A_BGL_BGL}, the elements $p_i$, resp.\ $p_i$ and $p_i'$, are mapped, under the morphism induced by \rref{p:BGL_BGL:lim}, to $p_i(\Uni_n) \in \Aa(\Gr(n,nt))$, resp.\ $p_i(\Uni_n \times_S \Gr(n,nt))$ and $p_i(\Gr(n,nt) \times_S \Uni_n)$.
\end{enumerate}
\end{para}

\section{The universal theory}
\label{sect:MH}
\numberwithin{theorem}{subsection}
\numberwithin{lemma}{subsection}
\numberwithin{proposition}{subsection}
\numberwithin{corollary}{subsection}
\numberwithin{example}{subsection}
\numberwithin{notation}{subsection}
\numberwithin{definition}{subsection}
\numberwithin{remark}{subsection}

Until the end of the paper, we will assume that the scheme $S$ is regular separated (and noetherian of finite dimension). 

\subsection{Hyperbolic preorientations}
In this section, we introduce a notion destined to facilitate some proofs in the next sections.

\begin{definition}
\label{def:hyp_preorientation}
Let $A \in \SH(S)$ and $n \in \Nn$. A \emph{hyperbolic $n$-preorientation} of $A$ is the datum of a class $\phyp_E \in A^{0,0}(X;E \oplus E)$ for each rank $n$ vector bundle $E \to X$ with $X \in \Sm_S$, subject to the following conditions:
\begin{enumerate}[(i)]
\item \label{def:hyp_preorientation:funct}
if $f \colon Y \to X$ is a morphism in $\Sm_S$ and $E \to X$ a vector bundle, then $f^*\phyp_E = \phyp_{f^*E}$,

\item \label{def:hyp_preorientation:isom}
if $E \xrightarrow{\sim} F$ is an isomorphism of vector bundles over $X \in \Sm_S$, then the induced isomorphism $\Aa(X;F \oplus F) \xrightarrow{\sim} \Aa(X;E \oplus E)$ maps $\phyp_F$ to $\phyp_E$.
\end{enumerate}
\end{definition}

\begin{para}
A weak hyperbolic orientation (see \rref{def:hyp_thom}) is in particular a hyperbolic $2$-preorientation.
\end{para}

\begin{para}
A hyperbolic orientation (see \rref{def:hyp_orientation}) induces a hyperbolic $n$-preorientation for every $n \in \Nn$. The hyperbolic orientation is determined by the collection of its induced hyperbolic $n$-preorientations, for $n\in \Nn$.
\end{para}

\begin{para}
\label{p:push_preorientation}
Let $\psi \colon A \to B$ be a morphism in $\SH(S)$. If $E \mapsto \phyp_E$ is a hyperbolic $n$-preorientation of $A$, then $E \mapsto \psi_*(\phyp_E)$ defines a hyperbolic $n$-preorientation of $B$.
\end{para}

\begin{para}[Jouanolou's trick]
\label{p:Jouanolou_affine}
If $X \in \Sm_S$, then $X$ is a regular separated noetherian scheme, hence by \cite[II, 2.2.7.1]{sga6} it admits an ample family of line bundles. Therefore by \cite[Proposition~4.4]{Weibel-KH} there exists an affine bundle $\tilde{X} \to X$ in $\Sm_S$ with $\tilde{X}$ affine.
\end{para}

\begin{proposition}
\label{prop:preorientations}
A hyperbolic $n$-preorientation of a ring spectrum $A \in \SH(S)$ is determined by the elements $\phyp_{\Uni_n} \in A^{0,0}(\Gr(n,nt);\Uni_n \oplus \Uni_n)$ for $t \in \Nn \smallsetminus \{0\}$.
\end{proposition}
\begin{proof}
Let $E \to X$ be a rank $n$ vector bundle with $X \in \Sm_S$. Pick an affine bundle $p \colon \widetilde{X} \to X$ such that the scheme $\widetilde{X}$ is affine, using Jouanolou's trick \rref{p:Jouanolou_affine}. Then we may find an inclusion $E \subset 1^{\oplus nt}$ for some $t \in \Nn \smallsetminus \{0\}$, giving a morphism $f\colon \widetilde{X} \to \Gr(n,nt)$ and an isomorphism $\alpha \colon p^*E\xrightarrow{\sim} f^*\Uni_n $. Consider the induced isomorphism $\Th(\alpha \oplus \alpha) \colon  \Th_{\widetilde{X}}(p^*E \oplus p^*E)\xrightarrow{\sim} \Th_{\widetilde{X}}(f^*\Uni_n \oplus f^*\Uni_n) $. Then
\[
\phyp_E  = (p^{-1})^* \circ (\Th(\alpha \oplus \alpha))^* \circ f^*(\phyp_{\Uni_n}).\qedhere
\]
\end{proof}

\subsection{The spaces \texorpdfstring{$\MH_n$}{MHn}}

\begin{para}
\label{p:def_MHn}
For $n,s \in \Nn$, let us consider the pointed motivic spaces
\[
\T(n,s)= \Th_{\Gr(n,s)}(\Uni_n \oplus \Uni_n) \in \Spcp(S).
\]
For $t\in \Nn \smallsetminus \{0\}$, the closed immersion $\Gr(n,nt) \to \Gr(n,n(t+1)$ described in \eqref{eq:Gr_transition}, along which the vector bundle $\Uni_n$ pulls back to $\Uni_n$, yields a map in $\Spcp(S)$
\begin{equation}
\label{eq:T_transition}
\T(n,nt) \to \T(n,n(t+1)).
\end{equation}
Taking the colimit over $t\in \Nn \smallsetminus \{0\}$, we obtain pointed motivic spaces, for $n \in \Nn$
\[
\MH_n = \colim_t \T(n,nt),
\]
together with canonical maps in $\Spcp(S)$, for $n\in \Nn$ and $t\in \Nn \smallsetminus \{0\}$
\begin{equation}
\label{eq:T_MH}
\lambda_{n,t} \colon \T(n,nt) \to \MH_n.
\end{equation}
\end{para}

\begin{lemma}
\label{lemm:MHn_lim}
Let $A \in \SH(S)$ be an $\eta$-periodic hyperbolically oriented ring spectrum. Then the morphisms induced by \eqref{eq:T_MH}
\[
\Aa(\MH_n) \to \lim_t \Aa(\T(n,nt))
\]
\[
\Aa(\MH_n \wedge \MH_n) \to \lim_t \Aa(\T(n,nt) \wedge \T(n,nt)) 
\]
are bijective.
\end{lemma}
\begin{proof}
The isomorphisms for $p,q \in \Zz$ and $r \in \{1,2\}$ (see \rref{lemm:hyp_isom})
\[
(-) \cup (\hypo_{\Uni_n})^{\cup r} \colon A^{p-4rn,q-2rn}(\Gr(n,nt)^{\times_S r}) \xrightarrow{\sim} A^{p,q}(\T(n,nt)^{\wedge r}),
\]
are compatible with the transition maps when $t$ varies. Thus by \dref{p:BGL_BGL}{p:BGL_BGL:lim1} we have $\lim_t^1\Aa(\T(n,nt)^{\wedge r})=0$, whence the statements by the Milnor sequence \rref{p:Milnor_seq}.
\end{proof}

\begin{lemma}
\label{lemm:map_MHn}
Let $n\in \Nn$. There exists a unique way to associate to each rank $n$ vector bundle $E \to X$ with $X \in \Sm_S$ a map $\theta_E \colon \Th_X(E\oplus E) \to \MH_n$ in $\Hop(S)$ such that:
\begin{enumerate}[(i)]
\item 
\label{lemm:map_MHn:1}
if $f \colon Y \to X$ is a morphism in $\Sm_S$, then the following diagram commutes:
\[ \xymatrix{
\Th_Y(f^*E \oplus f^*E)\ar[rr]^-{\theta_{f^*E}} \ar[d]_f && \MH_n \\ 
\Th_X(E\oplus E) \ar[rru]_-{\theta_E} && 
}\]

\item 
\label{lemm:map_MHn:2}
if $\alpha \colon E \xrightarrow{\sim} F$ is an isomorphism of vector bundles over $X \in \Sm_S$, then the following diagram commutes
\[ \xymatrix{
\Th_X(E \oplus E)\ar[rr]^-{\theta_E} \ar[d]_{\Th_X(\alpha \oplus \alpha)} && \MH_n \\ 
\Th_X(F \oplus F) \ar[rru]_-{\theta_F} && 
}\]

\item 
\label{lemm:map_MHn:3}
for every $t \in \Nn \smallsetminus \{0\}$, the map $\lambda_{n,t} \colon \T(n,nt) \to \MH_n$ of \eqref{eq:T_MH} is $\theta_{\Uni_n}$.
\end{enumerate}
\end{lemma}
\begin{proof}
Let $E \to X$ be a rank $n$ vector bundle, with $X\in \Sm_S$. Assume first that $X$ is affine. Then there exists an integer $t \in \Nn \smallsetminus \{0\}$ and a vector bundle inclusion $i \colon E \to 1^{\oplus nt}$. This gives a morphism $X \to \Gr(n,nt)$ along which $\Uni_n$ pulls back to $E$, and thus a map $g_i \colon \Th_X(E \oplus E) \to \T(n,nt)$ in $\Spcp(S)$. We claim that the composite in $\Hop(S)$
\[
\theta_E = \theta_E(i) \colon \Th_X(E \oplus E) \xrightarrow{g_i}  \T(n,nt) \xrightarrow{\lambda_{n,t}} \MH_n
\]
does not depend on $t$ or $i$. Indeed, let $j\colon E \to 1^{\oplus nt'}$ be an inclusion, with $t'\in \Nn \smallsetminus \{0\}$. While proving that $\theta_E(i) = \theta_E(j)$, we may assume that $t \geq t'$. Composing $j$ with the inclusion $1^{\oplus nt'} = (1^{\oplus t'})^{\oplus n} \subset (1^{\oplus t})^{\oplus n} = 1^{\oplus nt}$ induced by the inclusion $1^{\oplus t'} \subset 1^{\oplus t}$ given by the vanishing of the last $t-t'$ coordinates, we are reduced to assuming that $t'=t$. Let $k \colon \T(n,nt) \to \T(n,2nt)$ be the map given by the vanishing of the last $nt$ coordinates. By \rref{lemm:inclusion_indep} below, we have
\begin{equation}
\label{eq:k_g}
k \circ g_i = k \circ g_j \colon \Th_X(E \oplus E) \to \Th(n,2nt) \quad \text{in $\Hop(S)$}.
\end{equation}
Now the morphism
\[
(1^{\oplus t})^{\oplus n} \oplus (1^{\oplus t})^{\oplus n} \xrightarrow{\sim} (1^{\oplus t} \oplus 1^{\oplus t})^{\oplus n}, \quad ((x_1,\dots,x_n),(y_1,\dots,y_n)) \mapsto ((x_1,y_1),\dots,(x_n,y_n))
\]
induces a map $c \colon \Th(n,2nt) \to \Th(n,2nt)$, such that $c \circ k \colon \Th(n,nt) \to \Th(n,2nt)$ is the composite of the transition maps \eqref{eq:T_transition}. In particular the map $\lambda_{n,t} \colon \Th(n,t) \to \MH_n$ factors through $k$, hence \eqref{eq:k_g} implies that $\theta_E(i) = \theta_E(j)$ in $\Hop(S)$, proving the claim. It is then easy to verify the conditions \eqref{lemm:map_MHn:1} and \eqref{lemm:map_MHn:2}, under the additional assumption that $X$ and $Y$ are affine.

When $X$ is not necessarily affine, by Jouanolou's trick \rref{p:Jouanolou_affine} we find an affine bundle $p \colon \tilde{X} \to X$ with $\tilde{X} \in \Sm_S$ affine. We claim that the composite in $\Hop(S)$
\[
\theta_E \colon \Th_X(E \oplus E) \xrightarrow{p^{-1}} \Th_{\tilde{X}}(p^*E \oplus p^*E) \xrightarrow{\theta_{p^*E}} \MH_n
\]
does not depend on $\tilde{X}$ and $p$. Indeed, let $p' \colon \tilde{X'} \to X$ be an affine bundle. Let $Y= \tilde{X} \times_S \tilde{X'}$, and denote by $q \colon Y \to \tilde{X}$ and $q'\colon Y \to \tilde{X}'$ the induced morphisms. Then we have an isomorphism $\beta \colon q^*p^*E \xrightarrow{\sim} q'^*p'^*E$ yielding a commutative diagram in $\Hop(S)$
\[
\xymatrix{
\Th_X(E \oplus E) & \Th_{\tilde{X}}(p^*E \oplus p^*E)\ar[l]_-p & \Th_Y(q^*p^*E \oplus q^*p^*E) \ar[l]_-{q}\ar[d]_{\Th_Y(\beta \oplus \beta)} \ar[rr]^-{\theta_{q^*p^*E}}&&\MH_n\\ 
& \Th_{\tilde{X'}}(p'^*E \oplus p'^*E)\ar[lu]^{p'} &  \Th_Y(q'^*p'^*E \oplus q'^*p'^*E)\ar[l]_{q'} \ar[rru]_-{\theta_{q'^*p'^*E}}&&
}
\]
(here we have used the validity of the condition \eqref{lemm:map_MHn:2} over the scheme $Y$, which is affine as $S$ is separated). Since $\theta_{p^*E} \circ q = \theta_{q^*p^*E}$ and $\theta_{p'^*E} \circ q' = \theta_{q'^*p'^*E}$ by the condition \eqref{lemm:map_MHn:2} in the affine case, we deduce that $\theta_{p^*E} \circ p^{-1} = \theta_{p'^*E} \circ p'^{-1}$, proving the claim. 
It is now easy to verify the conditions \eqref{lemm:map_MHn:1}, \eqref{lemm:map_MHn:2}, \eqref{lemm:map_MHn:3}, as well as the uniqueness part of the statement.
\end{proof}

\begin{lemma}
\label{lemm:inclusion_indep}
Let $V \to S$ be a vector bundle, and $r\in \Nn$. Let $f \colon X \to S$ in $\Sm_S$, and $E \to X$ a rank $n$ vector bundle. If $k\colon E \subset f^*V \oplus f^*V$ is a vector bundle inclusion, corresponding to a morphism $X \to \Gr(n,V \oplus V)$ in $\Sm_S$, let us denote by $\varphi_k \colon \Th_X(E^{\oplus r}) \to \Th_{\Gr(n,V \oplus V)}(\Uni_n^{\oplus r})$ the induced map in $\Spcp(S)$. Then the image of $\varphi_{(i,0)}$ in $\Hop(S)$ does not depend on the inclusion $i\colon E \subset f^*V$.
\end{lemma}
\begin{proof}
Let $i,j \colon E \subset f^*V$ be vector bundle inclusions. Let $p\colon \Ab^1 \times_S X \to X$ be the projection. Consider the vector bundle inclusion $(p^*i,tp^*j) \colon p^*E \subset p^*f^*V \oplus p^*f^*V$, where $t$ is the tautological section over $\Ab^1$. Then the induced morphism $\Th_{\Ab^1 \times_S X}(p^*E^{\oplus r}) \to \Th_{\Gr(n,V \oplus V)}(\Uni_n^{\oplus r})$ restricts to $\varphi_{(i,0)}$ at $0\in \Ab^1$ and to $\varphi_{(i,j)}$ at $1 \in \Ab^1$. This yields $\varphi_{(i,0)} = \varphi_{(i,j)}$ in $\Hop(S)$. Similarly $\varphi_{(0,j)} = \varphi_{(i,j)}$ in $\Hop(S)$. Therefore, in $\Hop(S)$ we have
\[
\varphi_{(i,0)} = \varphi_{(i,j)} = \varphi_{(0,j)} = \varphi_{(j,j)} = \varphi_{(j,0)}.\qedhere
\]
\end{proof}

\begin{para}
\label{p:MH_n_pre}
Let $A \in \SH(S)$ be a ring spectrum equipped with a hyperbolic $n$-preorientation. Then the elements $\hypt_{\Uni_n} \in A^{4n,2n}(\T(n,nt))$ are compatible with the transition maps \eqref{eq:T_transition} as $t$ varies in $\Nn \smallsetminus \{0\}$, hence by \rref{lemm:MHn_lim} are the images of a unique element of $\hypt \in A^{4n,2n}(\MH_n)$.
\end{para}

\begin{para}
\label{p:MH_pre_n}
Conversely, let $A \in \SH(S)$ be a ring spectrum and $n \in \Nn$. Assume given an element $\hypt \in A^{4n,2n}(\MH_n)$. Applying \rref{lemm:map_MHn}, we associate to each rank $n$ vector bundle $E \to X$ with $X \in \Sm_S$ an element $\hypt_E=\theta_E^*(\hypt) \in A^{0,0}(X;E \oplus E)$. 
\end{para}

\begin{proposition}
\label{prop:MH_n_pre}
Let $A \in \SH(S)$ be a ring spectrum. The procedures described in \rref{p:MH_n_pre} and \rref{p:MH_pre_n} yield mutually inverse bijections between the set of hyperbolic $n$-preorientations of $A$ and the elements of $A^{4n,2n}(\MH_n)$.
\end{proposition}
\begin{proof}
This follows from \rref{prop:preorientations} and \dref{lemm:map_MHn}{lemm:map_MHn:3}.
\end{proof}

\begin{corollary}
\label{prop:MH_2_weak}
Let $A \in \SH(S)$ be a commutative ring spectrum. The procedures described in \rref{p:MH_n_pre} and \rref{p:MH_pre_n} yield mutually inverse bijections between the set of weak hyperbolic orientations of $A$ (see \rref{def:hyp_thom}) and the elements of $A^{8,4}(\MH_2)$ whose restriction  along $\theta_{1^{\oplus 2}} \colon \Th_S(1^{\oplus 4}) \to \MH_2$ is $\sw_{1,1}^*(\Sigma^{8,4}1) \in A^{0,0}(S;1^{\oplus 4}) = A^{8,4}(\Th_S(1^{\oplus 4}))$.
\end{corollary}
\begin{proof}
This follows from \rref{p:com_whyp}.
\end{proof}

\subsection{The ring spectrum \texorpdfstring{$\MH$}{MH}}
In this section, we assemble the pointed motivic spaces $\MH_n$ defined in the previous section into a commutative ring spectrum $\MH \in \SH(S)$. We follow very closely the strategy employed in \cite[\S2.1]{PPR-MGL} to construct the spectrum $\MGL$, the main difference being that we naturally obtain a $T^{\wedge 2}$-spectrum instead of a $T$-spectrum, as was the case for the spectrum $\MSp$ constructed by Panin--Walter \cite[\S6]{PW-MSL-Msp}. Since the model categories of symmetric $T$- and $T^{\wedge 2}$-spectra have equivalent homotopy categories (with their symmetric monoidal structures) by \cite[Theorem~3.2]{PW-MSL-Msp}, we still obtain a commutative ring spectrum in $\SH(S)$.

\begin{para}
\label{p:Th_product_sw}
Let us first describe a construction which will be useful in this section. Let $E_i \to X_i$ be vector bundles with $X_i \in \Sm_S$, for $i \in \{1,2\}$. Set $P=X_1 \times_S X_2$ and let $V_i \to P$ be the pullback of $E_i \to X_i$ along the $i$-th projection $P \to X_i$, for $i \in \{1,2\}$. Assume given a morphism $f \colon P \to Z$ in $\Sm_S$, and $W \to Z$ a vector bundle together with an isomorphism $f^*W \simeq V_1 \oplus V_2$. Then there is an induced map in $\Spcp(S)$
\begin{align*}
\Th_{X_1}(E_1 \oplus E_1) &\wedge \Th_{X_2}(E_2 \oplus E_2) = \Th_P(V_1 \oplus V_1 \oplus V_2 \oplus V_2)\\
&\xrightarrow{(\sw_{V_1,V_2})^{-1}} \Th_P(V_1 \oplus V_2 \oplus V_1 \oplus V_2) \xrightarrow{f} \Th_Z(W \oplus W).
\end{align*}
\end{para}

\begin{para}
For $n \in \Nn$ and $t\in \Nn \smallsetminus \{0\}$, consider the closed immersions
\begin{equation}
\label{eq:Gr_nt_(n+1)t}
\gamma_{n,t} \colon \Gr(n,nt) \to \Gr(n+1,(n+1)t))
\end{equation}
given by mapping a subbundle $E \subset 1^{\oplus nt}$ to $E \oplus 1 \subset 1^{\oplus nt} \oplus 1^{\oplus t} = 1^{\oplus (n+1)t}$, where the inclusion $1 \subset 1^{\oplus t}$ is induced by the vanishing of the $t-1$ last coordinates. Along the morphism \eqref{eq:Gr_nt_(n+1)t}, the vector bundle $\Uni_{n+1}$ pulls back to $\Uni_n \oplus 1$. Using the procedure described in \rref{p:Th_product_sw}, we thus obtain maps in $\Spcp(S)$, for $n \in \Nn$ and $t \in \Nn\smallsetminus \{0\}$ (recall that $T=\Th_S(1)$)
\begin{equation}
\label{eq:tr_T}
\tau_{n,t} \colon \T(n,nt) \wedge T^{\wedge 2} \to \T(n+1,(n+1)t),
\end{equation}
which are compatible with the transition maps \eqref{eq:T_transition} as $t$ varies. Taking the colimit over $t\in \Nn \smallsetminus \{0\}$, we obtain maps in $\Spcp(S)$, for $n\in \Nn$
\begin{equation}
\label{eq:tr_MH}
\MH_n \wedge T^{\wedge 2} \to \MH_{n+1}.
\end{equation}
\end{para}

\begin{para}
\label{p:MH_maps}
Let $n \in \Nn$ and $t\in \Nn \smallsetminus\{0\}$. The natural action of the symmetric group $\Sy_n$ (where $\Sy_0=1$) on $(1^{\oplus t})^{\oplus n}$ induces an action on $\Gr(n,nt)$, for which the vector bundle $\Uni_n$ is $\Sy_n$-equivariant. This yields a $\Sy_n$-action on $\T(n,nt)=\Th_{\Gr(n,nt)}(\Uni_n \oplus \Uni_n)$. The transition maps $\Gr(n,nt) \to \Gr(n,n(t+1))$ of \rref{eq:Gr_transition} and $\T(n,nt) \to \T(n,n(t+1))$ of \rref{eq:T_transition} are $\Sy_n$-equivariant. For $m \in \Nn$, consider the $(\Sy_n \times \Sy_m)$-equivariant morphism in $\Sm_S$
\begin{equation}
\label{eq:mu_n_m_s}
\mu_{n,m,t} \colon \Gr(n,nt) \times_S \Gr(m,mt) \to \Gr(n+m,(n+m)t)
\end{equation}
given by
\[
(U \subset 1^{\oplus nt},V \subset 1^{\oplus mt}) \mapsto U\oplus V \subset 1^{\oplus nt} \oplus 1^{\oplus mt} = 1^{\oplus (n+m)t},
\]
under which the vector bundle $\Uni_{n+m}$ pulls back to $p_1^*\Uni_n \oplus p_2^*\Uni_m$, where $p_1 \colon \Gr(n,nt) \times_S \Gr(m,mt) \to \Gr(n,nt)$ and $p_2 \colon \Gr(n,nt) \times_S \Gr(m,mt) \to \Gr(m,mt)$ are the two projections. By the procedure described in \rref{p:Th_product_sw}, we obtain a $(\Sy_n \times \Sy_m)$-equivariant map in $\Spcp(S)$
\begin{equation}
\label{eq:T_n_m_s}
\T(n,nt) \wedge \T(m,mt) \to \T(n+m,(n+m)t).
\end{equation}
These morphisms are compatible with the transition maps \eqref{eq:Gr_transition} and \eqref{eq:T_transition} as $t$ varies in $\Nn \smallsetminus \{0\}$. Taking the colimit over $t$, we obtain a $\Sy_n$-action on each $\MH_n$, and a $(\Sy_n \times \Sy_m)$-equivariant map in $\Spcp(S)$, for $n,m \in \Nn$
\[
\mu_{n,m} \colon \MH_n \wedge \MH_m \to \MH_{n+m}.
\]
We also have in $\Spcp(S)$ a canonical isomorphism and a map (see \eqref{eq:T_MH})
\begin{equation}
\label{eq:def_e1}
e_0 \colon S_+ \xrightarrow{\sim} \MH_0 \quad \text{and} \quad e_1 \colon T^{\wedge 2} = \Th(1,1) \xrightarrow{\lambda_{1,1}} \MH_1.
\end{equation}
\end{para}

\begin{definition}
\label{def:MH}
A straightforward verification shows that the data described in \rref{p:MH_maps} define a commutative $T^{\wedge 2}$-monoid in $\Spcp(S)$ in the sense of \cite[Definition~3.3]{PW-MSL-Msp}, and thus by \cite[Theorem~3.4]{PW-MSL-Msp} a commutative monoid in the category of symmetric $T^{\wedge 2}$-spectra, that we denote by $\MH$. Using the natural equivalence between the homotopy categories of symmetric $T^{\wedge 2}$-spectra and of $T$-spectra (see \cite[Theorem~3.2]{PW-MSL-Msp} and \cite[Theorem~4.31]{Jardine-Spt}), we may view $\MH$ as a commutative ring spectrum in $\SH(S)$.
\end{definition}

\begin{para}
By definition, the bonding maps of the $T^{\wedge 2}$-spectrum $\MH$ are the composites
\begin{equation}
\label{eq:tr_MH:2}
\MH_n \wedge T^{\wedge 2} \xrightarrow{\id \wedge e_1} \MH_n \wedge \MH_1 \xrightarrow{\mu_{n,1}} \MH_{n+1},
\end{equation}
and it follows form the construction that they coincide with the maps \eqref{eq:tr_MH}.
\end{para}

\begin{para}
Let $n\in \Nn$. The functor mapping a $T^{\wedge 2}$-spectrum $E$ to its level $n$ component $E_n$ admits a left adjoint, which maps a pointed motivic space $Y$ to the spectrum $\Su_{T^{\wedge 2}} Y(-n)$ given by $(*,\dots,*,Y,T^{\wedge 2} \wedge Y, \dots)$, whose image in $\SH(S)$ is naturally isomorphic to $\Sigma^{-4n,-2n}\Su Y$. Thus the identity of $\MH_n$ yields, by adjunction and application of the functor $\Sigma^{4n,2n}$, a canonical map in $\SH(S)$
\begin{equation}
\label{eq:MHn_MH}
\rho_n \colon \Su \MH_n \to \Sigma^{4n,2n} \MH.
\end{equation}
Note that, under the above adjunction for $n=1$, the unit $1_{\MH} \colon \Un_S \to \MH$ of the ring spectrum $\MH$ corresponds to the map $e_1$ of \eqref{eq:def_e1}, hence factors in $\SH(S)$ as
\begin{equation}
\label{eq:unit_MH}
1_{\MH} \colon \Un_S = \Sigma^{-4,-2} \Su T^{\wedge 2} \xrightarrow{\Sigma^{-4,-2} \Su e_1} \Sigma^{-4,-2} \Su \MH_1 \xrightarrow{\Sigma^{-4,-2}\rho_1} \MH.
\end{equation}
Moreover, it follows from the construction of the product $\mu \colon \MH \wedge \MH \to \MH$ that the following diagram commutes in $\SH(S)$ (using the identification \rref{p:sigma_exchange} below)
\begin{equation}
\label{diag:mu}
\begin{gathered}
\xymatrix{
\Su \MH_n \wedge \Su \MH_m\ar[rrr]^-{\Su \mu_{n,m}} \ar[d]_{\rho_n \wedge \rho_m} &&& \Su \MH_{n+m} \ar[d]^{\rho_{n+m}} \\ 
\Sigma^{4n,2n}\MH \wedge \Sigma^{4m,2m}\MH \ar[rrr]^-{ \Sigma^{4(n+m),2(n+m)}\mu} &&& \Sigma^{4(n+m),2(n+m)}\MH
}
\end{gathered}
\end{equation}
\end{para}

\begin{para}
\label{p:sigma_exchange}
Let $A,B \in \SH(S)$ and $a,b \in \Nn$. The isomorphism $\tau \colon T^{\wedge a} \wedge B \to B \wedge T^{\wedge a}$ exchanging the factors $B$ and $T^{\wedge a}$ induces an identification
\[
\Sigma^{2a,a}A \wedge \Sigma^{2b,b}B = A \wedge T^{\wedge a} \wedge B \wedge T^{\wedge b} \simeq A \wedge B \wedge T^{\wedge a} \wedge T^{\wedge b} = \Sigma^{2(a+b),a+b}(A \wedge B),
\]
where the middle isomorphism is $\id_A \wedge \tau \wedge \id_{T^{\wedge b}}$.
\end{para}

\begin{proposition}
\label{prop:hyp_MH}
One may define a hyperbolic orientation on the ring spectrum $\MH$ using the elements $\theta_E$ of \rref{lemm:map_MHn}, by setting
\[
\hypo_E = \rho_n \circ \Su \theta_E \in \MH^{4n,2n}(\Th_X(E \oplus E))= \MH^{0,0}(X;E \oplus E)
\]
for every vector bundle $E\to X$ of rank $n \in \Nn$, with $X \in \Sm_S$ (and extending this definition to arbitrary vector bundles in an obvious way).
\end{proposition}
\begin{proof}
We verify the axioms of \rref{def:hyp_orientation}. Centrality \dref{def:hyp_orientation}{def:hyp_orientation:central} is automatic because the ring spectrum $\MH$ is commutative (see \rref{p:com_hyp}). Compatibility with pullbacks \dref{def:hyp_orientation}{def:hyp_orientation:funct} and isomorphisms \dref{def:hyp_orientation}{def:hyp_orientation:isom} follow from \dref{lemm:map_MHn}{lemm:map_MHn:1} and \dref{lemm:map_MHn}{lemm:map_MHn:2}. Applying \dref{lemm:map_MHn}{lemm:map_MHn:3} with $n=t=1$ shows that $\theta_1$ is the map $\lambda_{1,1} =e_1$ of \eqref{eq:def_e1}. Thus it follows from the factorisation \eqref{eq:unit_MH} that $\hypo_1 = \Sigma^{4,2} 1_{\MH}$, where $1_{\MH} \colon \Un_S \to \MH$ is the unit of the ring spectrum $\MH$, proving the normalisation axiom \dref{def:hyp_orientation}{def:hyp_orientation:norm}.

Finally let us prove the multiplicativity axiom \dref{def:hyp_orientation}{def:hyp_orientation:mult}. Let $E,F$ be vector bundles over $X \in \Sm_S$, of respective ranks $n,m \in \Nn$. In view of the diagram \eqref{diag:mu}, it will suffice to prove that the map $\theta_{E\oplus F}$ in $\Hop(S)$ factors as (see \rref{p:sw} and \rref{p:cup_product})
\begin{align*}
\Th_X(E \oplus F \oplus E \oplus F) \xrightarrow{\sw_{E,F}}& \Th_X(E \oplus E \oplus F \oplus F) \xrightarrow{\Delta_X} \Th_X(E\oplus E) \wedge \Th_X(F \oplus F) \\
&\xrightarrow{\theta_E \wedge \theta_F} \MH_n \wedge \MH_m \xrightarrow{\mu_{n,m}} \MH_{n+m}.
\end{align*}
While doing so, we may assume that $X$ is affine by Jouanolou's trick \rref{p:Jouanolou_affine}, and thus assume given inclusions $E \subset 1^{\oplus nt}$ and $F \subset 1^{\oplus mt}$ for some $t \in \Nn \smallsetminus \{0\}$, corresponding to morphisms $e \colon X \to \Gr(n,nt)$ and $f \colon X \to \Gr(m,mt)$. Then we have a commutative diagram in $\Hop(S)$
\[
\resizebox{\linewidth}{!}{\xymatrix{
\Th_X(E \oplus E) \wedge \Th_X(F \oplus F) \ar[rd]_{\theta_E \wedge \theta_F} \ar[r]^-{e \wedge f} & \T(n,nt) \wedge \T(m,mt) \ar[rr]^-{\eqref{eq:T_n_m_s}} \ar[d]^{\lambda_{n,nt} \wedge \lambda_{m,mt}}&&\Th(n+m,(n+m)t) \ar[d]_{\lambda_{n+m,(n+m)t}}\\ 
 & \MH_n \wedge \MH_m \ar[rr]^-{\mu_{n,m}}  && \MH_{n+m}
}
}\]
Consider the morphism $h \colon X \to \Gr(n+m,(n+m)t)$ in $\Sm_S$ corresponding to the inclusion $E \oplus F \subset 1^{\oplus nt} \oplus 1^{\oplus mt}=1^{\oplus (n+m)t}$. Then $\theta_{E\oplus F}$ factors as
\[
\Th_X(E \oplus F \oplus E \oplus F) \xrightarrow{h}  \T(n+m,(n+m)t) \xrightarrow{\lambda_{n+m,(n+m)t}} \MH_{n+m}.
\]
Therefore it will suffice to prove that the composite
\begin{align*}
\Th_X(E \oplus F \oplus E \oplus F) \xrightarrow{\sw_{E,F}}& \Th_X(E \oplus E \oplus F \oplus F) \xrightarrow{\Delta_X} \Th_X(E\oplus E) \wedge \Th_X(F \oplus F) \\
&\xrightarrow{e \wedge f} \T(n,nt) \wedge \T(m,mt) \xrightarrow{\eqref{eq:T_n_m_s}} \T(n+m,(n+m)t)
\end{align*}
is the morphism induced by $h$. Set $P=\Gr(n,nt) \times_S \Gr(m,mt)$. Let $U \to P$ be the pullback of $\Uni_n$ under the first projection, and $V \to P$ the pullback of $\Uni_m$ under the second projection. The consideration of the commutative diagram in $\Hop(S)$
\[ \xymatrix{
\Th_X(E \oplus F \oplus E \oplus F) \ar[d]_{\sw_{E,F}} \ar[r]_{(e,f)} \ar@/^2.0pc/[rr]^h &  \Th_P(U \oplus V \oplus U \oplus V) \ar[d]^{\sw_{U,V}}  \ar[r]_{\mu_{n,m,t}}& \T(n+m,(n+m)t)\\ 
\Th_X(E \oplus E \oplus F \oplus F) \ar[r]_{(e,f)}\ar[d]_-{\Delta_X}& \Th_P(U \oplus U \oplus V \oplus V) & \T(n,nt) \wedge \T(m,mt) \ar@{=}[l] \ar[u]_{\eqref{eq:T_n_m_s}}\\
\Th_X(E \oplus E) \wedge \Th_X(F \oplus F)\ar@/_1.0pc/[rru]_{e\wedge f}&&
}\]
concludes the proof.
\end{proof}

\begin{para}
\label{p:MH_univ}
Let us denote by $E \mapsto \hypo_E^{\MH}$ the hyperbolic orientation of $\MH$ described in \rref{prop:hyp_MH}. Let $A \in \SH(S)$ be a commutative ring spectrum. If $\psi \colon \MH \to A$ is a morphism of ring spectra in $\SH(S)$, then $E \mapsto \psi_*(\hypo_E^{\MH})$ defines a hyperbolic orientation of $A$.
\end{para}

\begin{lemma}
\label{lemm:MH_lim}
Let $A \in \SH(S)$ be an $\eta$-periodic hyperbolically oriented ring spectrum. Then for any $p,q\in \Zz$, the morphisms
\[
A^{p,q}(\MH) \to \lim_n A^{p+4n,q+2n}(\MH_n) \; \text{ and } \; A^{p,q}(\MH \wedge \MH) \to \lim_n A^{p+8n,q+4n}(\MH_n \wedge \MH_n) 
\]
induced by \rref{eq:MHn_MH} are bijective.
\end{lemma}
\begin{proof}
Let  $r\in \{1,2\}$. By \cite[Corollaries~2.1.4, 2.1.5]{PMR} (see also \cite[Theorem~5.6]{PW-MSL-Msp}) we have for each $p,q \in \Zz$ a short exact sequence of abelian groups
\[
0 \to {\lim_n}^1 A^{p+4rn-1,q+2rn}(\MH_n^{\wedge r}) \to A^{p,q}(\MH^{\wedge r}) \to \lim_n A^{p+4rn,q+2rn}(\MH_n^{\wedge r}) \to 0.
\]
The isomorphisms, for $n \in \Nn$ and $t \in \Nn \smallsetminus \{0\}$ (see \rref{lemm:hyp_isom})
\[
(-) \cup (\hypo_{\Uni_n})^{\cup r} \colon \Aa(\Gr(n,nt)^{\times_S r}) \xrightarrow{\sim} \Aa(\T(n,nt)^{\wedge r})
\]
(which are graded of degree $(4rn,2rn)$) are compatible with the morphisms \eqref{eq:Gr_nt_(n+1)t} and \eqref{eq:tr_T}, which in view of \dref{p:BGL_BGL}{p:BGL_BGL:lim} and \rref{lemm:MHn_lim} yields a commutative diagram
\[ 
\xymatrix{
\Aa((\BGL_{n+1})_+^{\wedge r})\ar[r] \ar[d]_{\sim} & \Aa((\BGL_n)_+^{\wedge r}) \ar[d]^{\sim} \\ 
\Aa(\MH_{n+1}^{\wedge r}) \ar[r] & \Aa(\MH_n^{\wedge r})
}
\]
Under the identifications given in \rref{th:A_BGL} and \rref{prop:A_BGL_BGL},  by \dref{p:BGL_BGL}{p:BGL_BGL:p} the upper horizontal arrow is a morphism of $\Aa(S)$-algebras mapping $p_j$ to $p_j$, as well as $p_j'$ to $p_j'$ when $r=2$, hence is surjective. Since the vertical arrows are bijective, it follows that the lower horizontal arrow is surjective. This implies that the $\lim^1$-term vanishes in the above exact sequence, concluding the proof.
\end{proof}

\begin{theorem}
\label{th:MH_univ}
Let $A \in \SH(S)$ be an $\eta$-periodic commutative ring spectrum. The procedure described in \rref{p:MH_univ} yields a bijection between the set of morphisms of ring spectra $\MH \to A$ in $\SH(S)$ and the set of hyperbolic orientations of $A$.
\end{theorem}
\begin{proof}
Assume that $A$ carries a hyperbolic orientation $E \mapsto \hypo_E^A$. Then for each $n$, the family $\hypo^A_{\Uni_n} \in A^{4n,2n}(\T(n,nt))$ where $t$ runs over $\Nn \smallsetminus \{0\}$, lifts to a unique element $\sigma_n \in A^{4n,2n}(\MH_n)$ by \rref{lemm:MHn_lim}. Since the vector bundle $\Uni_{n+1}$ restricts to $\Uni_n \oplus 1$ along the morphism $\gamma_{n,t}\colon \Gr(n,nt) \to \Gr(n+1,n(t+1))$ of \eqref{eq:Gr_nt_(n+1)t}, it follows that the morphism  $\tau_{n,t}$ of \eqref{eq:tr_T} verifies
\[
\tau_{n,t}^*(\hypo^A_{\Uni_{n+1}}) \overset{\rref{p:Th_product_sw}}{=} (\sw_{\Uni_n,1}^*)^{-1} \circ \gamma_{n,t}^*(\hypo^A_{\Uni_{n+1}}) \overset{\dref{def:hyp_orientation}{def:hyp_orientation:funct}}{=} (\sw_{\Uni_n,1}^*)^{-1}(\hypo^A_{\Uni_n \oplus 1}) \overset{\dref{def:hyp_orientation}{def:hyp_orientation:mult}}{=} \hypo^A_{\Uni_n} \cup \hypo^A_1,
\]
which equals $\Sigma^{4,2}\hypo^A_{\Uni_n}$ by \dref{def:hyp_orientation}{def:hyp_orientation:norm} and \rref{eq:Sigma_cup}. Taking the limit over $t$, we deduce that $\sigma_{n+1} \in A^{4(n+1),2(n+1)}(\MH_{n+1})$ maps to $\sigma_n \in A^{4n,2n}(\MH_n)$ under the pullback along \eqref{eq:tr_MH}. Therefore by \rref{lemm:MH_lim}, we obtain a well-defined element $\varphi \in A^{0,0}(\MH)$, in other words a morphism $\varphi \colon \MH \to A$ in $\SH(S)$.  To verify that $\varphi$ is indeed a morphism of ring spectra, we investigate the commutativity of the diagrams in $\SH(S)$
\begin{equation}
\label{squ:tr_squ}
\begin{gathered}
\xymatrix{
\Un_S\ar[r]^-{1_{\MH}} \ar[rd]_{1_A} & \MH \ar[d]^{\varphi} &&\MH \wedge \MH \ar[d]_{\varphi \wedge \varphi} \ar[r]^-{\mu_{\MH}} & \MH \ar[d]^{\varphi}\\ 
 &A&& A \wedge A \ar[r]^-{\mu_A} & A
}
\end{gathered}
\end{equation}

Observe that, by construction of $\varphi$, for $n\in \Nn$ the element  $\hypo^A_{\Uni_n}$ is the composite
\begin{equation}
\label{eq:lambda_rho_varphi}
\Th(n,nt) \xrightarrow{\lambda_{n,t}} \MH_n \xrightarrow{\rho_n} \MH \xrightarrow{\varphi} A.
\end{equation}
On the other hand, by definition of the hyperbolic orientation of $\MH$ (in \rref{prop:hyp_MH}) and in view of \dref{lemm:map_MHn}{lemm:map_MHn:3}, the composite \eqref{eq:lambda_rho_varphi} is $\varphi_*(\hypo^{\MH}_{\Uni_n})$. In view of \rref{prop:preorientations} and \rref{p:push_preorientation}, it follows that
\begin{equation}
\label{eq:varphi_hypo}
\varphi_*(\hypo_E^{\MH}) = \hypo_E^A \;\,  \text{for any vector bundle $E \to X$ with $X \in \Sm_S$.}
\end{equation}
Taking $E=1$ and $X=S$, and using the normalisation axiom \dref{def:hyp_orientation}{def:hyp_orientation:norm} for $\MH$ and $A$, we deduce the commutativity of the triangle in \eqref{squ:tr_squ}.

Applying \rref{lemm:MH_lim} and \rref{lemm:MHn_lim}, and in view of the commutative diagram \eqref{diag:mu}, the commutativity of the square in \eqref{squ:tr_squ} boils down to the commutativity in $\SH(S)$ of the following square (using the identification \rref{p:sigma_exchange}), for each $n \in \Nn$ and $t \in \Nn \smallsetminus\{0\}$,
\begin{equation}
\label{eq:squ:mu}
\begin{gathered}
\xymatrix{
\Su \T(n,nt) \wedge \Su \T(n,nt) \ar[d]_{\hypo^A_{\Uni_n} \wedge \hypo^A_{\Uni_n}} \ar[rr]^-{\Su\eqref{eq:T_n_m_s}} && \Su \T(2n,2nt) \ar[d]^{\hypo^A_{\Uni_{2n}}}\\ 
\Sigma^{4n,2n}A \wedge \Sigma^{4n,2n}A \ar[rr]^{\Sigma^{8n,4n}\mu_A} && \Sigma^{8n,4n}A
}
\end{gathered}
\end{equation}
The composites in the square \eqref{eq:squ:mu} may be viewed as elements of 
\[
A^{8n,4n}(\T(n,nt) \wedge \T(n,nt)) = A^{0,0}(P;p_1^*\Uni_n \oplus p_1^*\Uni_n \oplus p_2^*\Uni_n\oplus p_2^*\Uni_n),
\]
where $p_1,p_2 \colon P=\Gr(n,nt) \times_S \Gr(n,nt) \to \Gr(n,nt)$ are the projections. Then in \eqref{eq:squ:mu}, the top horizontal map followed by the right vertical one is (see \eqref{eq:mu_n_m_s})
\[
(\sw_{p_1^*\Uni_n,p_2^*\Uni_n}^*)^{-1} \circ \mu_{n,n,t}^*(\hypo^A_{\Uni_{2n}}) \overset{\rref{p:Th_product_sw}}{=} (\sw_{p_1^*\Uni_n,p_2^*\Uni_n}^*)^{-1}(\hypo^A_{p_1^*\Uni_n \oplus p_2^*\Uni_n}) \overset{\text{\dref{def:hyp_orientation}{def:hyp_orientation:mult}}}{=} p_1^*\hypo^A_{\Uni_n} \cup p_2^*\hypo^A_{\Uni_n},
\]
which coincides with, writing $Y=\Th(n,nt)$ and using the notation of \rref{p:cup_product},
\[
Y \wedge Y \xrightarrow{\Delta_P} Y \wedge Y \wedge Y \wedge Y \xrightarrow{p_1 \wedge p_2} Y \wedge Y \xrightarrow{\hypo^A_{\Uni_n} \wedge \hypo^A_{\Uni_n}} \Sigma^{4n,2n}A \wedge \Sigma^{4n,2n}A \xrightarrow{\Sigma^{8n,4n}\mu_A} \Sigma^{8n,4n}A.
\]
Since $(p_1 \wedge p_2) \circ \Delta_P=\id_{Y \wedge Y}$, the composite just above coincides with the left vertical map followed by the lower horizontal map in the square \eqref{eq:squ:mu}. We have proved that the diagrams \eqref{squ:tr_squ} commute, so that $\varphi$ is a morphism of ring spectra.

It remains to verify that the above construction is the inverse of the one given in \rref{p:MH_univ}. This follows from \rref{eq:varphi_hypo}, and from the fact that a morphism $\psi \colon \MH \to A$ in $\SH(S)$ is determined by the composites $\T(n,nt) \xrightarrow{\lambda_{n,t}} \MH_n \xrightarrow{\rho_n} \MH \xrightarrow{\psi} A$, for $n\in \Nn$ and $t\in \Nn \smallsetminus \{0\}$, which is a consequence of \rref{lemm:MHn_lim} and \rref{lemm:MH_lim}.
\end{proof}

\end{document}